\documentclass[reqno,11pt]{amsart}

\usepackage[top=2.0cm,bottom=2.0cm,left=3cm,right=3cm]{geometry}
\usepackage{amsthm,amsmath,amssymb,dsfont}
\usepackage{mathrsfs,amsfonts,functan,extarrows,mathtools}
\usepackage[colorlinks]{hyperref}

\usepackage{marginnote}
\usepackage{xcolor}
\usepackage{stmaryrd}
\usepackage{esint}
\usepackage{graphicx}
\usepackage{bbm}
\usepackage{wasysym}
\usepackage{textcomp}
\usepackage{float}
\usepackage{wasysym}

\usepackage{tikz}
\usetikzlibrary{calc,decorations.pathreplacing}

\newtheorem{theorem}{Theorem}[section]

\newtheorem{remark}{Remark}[section]
\newtheorem{lemma}{Lemma}[section]

\numberwithin{equation}{section}
\allowdisplaybreaks

\def\d{\mathrm{d}}
\def\no{\nonumber}
\def\R{\mathbb{R}}
\def\eps{\varepsilon}

\def\u{\mathbf{u}}

\def\exp{\mathrm{exp}}
\def\M{\mathfrak{M}}

\def\B{\mathcal{B}}

\def\u{\mathfrak{u}}
\def\L{\mathcal{L}}

\def\P{\mathcal{P}}

\def\LL{[\![}
\def\RR{]\!]}

\newcounter{wronumber}

\begin{document}
\title[Knudsen layer equation]
			{Knudsen boundary layer equations with incoming boundary condition: full range of cutoff collision kernels and Mach numbers of the far field}

\author[Ning Jiang]{Ning Jiang}
\address[Ning Jiang]{\newline School of Mathematics and Statistics, Wuhan University, Wuhan, 430072, P. R. China}
\email{njiang@whu.edu.cn}

\author[Yi-Long Luo]{Yi-Long Luo}
\address[Yi-Long Luo]
{\newline School of Mathematics, South China University of Technology, Guangzhou, 510641, P. R. China}
\email{luoylmath@scut.edu.cn}

\author[Yulong Wu]{Yulong Wu}
\address[Yulong Wu]{\newline School of Mathematics and Statistics, Wuhan University, Wuhan, 430072, P. R. China}
\email{yulong\string_wu@whu.edu.cn}
\author[Tong Yang]{Tong Yang}
\address[Tong Yang]
{\newline Institute for Math \string& AI, Wuhan University, Wuhan, 430072, P. R. China}
\email{tongyang@whu.edu.cn}

%\thanks{${}^*$ Corresponding author \quad \today}

\maketitle

\begin{abstract}
   This paper establishes the existence and uniqueness of the nonlinear Knudsen layer equation with incoming boundary conditions. It is well-known that the solvability conditions of the problem vary with the Mach number of the far Maxwellian $\mathcal{M}^\infty$. We consider full ranges of cutoff collision kernels (i.e., $- 3 < \gamma \leq 1$) and all the Mach numbers of the far field in the $L^\infty_{x,v}$ framework.  Additionally, the solution exhibits exponential decay $\exp \{- c x^\frac{2}{3 - \gamma} - c |v|^2 \}$ for some $c > 0$. To address the general angular cutoff collision kernel, we introduce a $(x,v)$-mixed weight $\sigma$. The proof is essentially based on adding an artificial damping term. \\


\end{abstract}

%\vspace*{10pt}
%
%\phantomsection
%\addcontentsline{toc}{section}{\contentsname}
%
%\tableofcontents

%%%%%%%%%%%%%%%%%%%%%%%%%%%%%%%%%%%%%%（正文）%%%%%%%%%%%%%%%%%%%%%%%%%%%%%
%%%%%%%%%%%%%%%%%%%%%%%%%%%%%%%%%%%%%%%%%%%%%%%%%%%%%%%%%%%%%%%%%%%%%%%%%%%

\section{Introduction and main results}

\subsection{The description of the problem}

In the studies of the hydrodynamic limits of the scaled Boltzmann equation in the domains with boundary, a (linear or nonlinear) kinetic boundary layer equation, known as the Knudsen layer equation usually will be generated, see \cite{Aoki-2017-JSP,Guo-Huang-Wang-2021-ARMA,JLT-2021-arXiv,JLT-2021-arXiv-2,Sone-Book-2002,Sone-Book-2007} for instance. In this paper, we consider the nonlinear Knudsen layer equation over $(x,v) \in \R_+ \times \R^3$ with the incoming (sometimes called absorbing) boundary condition at $x = 0$:
\begin{equation}\label{KL-NL}
	\begin{aligned}
		\left\{
		    \begin{aligned}
		    	& v_3 \partial_x F = \mathcal{B} (F, F) + H \,, \ x > 0, v \in \R^3 \,, \\
		    	& F (0, v) |_{v_3 > 0 } =  F_b (v) \,, \\
		    	& \lim_{x \to + \infty} F (x,v) = \M (v) \,,
	    	\end{aligned}
		\right.
	\end{aligned}
\end{equation}
where $H = H(x,v) \in \mathrm{Null}^\perp (\L)$, with the linearized Boltzmann collision operator $\L$ defined in \eqref{Lf}, and $F_b (v) = 0$ if $v_3 < 0$, and $\M$ denotes the global Maxwellian in the far field, defined as
\begin{equation*}
	\begin{aligned}
		\M (v) = \frac{\rho}{(2 \pi T)^{3/2}} \exp \big[ - \tfrac{|v - \u|^2}{2 T} \big] \,.
	\end{aligned}
\end{equation*}
Here $\rho, T > 0$ and $\u = (\u_1, \u_2, \u_3) \in \R^3$ are given constants representing the macroscopic density, temperature, and velocity, respectively. By a shift of the variables $v_2,v_3$, we can assume without loss of generality that $\u_1=\u_2 = 0$. Then, the sound speed and Mach number of this equilibrium state are given by
\begin{equation}
	c=\sqrt{\frac{5}{3}T} \,,\quad \mathcal{M}^\infty=\frac{\u_3}{c}\,.
\end{equation}
In this paper, we consider all the ranges of the Mach number of the far Maxwellian, i.e. $\mathcal{M}^\infty \in \mathbb{R}$.

The Boltzmann collision operator $\mathcal{B} (F, F)$ is defined as
\begin{equation}
	\begin{aligned}
		\mathcal{B} (F, F) = \iint_{\mathbb{R}^3 \times \mathbb{S}^2} (F' F_*' - F F_*) b (\omega, v_* - v) \d \omega \d v_* \,.
	\end{aligned}
\end{equation}
Here $\omega \in \mathbb{S}^2$ is a unit vector, and $\d \omega$ represents the rotationally invariant surface integral on $\mathbb{S}^2$. The terms $F_*'$, $F'$, $F_*$ and $F$ denote the number density $F(\cdot)$ evaluated at the velocities $v_*'$, $v'$, $v_*$ and $v$ respectively, i.e.,
\begin{equation*}
	\begin{aligned}
		F_*' = F(v_*') \,, \ F' = F(v') \,, \ F_* = F(v_*) \,, \ F = F(v) \,.
	\end{aligned}
\end{equation*}
Here $(v_*', v')$ are the velocities after an elastic binary collision between two molecules with initial velocities $(v_*, v)$, and vice versa. Since both momentum and energy are conserved during the elastic collision, $v_*'$ and $v'$ can be expressed in terms of $v_*$ and $v$ as
\begin{equation}
	\begin{aligned}
		v_*' = v_* - [ (v_* - v) \cdot \omega ] \omega \,, \ v' = v + [ (v_* - v) \cdot \omega ] \omega \,,
	\end{aligned}
\end{equation}
where the unit vector $\omega \in \mathbb{S}^2$ is parallel to the deflections $v_*' - v_*$ and $v'-v$, and is therefore perpendicular to the plane of reflection. In the collision term $\mathcal{B} (F, F)$, $F_*'F'$ is the gain term, while $F_* F$ is referred to the loss term.

According to Grad's work \cite{Grad-1963}, the collision kernel $b (\omega, v_* - v)$ can be expressed in the factored form
\begin{equation}\label{b}
	\begin{aligned}
		b (\omega, v_* - v) = \tilde{b} (\theta) |v_* - v|^\gamma \,, \ \cos \theta = \frac{(v_* - v) \cdot \omega}{|v_* - v|} \,, \ -3 < \gamma \leq 1 \,,
	\end{aligned}
\end{equation}
where $\tilde{b} (\theta)$ satisfies the small deflection cutoff condition
\begin{equation}\label{Cutoff}
	\begin{aligned}
		\int_{\mathbb{S}^2} \tilde{b} (\theta) \d \omega = 1 \,, \ 0 \leq \tilde{b} (\theta) \leq \tilde{b}_0 |\cos \theta|
	\end{aligned}
\end{equation}
for some constant $\tilde{b}_0 > 0$. The cases $-3 < \gamma < 0$ and $0 \leq \gamma \leq 1$ are respectively referred to as the ``soft" and ``hard" potential cases. In particular, $\gamma = 0$ is the Maxwell potential case, and $\gamma = 1$ is the hard sphere case, in which $\tilde{b} (\theta) = \tilde{b}_0 |\cos \theta|$.

Considering the perturbation $F = \M + \sqrt{\M} f$ around the far Maxwellian $\M$, and neglecting the nonlinear term, we obtain the linearized Knudsen layer problem with a source term:
\begin{equation}\tag{KL}\label{KL}
	\left\{
	  \begin{aligned}
	  	& v_3 \partial_x f + \L f = S \,, \\
	  	& f (0,v) |_{v_3 > 0} = f_b ( v) \,, \\
	  	& \lim_{x \to + \infty} f (x,v) = 0 \,,
	  \end{aligned}
	\right.
\end{equation}
where $S = S(x,v) \in \mathrm{Null}^\perp (\L)$ and $f_b (v) = 0$ if $v_3 < 0$. The symbol $\L$ denotes the linearized Boltzmann collision operator $\B (F, F)$ around the far Maxwellian $\M$, defined as
\begin{equation}\label{Lf}
	\begin{aligned}
		\L f = & - \M^{- \frac{1}{2}} (v) \big[ \mathcal{B} (\M, \sqrt{\M} f) + \mathcal{B} ( \sqrt{\M} f , \M ) \big] \\
		= & \nu (v) f (v) - K f (v) \,,
	\end{aligned}
\end{equation}
where $\nu (v)$ denotes the collision frequency, given by
\begin{equation}\label{nu}
	\begin{aligned}
		\nu (v) = \iint_{\mathbb{R}^3 \times \mathbb{S}^2} \M (v_*) b (\omega, v_* - v) \d \omega \d v_* \overset{\eqref{Cutoff}}{=} \int_{\R^3} |v_* - v|^\gamma \M (v_*) \d v_* \,,
	\end{aligned}
\end{equation}
and the operator $K f (v)$ can be decomposed into two parts:
\begin{equation}
	\begin{aligned}\label{K-K1-K2}
		K f (v) = - K_1 f (v) + K_2 f (v) \,.
	\end{aligned}
\end{equation}
Here the loss term $K_1 f (v)$ is
\begin{equation}\label{K1}
	\begin{aligned}
		K_1 f(v) = \M^\frac{1}{2} (v) \iint_{\mathbb{R}^3 \times \mathbb{S}^2} f (v_*) \M^\frac{1}{2} (v_*) b (\omega, v_* - v) \d \omega \d v_* \,,
	\end{aligned}
\end{equation}
and the gain term $K_2 f (v)$ is
\begin{equation}\label{K2}
	\begin{aligned}
		K_2 f(v) = & \M^\frac{1}{2} (v) \iint_{\mathbb{R}^3 \times \mathbb{S}^2} \big[ \M^{- \frac{1}{2}} (v') f (v') + \M^{- \frac{1}{2}} (v_*') f(v_*') \big] \M (v_*) b (\omega, v_* - v) \d \omega \d v_* \\
		= & \iint_{\mathbb{R}^3 \times \mathbb{S}^2} \big[ \M^{\frac{1}{2}} (v'_*) f (v') + \M^{\frac{1}{2}} (v') f(v_*') \big] \M^\frac{1}{2} (v_*) b (\omega, v_* - v) \d \omega \d v_* \,,
	\end{aligned}
\end{equation}
where the last equality is derived from the collisional invariant $\M (v) \M (v_*) = \M (v') \M (v_*')$.

To deal with the general Mach number,  we define the following linearized entropy flux as in \cite{Coron-Golse-Sulem-1988-CPAM}:
\begin{equation*}
	 P(f,g)=\int_{\R ^3}v_3 fg \d v\,.
\end{equation*}
The null space $\mathrm{Null} (\L)$ of the operator $\L$ is spanned by the basis (Coron-Golse-Sulem \cite{Coron-Golse-Sulem-1988-CPAM})
\begin{equation}\label{psi-i-bases}
	\begin{aligned}
		\psi_0 = \xi_1 \sqrt{\M} \,, \  \psi_1 = \xi_2 \sqrt{\M} \,, \ \psi_2 = \Big( \sqrt{ \tfrac{ 5}{ 2}} - \tfrac{ |\xi|^2 }{\sqrt{10}}\Big) \sqrt{\M} \,, \\
		\psi_3 = \Big(  \tfrac{ \xi_3 }{ \sqrt{2}} + \tfrac{ |\xi|^2 }{\sqrt{30}}\Big) \sqrt{\M} \,, \
		\psi_4 = \Big(  -\tfrac{ \xi_3 }{ \sqrt{2}} + \tfrac{ |\xi|^2 }{\sqrt{30}}\Big) \sqrt{\M} \,.
	\end{aligned}
\end{equation}
where $\xi=\frac{v- \u }{ \sqrt{T}}$. In fact, the basis is chosen to be orthogonal in the Hilbert space $L^2 : = L^2 (\R^3)$ and the entropy flux $P$. Specifically,
\begin{equation}\label{psi-P}
	\begin{aligned}
		\int_{\R^3} \psi_i \psi_j \d v = P(\psi_i, \psi_j) =0\,, \quad  i\neq j \,.
	\end{aligned}
\end{equation}
It is well-known (see \cite{Coron-Golse-Sulem-1988-CPAM,Ukai-Yang-Yu-2003-CMP,Chen-Liu-Yang-2004-AA,Yang-2008-JMAA}) that the solvability of the Knudsen layer problem varies with the Mach number of the far Maxwellian, which is related to the signature of the entropy flux. It is important to note that
\begin{equation}
	\begin{aligned}
		P(\psi_i,\psi_i) &= \rho c \mathcal{M}^\infty\,, \quad i=0,1,2\,,\\
		P(\psi_3,\psi_3) &= \rho c \Big(\mathcal{M}^\infty + 1\Big)\,,\\
		P(\psi_4,\psi_4) &= \rho c \Big(\mathcal{M}^\infty - 1\Big)\,.
	\end{aligned}
\end{equation}
As we will see below, the number of solvability conditions $n^+$ of the Knudsen layer problem (both linear and nonlinear) changes with the Mach number $\mathcal{M}^\infty$ as follows:
\begin{equation}
		n^+=\# \{ I^+\cup I^0 \}=\left\{\begin{array}{l}
			0\,, \quad \mathcal{M}^{\infty}<-1\,, \\
			1\,, \quad-1 \leq \mathcal{M}^{\infty}<0\,, \\
			4\,, \quad 0 \leq \mathcal{M}^{\infty}<1\,, \\
			5\,, \quad 1 \leq \mathcal{M}^{\infty}\,,
			\end{array}\right.
\end{equation}
where the sets $I^+$ and $I^0$ are defined in \eqref{I+0}. Let $\mathrm{Null}^\perp (\L)$ be the orthogonal space of $\mathrm{Null} (\L)$ in $L^2$, namely,
\begin{equation*}
	\begin{aligned}
		\mathrm{Null} (\L) \oplus \mathrm{Null}^\perp (\L) = L^2 \,.
	\end{aligned}
\end{equation*}
We define the projection $\P : L^2 \to \mathrm{Null} (\L)$ by
\begin{equation}\label{Projt-P}
	\begin{aligned}
		\P f = \sum_{i = 0}^4 a_i (f) \psi_i \,, \quad a_i (f) = \frac{ \int_{\R^3} f \psi_i \d v }{\int_{\R^3}  \psi^2_i \d v} \,, \quad i = 0, 1, 2, 3, 4 \,.
	\end{aligned}
\end{equation}

 To construct the damping term, we define
\begin{equation}\label{I+0}
	I^+= \{ j \mid P(\psi_j, \psi_j) > 0 \} \,, \quad I^0= \{ j \mid P(\psi_j, \psi_j) =0 \}, \quad I^- = \{ j \mid P(\psi_j, \psi_j) < 0 \} \,.
\end{equation}
For $j \in I^0$, it follows that
\begin{equation*}
	\int_{\R^3} v_3 \psi_j^2 \d v=0 \,,
\end{equation*}
which, together with \eqref{psi-P}, implies that $v_3 \psi_j \in \mathrm{Null}^\perp (\L)$. The Fredholm alternative for the linearized Boltzmann operator implies that there exists a unique $X_j$ such that
\begin{equation}\label{Xj}
	X_j \in \mathrm{Null}^\perp (\L) \quad \text{and}\quad \! \L X_j = v_3 \psi_j \,.
\end{equation}
We now define the operator
\begin{equation}\label{damp-operator}
\begin{aligned}
	\P^+ f  &=\sum_{j\in I^+} \frac{(\psi_j,f)}{(\psi_j, \psi_j)} \psi_j \,, \\
		\P^0 f &= \sum_{j\in I^0} \frac{(X_j, v_3 f)}{(X_j, \L X_j)}  \psi_j \,,\\
		\mathbb{P} f &= \sum_{j\in I^0} \frac{(X_j, f)}{(X_j,    \L X_j)} \L X_j \,.\\
\end{aligned}
\end{equation}
One can easily check that
\begin{equation*}
	\P^+ \P^+ = \P^+ \,, \quad \mathbb{P} \mathbb{P} =\mathbb{P} \,, \quad \P^0\P^0=\P^0\,.
\end{equation*}
Therefore, the operators defined above are indeed projections. Additionally, we present the following lemma:
\begin{lemma}\label{Lmm-split}(Lemma 2.3 of \cite{Golse-2008-BIMA})
	The projections $\mathbb{P}$ and $\P^0$ satisfy the relations
	\begin{equation*}
		\mathbb{P} v_3 f =v_3 \P^0  f; \quad \mathbb{P} \L f =0\,, \quad \text{if} \quad \! \ v_3 f \in \mathrm{Null} \L\,.
	\end{equation*}
	Moreover, $ (\P^0 f, f )  \geq 0.$
\end{lemma}
We emphasize that the above operators will be utilized in designing the artificial damping term.
\subsection{Overdetermined far-field condition $\lim_{x \to + \infty} f (x, v) = 0$}

Inspired by Theorem 3.3 of \cite{Bardos-Caflisch-Nicolaenko-1986-CPAM}, the function $f (x,v)$ solving the problem
\begin{equation}\label{KL-d}
	\left\{
	  \begin{aligned}
	  	& v_3 \partial_x f + \L f = S \,, \ x > 0 \,, \\
	  	& f (0,v) |_{v_3 > 0} = f_b ( v)
	  \end{aligned}
	\right.
\end{equation}
will exhibit the asymptotic behavior
\begin{equation}\label{Asymptotic-f}
	\begin{aligned}
		\lim_{x \to + \infty} f (x,v) = f_\infty = a_\infty \psi_0 (v) + b_{1 \infty} \psi_1 (v) + b_{2 \infty} \psi_2 (v) + b_{3 \infty} \psi_3 (v) + c_\infty \psi_4 (v)
	\end{aligned}
\end{equation}
for some constants $a_\infty$, $b_{1 \infty}$, $b_{2 \infty}$, $b_{3 \infty}$ and $c_\infty$, with the functions $\psi_i (v)$ are given in \eqref{psi-i-bases}. In other words, the far-field condition $\lim_{x \to + \infty} f (x, v) = 0$ in \eqref{KL} is overdetermined for general source terms $S \in \mathrm{Null}^\perp (\L)$ and $f_b$.

However, the far-field condition $\lim_{x \to + \infty} f (x, v) = 0$ is necessary in proving the hydrodynamic limits of the Boltzmann equation with Maxwell reflection or incoming boundary conditions. In order to overcome the overdetermination of zero far-field condition, some further assumptions on the source terms $S \in \mathrm{Null}^\perp (\L)$ and $f_b$ are required. Similar to \cite{JL-2024-arXiv}, we introduce the so-called {\em vanishing sources set} ($\mathbb{VSS}$) defined by
\begin{equation}\label{ASS}
	\begin{aligned}
		\mathbb{VSS} = \big\{ (S, f_b) ; S (x,v) \in \mathrm{Null}^\perp (\L) \textrm{ and } f_b (v) \textrm{ in \eqref{KL-d} such that } \lim_{x \to + \infty} f (x, v) = 0 \big\} \,.
	\end{aligned}
\end{equation}
Since \eqref{KL-d} is linear and $f_\infty\in \mathrm{Null} (\L)$, we observe that $\tilde{f}=f-f_\infty$ solves \eqref{KL} with sources $(S,f_b-f_\infty)$. Therefore, we know that
\begin{equation}
	\begin{aligned}
		\mathbb{VSS} \neq \varnothing \,.
	\end{aligned}
\end{equation}

\subsection{Toolbox}

In this subsection, we will collect all notations, functions spaces and energy functionals that will be utilized throughout the paper.

\subsubsection{Notations}

We use the symbol $A \lesssim B$ to indicate that $A \leq C B$ for some harmless constant $C > 0$. Furthermore, $A \thicksim B$ means $C_1 B \leq A \leq C_2 B$ for some harmless constants $C_1, C_2 > 0$. Inspired by \cite{Chen-Liu-Yang-2004-AA}, we introduce the following $(x,v)$-mixed weight
\begin{equation}\label{sigma}
	\begin{aligned}
		\sigma (x,v) = & 5 (\delta x + l)^\frac{2}{3 -\gamma} \Big[ 1 - \Upsilon \Big( \tfrac{\delta x + l}{( 1 + |v - \u| )^{3 - \gamma}} \Big) \Big] \\
		& + \Big( \tfrac{\delta x + l}{(1 + |v - \u|)^{1 - \gamma}} + 3 |v - \u|^2 \Big) \Upsilon \Big( \tfrac{\delta x + l}{( 1 + |v - \u| )^{3 - \gamma}} \Big)
	\end{aligned}
\end{equation}
for small $\delta > 0$ and large $l > 0$ to be determined, where $\Upsilon : \R_+ \to [0,1]$ is a smooth monotonic function satisfying
\begin{equation}\label{Ups-Cut}
	\begin{aligned}
		\Upsilon (x) =
		\left\{
		\begin{aligned}
			1 \quad & \textrm{for } \ 0 \leq x \leq 1 \,, \\
			0 \quad & \textrm{for } \ x \geq 2 \,.
		\end{aligned}
		\right.
	\end{aligned}
\end{equation}
This weight has also been used in works such as \cite{Wang-Yang-Yang-2006-JMP,Wang-Yang-Yang-2007-JMP,Yang-2008-JMAA}. For any small $\hbar > 0$ to be determined, we denote by
\begin{equation}
	\begin{aligned}
		f_\sigma : = e^{\hbar \sigma} f \ (\forall f = f (x,v)) \,.
	\end{aligned}
\end{equation}
We also introduce a weighted function
\begin{equation}\label{wv}
	\begin{aligned}
		w_{\beta, \vartheta} (v) = (1 + |v|)^\beta e^{\vartheta |v - \u|^2}
	\end{aligned}
\end{equation}
for the constants $\beta \in \R$ and $ \vartheta \geq 0 $. Then, for $\alpha \in \R$, we define a weight function
\begin{equation}\label{z-alpha}
	z_\alpha (v) =
	\left\{
	\begin{array}{cl}
		|v_3|^\alpha \,, & |v_3| < 1 \,, \\
		1 \,, & |v_3| \geq 1 \,. \\
	\end{array}
	\right.
\end{equation}

We now define the numbers $\beta_\gamma$ as follows:
\begin{equation}\label{beta-gamma}
	\begin{aligned}
		\beta_\gamma =
		\left\{
		\begin{aligned}
			0 \,, \quad & \textrm{if } 0 \leq \gamma \leq 1 \,, \\
			- \tfrac{\gamma}{2} \,, \quad & \textrm{if } - 3 < \gamma < 0 \,.
		\end{aligned}
		\right.
	\end{aligned}
\end{equation}
Moreover, the constant $\Theta$ is introduced by
\begin{equation}\label{Theta}
	\Theta=\tfrac{1-\gamma}{3-\gamma}\geq 0\,.
\end{equation}
We note that the above numbers will be utilized in designing various weighted norms.

\subsubsection{Functions spaces}

Based on the above weights, we introduce the spaces $X^\infty_{ \beta, \vartheta} (A)$ and $Y^\infty_{  \beta, \vartheta} (A)$ over $(x,v) \in {\Omega_A} \times \R^3$ for $\beta, \vartheta \in \R$,  and $0 < A \leq \infty$ with the weighted norms
\begin{equation}\label{XY-space}
	\begin{aligned}
		& \| f \|_{A; \beta, \vartheta} : = \|   w_{\beta, \vartheta} f \|_{L^\infty_{x,v}} = \sup_{(x, v) \in \Omega_A \times \R^3} |   w_{\beta, \vartheta} (v) f (x,v) | \,, \\
		& \LL f \RR_{A;   \beta, \vartheta} : = \| \sigma^{\frac{1}{2}}_x  w_{\beta, \vartheta} f \|_{L^\infty_{x,v}} \,,
	\end{aligned}
\end{equation}
respectively. For simplicity, we denote
\begin{equation}
	\begin{aligned}
		 \LL f \RR_{  \beta, \vartheta} : = \LL f \RR_{\infty;   \beta, \vartheta}\,.
	\end{aligned}
\end{equation}
On the boundary $\Sigma : = \partial {\Omega_A} \times \R^3 $, we introduce the spaces $L^\infty_\Sigma$, $X^\infty_{\beta, \vartheta, \Sigma}$ and $Y^\infty_{  \beta, \vartheta, \Sigma}$ endowed with the following norms
\begin{equation}\label{XY-Sigma-space}
	\begin{aligned}
		& \| f \|_{L^\infty_\Sigma} : = \sup_{(x, v) \in \Sigma} |f (x,v)| \,, \ \| f \|_{\beta, \vartheta, \Sigma} : = \| w_{\beta, \vartheta}   f \|_{L^\infty_\Sigma} \,, \\
		& \LL f \RR_{  \beta, \vartheta, \Sigma} : = \| \sigma^{\frac{1}{2}}_x  w_{\beta, \vartheta} f \|_{L^\infty_\Sigma} \,,
	\end{aligned}
\end{equation}
respectively. Let $L^p_x$ and $L^p_v$ with $1 \leq p \leq \infty$ be the standard $L^p$ space over $x \in \Omega_A$ and $v \in \R^3$, respectively. $(\cdot, \cdot)$ denotes the
$L^2$ inner product in $\R^3_v$. Moreover, the norm $\| \cdot \|_{L^r_v (L^p_x)}$ is defined by
\begin{equation*}
	\begin{aligned}
		\| \cdot \|_{L^r_v (L^p_x)} = \| \big( \| f (\cdot, v) \|_{L^x_p} \big) \|_{L^r_v} \,.
	\end{aligned}
\end{equation*}
We also introduce the space $L^\infty_x L^2_v$ endowed with the norm
\begin{equation*}
	\begin{aligned}
		\| f \|_{L^\infty_x L^2_v} = \sup_{x \in \Omega_A} \| f (x, \cdot) \|_{L^2_v} \,.
	\end{aligned}
\end{equation*}
Furthermore, the space $L^2_{x,v}$ over $\Omega_A \times \R^3$ endows with the norm
\begin{equation}
	\begin{aligned}
		\| f \|_A = \big( \iint_{\Omega_A \times \R^3} |f (x,v) |^2 \d v \d x \big)^\frac{1}{2} \,.
	\end{aligned}
\end{equation}
If $A = \infty$, one simply denotes by
\begin{equation*}
	\begin{aligned}
		\| f \| : = \| f \|_\infty = \big( \int_0^\infty \int_{\R^3} |f (x,v) |^2 \d v \d x \big)^\frac{1}{2} \,.
	\end{aligned}
\end{equation*}

One introduces the out normal vector $n (0) = (0,0,-1)$ and $n (A) = (0,0,1)$ of the boundary $\partial \Omega_A$. Define
\begin{equation*}
	\begin{aligned}
		\Sigma_\pm = \{ (x, v) \in \Sigma ; \pm n (x) \cdot v > 0 \} \,.
	\end{aligned}
\end{equation*}
We then define the space ${L^2_{\Sigma_\pm}}$ over $(x,v) \in \Sigma_\pm$ as follows:
\begin{equation}\label{L2-Sigma+}
	\begin{aligned}
		& \| f \|^2_{{L^2_{\Sigma_\pm}}} = \| f \|^2_{{L^2_{\Sigma_\pm^0}}} + \| f \|^2_{{L^2_{\Sigma_\pm^A}}} \,, \ \| f \|^2_{{L^2_{\Sigma_\pm^0}}} = \int_{ \pm v_3 < 0} |f (0, v)|^2 \d v \,, \\
		& \| f \|^2_{{L^2_{\Sigma_\pm^A}}} = \int_{\pm v_3 > 0} |f (A, v)|^2 \d v \,.
	\end{aligned}
\end{equation}

\subsubsection{Energy functionals}\label{Subsubsec:EF}

We now introduce total energy functional as
\begin{equation}\label{Eg-lambda}
	\begin{aligned}
		\mathscr{E}^A (g) = \mathscr{E}_\infty^A (g) + \mathscr{E}_{\mathtt{cro}}^A (g)  + \mathscr{E}_{2 }^A (w_{\beta + \beta_\gamma, \vartheta}g) + \mathscr{E}_{2 }^A (g)
	\end{aligned}
\end{equation}
for $0 < A \leq \infty$, where the weighted $L^\infty_{x,v}$ energy functional $\mathscr{E}_\infty^A (g)$ is defined as
\begin{equation}\label{E-infty}
	\begin{aligned}
		\mathscr{E}_\infty^A (g) = \LL  g \RR_{A;   \beta, \vartheta} \,,
	\end{aligned}
\end{equation}
the cross energy functional $ \mathscr{E}_{\mathtt{cro}}^A (g) $ connecting the $L^\infty_{x,v}$ and $L^2_{x,v}$ estimates is defined as
\begin{equation}\label{E-cro}
	\begin{aligned}
		\mathscr{E}_{\mathtt{cro}}^A (g) = \| \nu^\frac{1}{2} z_{- \alpha} \sigma_x^ \frac{1}{2}   w_{\beta + \beta_\gamma, \vartheta} g \|_A + \| \nu^{ - \frac{1}{2} } z_{- \alpha} \sigma_x^ \frac{1}{2}   z_1 w_{\beta + \beta_\gamma, \vartheta} \partial_x g \|_A \,,
	\end{aligned}
\end{equation}
and the weighted $L^2_{x,v}$ energy functional $\mathscr{E}_{2 }^A (g)$ is represented by
\begin{equation}\label{E2-lambda}
	\begin{aligned}
		\mathscr{E}_{2 }^A (g) = \| ( \delta x + l )^{ - \frac{\Theta}{2} } \P g \|_A + \| \nu^\frac{1}{2} \P^\perp  g \|_A  \,.
	\end{aligned}
\end{equation}
Here the constants $\beta_\gamma$ and $\Theta$ are given in \eqref{beta-gamma} and \eqref{Theta}, respectively.

We further introduce the source energy functional $\mathscr{A}^A (h)$ and the boundary source energy functional $\mathscr{B} (\varphi)$ as follows:
\begin{equation}\label{Ah-lambda}
	\begin{aligned}
		& \mathscr{A}^A (h) = \mathscr{A}_\infty^A (h) + \mathscr{A}_{\mathtt{cro}}^A (h) + \mathscr{A}_{2 }^A (h) + \mathscr{A}_{2 }^A (w_{\beta+\beta_\gamma,\vartheta} h)\,, \\
		& \mathscr{B} (\varphi) = \mathscr{B}_\infty (\varphi) + \mathscr{B}_{\mathtt{cro}} (\varphi) + \mathscr{B}_2 (\varphi) +\mathscr{B}_2 (w_{\beta+\beta_\gamma,\vartheta}\varphi) \,,\\
		& \mathscr{C} (f_b) = \mathscr{C}_\infty (f_b) + \mathscr{C}_{\mathtt{cro}} (f_b) + \mathscr{C}_2 (f_b)+ \mathscr{C}_2 (w_{\beta+\beta_\gamma,\vartheta}f_b) \,,\\
	\end{aligned}
\end{equation}
where
\begin{equation}\label{As-def}
	\begin{aligned}
		\mathscr{A}_\infty^A (h) = & \LL \nu^{-1} h \RR_{A;   \beta, \vartheta} \,, \ \mathscr{A}_{\mathtt{cro}}^A (h) = \| \nu^{ - \frac{1}{2} } z_{- \alpha} \sigma_x^ \frac{1}{2}   w_{\beta + \beta_\gamma, \vartheta} h \|_A \,, \\
		\mathscr{A}_{2 }^A (h) =  &\| ( \delta x + l )^{\frac{\Theta}{2 } } \P h \|_A + \| \nu^{ - \frac{1}{2} } \P^\perp  h \|_A  \,,
	\end{aligned}
\end{equation}
and
\begin{equation}\label{B-def}
	\begin{aligned}
		\mathscr{B}_\infty (\varphi) = & \| \sigma_x^ \frac{1}{2} (A, \cdot)     w_{\beta, \vartheta} \varphi \|_{L^\infty_v } \,, \quad \mathscr{B}_2 (\varphi) = \| |v_3|^\frac{1}{2}  \varphi \|_{L^2_{ \Sigma_-^A }} \,, \\
		\mathscr{B}_{\mathtt{cro}} (\varphi) = & \| |v_3|^\frac{1}{2} \sigma_x^ \frac{1}{2}     z_{- \alpha} w_{\beta + \beta_\gamma, \vartheta} \varphi \|_{L^2_{ \Sigma_-^A }} \,,
	\end{aligned}
\end{equation}
and
\begin{equation}\label{C-def}
	\begin{aligned}
		\mathscr{C}_\infty (f_b) = & \| \sigma_x^ \frac{1}{2} (0, \cdot)     w_{\beta, \vartheta} f_b \|_{L^\infty_v } \,, \quad \mathscr{C}_2 (f_b) = \| |v_3|^\frac{1}{2}  f_b \|_{L^2_{ \Sigma_-^0 }} \,, \\
		\mathscr{C}_{\mathtt{cro}} (f_b) = & \| |v_3|^\frac{1}{2} \sigma_x^ \frac{1}{2}     z_{- \alpha} w_{\beta + \beta_\gamma, \vartheta} f_b \|_{L^2_{ \Sigma_-^0 }} \,.
	\end{aligned}
\end{equation}
At the end, we define the following source energy functional to deal with the degenerate case $\mathcal{M}^\infty =0\,,\pm 1$:
\begin{equation}\label{D-def1}
	\mathscr{D}_\hbar^\infty (h) = \mathscr{E}^\infty (e^{\hbar \sigma(x,\cdot) } \tfrac{1}{v_3} \int_{x}^{\infty} \mathbb{P} h (y, \cdot) \d y)\,,
\end{equation}
where $ \mathscr{E}^\infty ( \cdot ) $ is defined in \eqref{Eg-lambda} with $A = \infty$.

\subsection{Main results}

\subsubsection{Linear problem \eqref{KL}}

We now present the existence result on the linear Knudsen layer equation \eqref{KL}. Before doing this, we first outline the parameters assumptions occurred in $\mathscr{E}^A (g)$ (see \eqref{Eg-lambda}) and $ \mathscr{A}^A (h) $, $\mathscr{B} (\varphi)$, $\mathscr{C} (f_b)$(see \eqref{Ah-lambda}) above.

{\bf Parameters Hypotheses (PH):} For the parameters $\{ \gamma, \delta, \hbar, \vartheta, l, \alpha, \beta, \rho, T, \u , \Theta \}$, we assume that
\begin{itemize}
	\item $- 3 < \gamma \leq 1$;
	
	\item $ \delta, \hbar, \vartheta > 0 $ are all sufficiently small, and $l > 1$ is large enough;
	
	\item $0 < \alpha < \mu_\gamma$, where $\mu_\gamma > 0$ is given in Lemma \ref{Lmm-Kh-L2} below;
	
	\item $\beta \geq \beta_\gamma + \frac{1 - \gamma}{2} + \max \{ 0, - \gamma \}$, where $\beta_\gamma$ is introduced in \eqref{beta-gamma};
	
    \item $\rho, T > 0, \u =(0,0, \u_3)$ with $ \u_3 \in \R$ \,.

	\item $\Theta\geq 0$ is defined in \eqref{Theta} \,.
\end{itemize}

We now state the first result of this paper.

\begin{theorem}\label{Thm-Linear}
	Let the mixed weight $\sigma (x,v)$ be given in \eqref{sigma}. Assume that the parameters $\{ \gamma, \delta, \hbar, \vartheta, l, \alpha, \beta, \rho, T, \u , \Theta\}$ satisfy the hypotheses {\bf (PH)}, and the source terms $(S, f_b) \in \mathbb{VSS}_{}$ defined in \eqref{ASS}. We assume that
	\begin{equation}
		\begin{aligned}
			\mathscr{A}^\infty ( e^{ \hbar \sigma } S ) < \infty \,,\  \mathscr{C}(f_b) < \infty  \,,
		\end{aligned}
	\end{equation}
	where $ \mathscr{A}^\infty ( \cdot ) $ is defined in \eqref{Ah-lambda} with $A = \infty$, , and $ \mathscr{C}(f_b) $ is given in \eqref{C-def}. Moreover, for the degenerate case, we further assume
	\begin{equation}\label{Dh-bnd}
		\mathscr{D}^\infty_\hbar ( S ) < \infty.
	\end{equation}
	 Then the linear Knudsen layer equation \eqref{KL} admits a unique solution $f (x, v)$ satisfying
	\begin{equation}\label{Bnd-f}
		\begin{aligned}
			\mathscr{E}^\infty ( e^{ \hbar \sigma } f ) \leq C ( \mathscr{A}^\infty ( e^{ \hbar \sigma } S ) +\mathscr{D}^\infty_\hbar (S) + \mathscr{C}(e^{\hbar \sigma(0,\cdot)}f_b) )
		\end{aligned}
	\end{equation}
	for some constant $C > 0$. Moreover, for given $S \in \mathrm{Null}^\perp (\L)$, the set $\mathbb{VSS}$ forms $C^1$ manifold of codimension $\# \{ I^+\cup I^0 \}$, where $I^+$ and $I^0$ are defined in \eqref{I+0}.
\end{theorem}
\begin{remark}
	Based on Theorem \ref{Thm-Linear}, we summary the number of the solvability conditions in Table \ref{table1}.
	 \begin{table}[h!]
		\begin{center}
		\caption{}
		\begin{tabular}{c|c}\label{table1}
		\textbf{Mach number} & \textbf{The number of solvability conditions }  \\
			\hline
		$\mathcal{M}^\infty <-1$ & $0$ \\
		$-1 \leq \mathcal{M}^\infty < 0$  & $1$\\
		$0 \leq\mathcal{M}^\infty <1 $  & $4$ \\
		$1 \leq \mathcal{M}^\infty $  & $5$
			\end{tabular}
		\end{center}
		\end{table}
\end{remark}

\begin{remark}[Exponential decay]
	Since Lemma \ref{Lmm-sigma} implies that $ \sigma (x,v) \geq c ( \delta x + l )^\frac{2}{3 - \gamma} $, and
	$$\sigma_x\geq c_1 \min \{( \delta x + l )^{-\Theta}, (1+|v-\u|)^{-1+\gamma}\}\geq c_1 ( \delta x + l )^{-\Theta} (1+|v-\u|)^{-1+\gamma}$$
	for $- 3 < \gamma \leq 1$, one has
	\begin{equation*}
		\begin{aligned}
			\mathscr{E}^\infty ( e^{ \hbar' \sigma } f ) \geq \| \sigma_x^ \frac{1}{2} w_{\beta, \vartheta} e^{ \hbar' \sigma } f \|_{L^\infty_{x,v}} \geq C' e^{ \frac{1}{2} c ( \delta x + l )^\frac{2}{3 - \gamma} } e^{ \frac{1}{2} \vartheta |v|^2 }( \delta x + l )^{-\frac{\Theta}{2}} (1+|v-\u|)^{\frac{-1+\gamma}{2}} | f (x,v) |
		\end{aligned}
	\end{equation*}
	uniformly in $(x,v) \in \R_+ \times \R^3$. Together with the bound \eqref{Bnd-f} in Theorem \ref{Thm-Linear}, the solution $f (x,v)$ to the problem \eqref{KL} enjoys the pointwise decay behavior
	\begin{equation}\label{Decay}
		\begin{aligned}
			| f (x,v) | \lesssim e^{ - \frac{1}{4} c ( \delta x + l )^\frac{2}{3 - \gamma} } e^{ - \frac{1}{4} \vartheta |v|^2 } \,.
		\end{aligned}
	\end{equation}
\end{remark}

\subsubsection{Application to nonlinear problem}

In this part, we will employ the linear theory constructed in Theorem \ref{Thm-Linear} to investigate the following nonlinear problem \eqref{KL-NL}. Let
\begin{equation}
	\begin{aligned}
		f = \tfrac{F - \M}{\sqrt{\M}} \,, \ \tilde{h} = \tfrac{H}{\sqrt{\M}} \,, \ \tilde{f}_b = \tfrac{F_b}{\sqrt{\M}} \,.
	\end{aligned}
\end{equation}
Then the problem \eqref{KL-NL} can be equivalently rewritten as
\begin{equation}\label{KL-NL-f}
	\left\{
	    \begin{aligned}
	    	& v_3 \partial_x f + \L f = \Gamma ( f, f ) + \tilde{h} \,, \ x > 0 \,, v \in \R^3 \,, \\
	    	& f (0, v) |_{v_3 > 0} =  \tilde{f}_b (v) \,, \\
	    	& \lim_{x \to + \infty} f (x,v) = 0 \,,
	    \end{aligned}
	\right.
\end{equation}
where the nonlinear operator $ \Gamma (f, f) $ is defined as
\begin{equation}\label{Gamma-ff}
	\begin{aligned}
		\Gamma (f, f) = \tfrac{1}{\sqrt{\M}} \mathcal{B} ( f \sqrt{\M}, f \sqrt{\M} ) \,.
	\end{aligned}
\end{equation}
We can establish the following result.

\begin{theorem}\label{Thm-Nonlinear}
	Let the mixed weight $\sigma (x,v)$ be given in \eqref{sigma}. Assume that the parameters $\{ \gamma, \delta, \hbar, \vartheta, l, \alpha, \beta, \rho, T, \u ,\Theta\}$ satisfy the hypotheses {\bf (PH)}, and the source terms $(\tfrac{H}{\sqrt{\M}}, \tfrac{F_b}{\sqrt{\M}}) \in \mathbb{VSS}_{}$ defined in \eqref{ASS}. Then there is a small $\varsigma_0 > 0$ such that if
	\begin{equation}\label{varsigma}
		\begin{aligned}
			\varsigma : = \mathscr{A}^\infty ( e^{\hbar \sigma} \tfrac{H}{\sqrt{\M}} ) + \mathscr{D}_\hbar (\tfrac{H}{\sqrt{\M}} ) + \mathscr{C} \big( e^{\hbar \sigma(0, \cdot )}\tfrac{F_b}{\sqrt{\M}} \big) \leq \varsigma_0 \,,
		\end{aligned}
	\end{equation}
	then the nonlinear problem \eqref{KL-NL} admits a unique solution $F (x, v)$ enjoying the bound
	\begin{equation}
		\begin{aligned}
			\mathscr{E}^\infty ( e^{ \hbar \sigma } \tfrac{F - \M}{\sqrt{\M}} ) \leq C \varsigma
		\end{aligned}
	\end{equation}
	for and some constant $C > 0$, where the functionals $ \mathscr{E}^\infty ( \cdot ) $, $ \mathscr{A}^\infty ( \cdot ) $, $\mathscr{C}(\cdot)$ and $
	\mathscr{D}_\hbar$ are respectively defined in \eqref{Eg-lambda}, \eqref{Ah-lambda} and \eqref{D-def1} with $A = \infty$. Moreover, for given $\tfrac{H}{\sqrt{\M}} \in \mathrm{Null}^\perp (\L)$, the set $\mathbb{VSS}$ forms $C^1$ manifold of codimension $\# \{ I^+\cup I^0 \}$, where $I^+$ and $I^0$ are defined in \eqref{I+0}.
\end{theorem}
\begin{remark}
	By Theorem \ref{Thm-Nonlinear}, we know that the number of the solvability conditions of the nonlinear problem is the same as the linear case, see Table \ref{table1}.
\end{remark}

\subsection{Outline of existence of the solutions to the system $\eqref{KL}$}\label{Subsec:OESKL}

We now outline the process for solving the Knudsen layer equation \eqref{KL}.

First, we apply Lemma \ref{Lmm-split} to decompose the Knudsen layer equation \eqref{KL} into
\begin{equation}\label{KL-split1}
	\left\{
		\begin{aligned}
		&v_3 \partial_x (I-\P^0)f +\L (I- \P^0)f =(I-\mathbb{P})S \,,\\
		&\lim_{x \to + \infty} (I-P^0)f =0 \,;
	\end{aligned}\right.
\end{equation}
and
\begin{equation}\label{KL-split2}
	\left\{
		\begin{aligned}
		v_3 \partial_x \P^0f  = \mathbb{P} S \,,\\
		\lim_{x \to + \infty} \P^0f =0 \,,
	\end{aligned}\right.
\end{equation}
where $\P^0$ and $\mathbb{P}$ are defined in \eqref{damp-operator}. Note that the solution of \eqref{KL-split2} can be explicitly expressed by
\begin{equation*}
	\P^0f (x, v) =-\frac{\int_{x}^{\infty} \mathbb{P}S \d x'}{v_3}\,.
\end{equation*}
 We also remark that in the nondegenerate case, i,e., $\mathcal{M}^\infty \neq 0\,,\pm1$, $\P^0 f(x,v) \equiv 0 $, and this decomposition is no longer required (see\cite{Ukai-Yang-Yu-2003-CMP,Chen-Liu-Yang-2004-AA,Wang-Yang-Yang-2007-JMP}). We first solve \eqref{KL-split2} and then solve \eqref{KL-split1} with the boundary condition
\begin{equation*}
	(I-\P^0)f |_{v_3 > 0} = f_{b}-\P^0f (0, v).
\end{equation*}

Second, in order to prove the existence of \eqref{KL-split1},
we consider the Knudsen layer equation with an artificial damping term:
\begin{equation}\tag{KLe}\label{KL-Damped}
	\left\{	
	\begin{aligned}
		& v_3 \partial_x g + \L g = h-\bar{\alpha}(\delta x + l)^{-\frac{\Theta}{2}} \P^+ v_3 g-\bar{\beta} (\delta x + l)^{-\frac{\Theta}{2}} \P^0 g  \,,\\
		& g (0,v) |_{v_3 > 0} = f_{b} \,, \\
		& \lim_{x \to + \infty} g (x,v) = 0 \,.
	\end{aligned}
	\right.
\end{equation}
Here $\bar{\alpha},\bar{\beta}= O(1)\hbar \ll 1$ will be clarified later, and the operators $\P^+, \P^0$ are defined in \eqref{damp-operator}. A key step of this paper is to establish the existence of \eqref{KL-Damped}.

Third, to demonstrate the existence of \eqref{KL-Damped}, we consider the problem in a finite slab with incoming boundary condition at $x = 0$ and $x = A$. Specifically, for $(x, v) \in \Omega_A \times \R^3$ with $\Omega_A = \{ x; 0 < x < A \}$,
\begin{equation}\tag{Ap-eq}\label{A1}
	\left\{	
	\begin{aligned}
		& v_3 \partial_x g + \L g = h-\bar{\alpha}(\delta x + l)^{-\frac{\Theta}{2}} \P^+ v_3 g-\bar{\beta} (\delta x + l)^{-\frac{\Theta}{2}} \P^0 g \,, \\
		& g (0,v) |_{v_3 > 0} = f_b \,, \\
		& g (A, v) |_{v_3 < 0} = \varphi_A (v) \,.
	\end{aligned}
	\right.
\end{equation}
The assumptions on the function $\varphi_A (v) $ will be clarified later. Here the boundary condition at $x = A$ represents the incoming data. We will show that the solution of the approximate problem \eqref{A1} converges to the solution of \eqref{KL-Damped} as $A \to + \infty$.

Finally, we will establish the existence and uniqueness of \eqref{KL} by showing that the solution of \eqref{KL-Damped} is a solution of \eqref{KL} under additional solvability conditions on the source terms.

\subsection{Methodology and novelties}\label{Subsec:MI}

  In this subsection, we methodologically sketch the proof of the main results and illustrate the novelties of this paper. As shown in the previous subsection, the core of current work is to derive the uniform estimates of the approximate problem \eqref{A1} associated with the parameter $A > 1$. The main ideas are displayed as follows.

  {\bf (I) Choice of the mixed weight $\sigma (x,v)$.} The major part of the equation \eqref{KL} reads
  	\begin{equation}\label{MajorPart}
  		\begin{aligned}
  			v_3 \partial_x f + \nu (v) f = s.o.t. \textrm{ (some other terms)} \,.
  		\end{aligned}
  	\end{equation}
  	For the hard sphere model $ \gamma = 1 $, $ v_3 $ and $ \nu (v) $ have the same order as $ |v| \gg 1 $, i.e. $ |v_3| \lesssim  \nu (v) $. Roughly speaking, the $ x $-decay can be expected from
  	\begin{equation*}
  		\begin{aligned}
  			\partial_x f + c f = s.o.t. \,,
  		\end{aligned}
  	\end{equation*}
  	which means that the expected $ x $-decay is $ e^{- c x} $.
  	
  	For the cases $ - 3 < \gamma < 1 $, $ |v_3| \lesssim \nu (v) |v_3|^{1 - \gamma} $. The Case is much more complicated. Note that \eqref{MajorPart} implies
  	\begin{equation*}
  		\begin{aligned}
  			\partial_x ( e^{\frac{\nu (v)}{v_3} x} f ) = s.o.t. \,.
  		\end{aligned}
  	\end{equation*}
  	This inspires us to introduce an $ (x,v) $-mixed weight $ \sigma (x,v) $ to deal with the power $ \frac{\nu (v)}{v_3} x $. More precisely, \eqref{MajorPart} reduces to
  	\begin{equation*}
  		\begin{aligned}
  			v_3 \partial_x ( e^{ \hbar \sigma } f ) + ( \nu (v) - \hbar v_3 \sigma_x ) e^{ \hbar \sigma } f = s.o.t. \,,
  		\end{aligned}
  	\end{equation*}
  	which inspires us to find a weight $\sigma (x,v)$ such that $v_3 \sigma_x (x,v) \thicksim \nu (v)$. Chen-Liu-Yang's work \cite{Chen-Liu-Yang-2004-AA} designed the weight $ \sigma (x,v) $, which satisfies
  	$$ |v_3| \sigma_x (x,v) \lesssim \nu (v) \,, \sigma (x,v) \geq c (\delta x + l)^\frac{2}{3 - \gamma} \,, \sigma_x (x,v) \lesssim (\delta x + l)^{- \frac{1-\gamma}{3-\gamma}} \,. $$
  	The derivative $\sigma_x$ of $ \sigma $ actually balances the disparity of $ |v_3| $ and $ \nu (v) $. Chen-Liu-Yang investigated the Knudsen layer equation with the nondegenerate moving boundary condition with hard potential collision kernel $0 < \gamma \leq 1$.
  	
  	{\bf (II) Artificial damping mechanism.} It is well-known that the linearized Boltzmann operator $\L$ does not provide the coercivity structure in the null space $\mathrm{Null} (\L)$, i.e., macroscopic damping effect. Previous studies introduced various artificial damping terms to obtain the macroscopic coercivity, and then remove these terms using appropriate techniques. For instance, in \cite{Coron-Golse-Sulem-1988-CPAM,Golse-Perthame-Sulem-1988-ARMA,HJW-2024-preprint,HJW-2023-AMASES,HW-2022-SIMA}, the damping $\eps f$ was introduced to deal with the linear Knudsen layer equation for hard sphere collision case with various boundary conditions, and later removed by taking $\eps \to 0$ using compact arguments. In \cite{Chen-Liu-Yang-2004-AA,Wang-Yang-Yang-2006-JMP,Wang-Yang-Yang-2007-JMP,Yang-2008-JMAA}, the artificial damping $- \gamma P_0^+ \xi_1 f$ ($\gamma > 0$) coming from the eigenspace corresponding to positive eigenvalues of the linear operator $P_0 \xi_1 P_0$ are applied to deal with the incoming boundary conditions for the nondegenerate case, i.e. the Mach number of the far Maxwellian $\mathcal{M}^\infty \neq 0 \,,\pm1$. In order to return the original equation, a further assumption on the incoming data $a (\xi)$ (the data $a (\xi)$ vanishes in the eigenspace corresponding to positive eigenvalues of the linear operator $P_0 \xi_1 P_0 $, i.e., $ P_0^+ \xi_1 P_0 a = 0 $) will be imposed such that the artificial damping $- \gamma P_0^+ \xi_1 f$ vanishes by the Gr\"onwall inequality argument. For the degenerate case, Golse \cite{Golse-2008-BIMA} introduced the artificial damping term $ - \alpha \Pi_+(v_1\phi) - \mathbf{p}(\phi)$, later removing it through similar arguments.
	
	In this paper, to deal with the full ranges of the Mach number of the far Maxwellian with all the cutoff collision kernels, we introduce the artificial damping term:
	$$-\bar{\alpha}(\delta x + l)^{-\frac{\Theta}{2}} \P^+ v_3 g-\bar{\beta} (\delta x + l)^{-\frac{\Theta}{2}} \P^0 g.$$
	Finally, we impose additional assumptions on the source terms to remove the damping term, establishing the solvability conditions of the Knudsen layer problem.

    {\bf (III) Designing the uniform norms of \eqref{A1}.} Now we illustrate the process of deriving the uniform bounds of the problem \eqref{A3-lambda} below (equivalently \eqref{A1}). Due to the complication of deriving the uniform bounds, the following sketch map will be initially drawn for the sake of readers' intuition (see Figure \ref{Fig1} below).

    \begin{figure}[h]
    	\begin{center}
    		\begin{tikzpicture}
    			\draw (-0.5,0) rectangle (5.45,-1.1);
            \node at (2.5,-0.55) {\footnotesize $\textcolor{red}{ \mathscr{E}^A_{\infty} (g_\sigma) } = \LL g_\sigma \RR_{A;   \beta, \vartheta}$};
            \draw[->] (2.5,-1.2) -- (2.5,-2.2);
    			\draw (2.5,-1.7) node[right]{\footnotesize Lemma \ref{Lmm-Y-bnd}};
				\draw (-0.5,-2.3) rectangle (5.45,-3.6);
            \node at (2.5,-2.95) {\footnotesize $\| z_{\alpha'} \sigma_x^{\frac{1}{2}} w_{\beta, \vartheta} g_\sigma \|_{L^\infty_x L^2_v}$};
            \draw[->] (2.5,-3.7) -- (2.5,-4.7);

    			\draw (2.5,-4.2) node[right]{\footnotesize Lemma \ref{Lmm-Linfty-L2}};

				\draw (-0.5,-4.8) rectangle (5.45,-6.8);
            \node at (2.5,-5.4) {\footnotesize $\textcolor{red}{ \mathscr{E}^A_{\mathtt{cro}} (g_\sigma) } = \| \nu^{\frac{1}{2}} z_{- \alpha} \sigma_x^{\frac{1}{2}} w_{\beta, \vartheta} g_\sigma \|_A$};
            \node at (2.5,-6) {\footnotesize $ + \| \nu^{-\frac{1}{2}} z_{- \alpha} \sigma_x^{\frac{1}{2}} z_1 w_{\beta, \vartheta} \partial_x g_\sigma \|_A $};
            \draw[-] (2.5,-6.9) -- (2.5,-7.9);
			\draw[->] (2.5,-7.9) -- (7.9,-7.9);
    			
    			\draw (2.5,-7.4) node[right]{\footnotesize Lemma \ref{Lmm-L2xv-alpha}};

				\draw (8,-7.2) rectangle (14.5,-8.5);
\node at (11.25,-7.6) {\footnotesize $ \textcolor{red}{ \mathscr{E}^A_2(w_{\beta+\beta_\gamma,\vartheta} g_\sigma)}= \| \nu^{-\frac{1}{2}} \P^\perp w_{\beta+\beta_\gamma,\vartheta} g_\sigma\|_A $};
\node at (11.25,-8.2) {\footnotesize $+ \| (\delta x + l)^{-\frac{\Theta}{2}} \P w_{\beta+\beta_\gamma,\vartheta} g_\sigma\|_A $};
\draw[->] (11.25,-7.1) -- (11.25,-6);

				\draw (11.25,-6.5) node[right]{\footnotesize Lemma \ref{Lmm-Linfty-L2}};

				\draw (8,-4.8) rectangle (14.5,-5.9);
				\node at (11.25,-5.35) {\footnotesize $ \| (\delta x + l)^{-\frac{\Theta}{2}} \P w_{\beta+\beta_\gamma,\vartheta} g_\sigma\|_A $};
				\draw (11.25,-4.2) node[right]{\footnotesize Lemma \ref{Lmm-L2xv}};
				\draw[->] (11.25,-4.7)--(11.25,-3.7);

				\draw (8,-2.3) rectangle (14.5,-3.6);
				\node at (11.25,-2.8) {\footnotesize $ \textcolor{red}{\mathscr{E}^A_2(g_\sigma)} = \| (\delta x + l)^{-\frac{\Theta}{2}} \P g_\sigma\|_A + \|\nu^{-\frac{1}{2}} \P^\perp g_\sigma \|_A$};
				\draw[->] (11.25,-2.2)--(11.25,-1.2);

    			\draw (8,-1.1)--(14.5,-1.1)--(14.5,0)--(8,0)--(8,-1.1)--cycle;
    			\draw (11.25,-1.65) node[right]{\footnotesize Lemma \ref{Lmm-L2xv.}};
    			\draw (10.1,-0.5) node[right]{\footnotesize Source terms};
				
    			\draw (-0.4,-9.7) node[right]{$A \to B$ : $A \lesssim B \ +$ some known quantities};
    			
    			\draw (-0.35,-9.4) rectangle (0.1,-9.9);
    			\draw (0.55,-9.4) rectangle (1,-9.9);
    			
    		\end{tikzpicture}
    	\end{center}
        \caption{Derivation of uniform bounds for the approximate equation \eqref{A1}. Here we denote by $g_\sigma = e^{\hbar \sigma} g$.}\label{Fig1}
    \end{figure}

    Based on the Figure \ref{Fig1}, we now illustrate the main ideas. We mainly want to control the weighted $L^\infty_{x,v}$ quantity $\mathscr{E}^A_\infty (g_\sigma)$. However, the operator $K$ is not compact in the weighted $L^\infty_{x,v}$ spaces, which fails to obtain a closed estimate in the $L^\infty_{x,v}$ framework. By applying the property of $K$ in Lemma \ref{Lmm-Kh-2-infty}, the quantity $\mathscr{E}^A_\infty (g_\sigma)$ can be bounded by the norm $ \| z_{\alpha'} \sigma_x^ \frac{1}{2}   w_{\beta, \vartheta} g_\sigma \|_{L^\infty_x L^2_v} $, see Lemma \ref{Lmm-Y-bnd}. In this step, the $L^\infty_{x,v}$ bounds for the operators $Y_A, Z, U$ defined in \eqref{YAn-f}-\eqref{Rn-f}-\eqref{U-f} are important, see Lemma \ref{Lmm-ARU}. By Lemma \ref{Lmm-Linfty-L2}, the quantity $ \| z_{\alpha'} \sigma_x^ \frac{1}{2}   w_{\beta, \vartheta} g_\sigma \|_{L^\infty_x L^2_v} $ is thereby dominated by $ \mathscr{E}^A_{\mathtt{cro}} (g_\sigma) $. It is actually the Sobolev type interpolation in one dimensional space with different weights. Due to the structure of the equation, the singular weight $z_{- \alpha} (v) $ is unavoidable.
    	
    By Lemma \ref{Lmm-L2xv-alpha}, the quantity $ \mathscr{E}^A_{\mathtt{cro}} (g_\sigma) $ can be bounded by $ \| (\delta x + l)^{ - \frac{ \Theta }{ 2 } } \nu^\frac{1}{2} w_{ \beta + \beta_\gamma  , \vartheta } g_\sigma \|_A $, which can be further dominated by $ \mathscr{E}^A_{2} (w_{\beta+\beta_\gamma,\vartheta} g_\sigma) $, see \eqref{EA2}. The main goal of this step is to control the singular weight $z_{- \alpha} (v)$ involved in $ \mathscr{E}^A_{\mathtt{cro}} (g_\sigma) $. The difficult is to deal with the singular weight $z_{- \alpha} (v)$ associated with the operator $K$. Thanks to Lemma \ref{Lmm-Kh-L2}, the weight $z_{- \alpha} (v)$ is successfully removed. Note that the quantity $ \mathscr{E}^A_{\mathtt{cro}} (g_\sigma) $ is considered in the $L^2_{x,v}$ framework. The quantity $  \mathscr{E}^A_{2 } (w_{\beta+\beta_\gamma,\vartheta} g_\sigma) $ can be controlled by $\|( \delta x + l)^{-\frac{\Theta}{2}} \P w_{\beta+\beta_\gamma,\vartheta}g_\sigma \|_A^2$, see Lemma \ref{Lmm-L2xv}. Moreover, thanks to Lemma \ref{Lmm-P-wg}, $\|( \delta x + l)^{-\frac{\Theta}{2}} \P w_{\beta+\beta_\gamma,\vartheta}g_\sigma \|_A^2$ can be controlled by $\mathscr{E}^A_2(g_\sigma)$. We finally complete the estimates by standard $L^2_{x,v}$ estimates. Therefore, we obtain the uniform a priori estimates for \eqref{A1} in Lemma \ref{Lmm-APE-A3}.

    {\bf (IV) Advantage of incoming data approximation.} While investigating the problem \eqref{KL-Damped} over half-space $x \in (0, \infty)$, we will consider the approximate problem in a finite slab $x \in (0, A)$. Then the boundary condition at $x = A$ should be imposed. In \cite{Golse-Perthame-Sulem-1988-ARMA,HJW-2024-preprint,HJW-2023-AMASES,HW-2022-SIMA} about the hard sphere model ($\gamma = 1$) with Maxwell reflection type boundary conditions, the specular reflection boundary condition $g (A, v) |_{v_3 < 0} = g (A, R_A v)$ at $x = A$ was imposed.
    In this paper, we employ the incoming data approximation at $x = A$ as in \cite{JL-2024-arXiv}, i.e., imposing the boundary condition $g (A, v) |_{v_3 < 0} = \varphi_A (v)$, where $\varphi_A (v)$ is any fixed function with $\mathscr{B} ( e^{ \hbar \sigma (A, \cdot) } \varphi_A ) < \infty$ which means $\varphi_A (v) \to 0$ as $A \to + \infty$. The {\em rationality} of the incoming data approximation at $x = A$ can be illustrated that the far field condition $ \lim_{x \to + \infty} g (x,v) = 0 $ is closer to the incoming data approximation $ g (A, v) |_{v_3 < 0} = \varphi_A (v) $ with $\varphi_A (v) \to 0$ as $A \to + \infty$ than the specular reflection approximation $ g (A, v) |_{v_3 < 0} = g_A (R_A, v) $. Actually, we will take $\varphi_A (v) \equiv 0$. But we still employ the general $\varphi_A (v)$ which will be used to study the uniqueness of the problem \eqref{KL}. Moreover, the {\em advantage} of the incoming data approximation at $x = A$ is to simplify the proof process relative to the works \cite{HJW-2024-preprint,HJW-2023-AMASES,HW-2022-SIMA}.

    {\bf (V) Existence of the equations \eqref{KL} and \eqref{KL-NL}.} Based on the uniform a priori estimates on the approximate problem \eqref{A1} in Lemma \ref{Lmm-APE-A3}, we first outline the proof of existence and uniqueness to the linear problem \eqref{KL}.

    The first step is to establish the existence and uniqueness of the approximate problem \eqref{A1}. The existence of the weak solution of \eqref{A1} follows from the Hahn-Banach theorem and the Riesz representation theorem, similar to the corresponding part in \cite{Ukai-Yang-Yu-2003-CMP,Chen-Liu-Yang-2004-AA,Wang-Yang-Yang-2007-JMP}. Subsequently, the mild solution exists due to the a priori estimates in Lemma \ref{Lmm-APE-A3} and the uniqueness of the weak solution. Thus, the existence and uniqueness of the approximate problem \eqref{A1} holds, see Lemma \ref{Lmm-A1}.

    The second step is to justify the existence of \eqref{KL-Damped} by using Lemma \ref{Lmm-A1}. By taking $\varphi_A (v) = 0$ and assuming $ \mathscr{A}^\infty ( e^{\hbar \sigma} h ) < \infty $, the solution $g^A (x,v)$ constructed in Lemma \ref{Lmm-A1} obeys the uniform-in-$A$ bound $\mathscr{E}^A ( e^{ \hbar \sigma } g^A ) \lesssim \mathscr{A}^\infty ( e^{\hbar \sigma} h )$. Note that $g^A$ is defined on $(x,v) \in (0,A) \times \R^3$. We extend $g^A (x,v)$ to $(x,v) \in \R_+ \times \R^3$ as $\tilde{g}^A (x,v) = \mathbf{1}_{x \in  (0, A)} g^A (x,v)$, which, together the bound of $g^A$, is uniformly bounded in the space $\mathbb{B}^\hbar_\infty$ defined in \eqref{Bh-1}. Moreover, we can prove that $\tilde{g}^A (x,v)$ is a Cauchy sequence in $\mathbb{B}^{\hbar'}_\infty$ for any fixed $\hbar' \in (0, \hbar)$. Then we can show that the limit of $ \tilde{g}^A (x,v) $ is the unique solution to \eqref{KL-Damped}.

	The third step establishes the existence and uniqueness of the linear Knudsen layer equation \eqref{KL}. We begin by decomposing \eqref{KL} as follows:
	\begin{equation}\label{Methodology-f}
		\left\{	
		\begin{aligned}
			& v_3 \partial_x (I- \P^0 ) f + \L (I- \P^0 )f = ( I - \mathbb{P} ) S   \,,\\
			& \lim_{x \to + \infty} (I- \P^0 ) f (x,v) = 0 \,,
		\end{aligned}
		\right.
	\end{equation}
    and
    \begin{equation*}
	\left\{	
	\begin{aligned}
		& v_3  \P^0  f =  \mathbb{P}  S   \,,\\
		& \lim_{x \to + \infty} \P^0 f (x,v) = 0 \,.
	\end{aligned}
	\right.
    \end{equation*}
	Note that for the nondegenerate case $\mathcal{M}^\infty \neq 0\,, \pm 1$, this decomposition is unnecessary, since  $\P^0 f = \mathbb{P}S \equiv 0$ for any $f$ and $S$. We can explicitly write $\P^0 f$ as
    \begin{equation*}
	\P^0 f_2 (x,v)= -\frac{1}{v_3}\int_{x}^{\infty} \mathbb{P} S (y , \cdot) \d y =-\sum_{j\in I^0} \psi_j \int_{x}^{\infty} \frac{ (X_j, S(y,\cdot))}{ (X_j, \L X_j)} \d y   \,.
    \end{equation*}
	To solve \eqref{Methodology-f}, we further consider the following system
	\begin{equation}\label{Methodology-f1}
		\left\{
			\begin{aligned}
				& v_3 \partial_x  f_1 + \L  f_1 =  (I-\mathbb{P}) S -\bar{\alpha}(\delta x + l)^{-\Theta} \P^+ v_3 f_1- \bar{\beta} (\delta x + l)^{-\Theta} \P^0  f_1  \,, \ x > 0 \,, v \in \R^3 \,, \\
				& f_1 (0, v) |_{v_3 > 0} = f_b (v) -\P^0 f(0,v)  \,, \\
				& \lim_{x \to + \infty} f_1 (x, v) = 0 \,.
			\end{aligned}
		\right.
	\end{equation}
	By Lemma \ref{Lmm-KLe}, there exists a unique solution to \eqref{Methodology-f1}. Furthermore, the solution of \eqref{Methodology-f1} is also a solution to \eqref{Methodology-f} if and only if $\P^+ v_3 f_1= \P^0  f_1 \equiv 0$, which provides the solvability conditions for \eqref{Methodology-f}. The number of solvability conditions depends on the Mach number $\mathcal{M}^\infty$, see Table \ref{table1}. Consequently, the result on Theorem \ref{Thm-Linear} is thereby obtained.

    Finally, we focus on the nonlinear problem \eqref{KL-NL}, which is equivalently represented by \eqref{KL-NL-f}. As in the linear case, we begin by decomposing \eqref{KL-NL-f} as follows:
	\begin{equation}\label{Methodology-f2}
		\left\{
			\begin{aligned}
				& v_3 \partial_x (I-\P^0) f + \L (I-\P^0)f= (I-\mathbb{P}) \Gamma ( f, f ) + (I-\mathbb{P}) \tilde{h} \,, \ x > 0 \,, v \in \R^3 \,, \\
				& v_3 \partial_x \P^0 f = \mathbb{P} \Gamma ( f, f ) + \mathbb{P} \tilde{h} \,, \ x > 0 \,, v \in \R^3 \,, \\
				&f(0,v) \mid_{v_3>0}=\tilde{f}_b(v) \,,\\
				& \lim_{x \to + \infty} f (x, v) = 0 \,.
			\end{aligned}
		\right.
	\end{equation}
	Next, we develop the iterative scheme \eqref{Iter-fi}. The limit of \eqref{Iter-fi} exactly solves the nonlinear problem \eqref{KL-NL-f} under some solvability conditions, which depends on the Mach number $\mathcal{M}^\infty$. The key point of studying the nonlinear problem is to obtain nonlinear estimate
    \begin{equation*}
    	\begin{aligned}
    		\mathscr{A}^\infty ( e^{ \hbar \sigma } \Gamma ( f, g ) ) &\lesssim \mathscr{E}^\infty ( e^{ \hbar \sigma } f ) \mathscr{E}^\infty ( e^{ \hbar \sigma } g ) \,,\\
			\mathscr{E}^\infty_\infty ( e^{ \hbar \sigma } \Gamma ( f, g ) ) &\lesssim e^{- \hbar^2 (\delta x + l)^{\frac{2}{3-\gamma}}} \mathscr{E}^\infty_\infty ( e^{ \hbar \sigma } f ) \mathscr{E}^\infty_\infty ( e^{ \hbar \sigma } g ) \,
    	\end{aligned}
    \end{equation*}
    given in Lemma \ref{Lmm-Gamma}. The result on Theorem \ref{Thm-Nonlinear} is therefore constructed.

\subsection{Historical remarks}
For the incoming data boundary conditions with angular cutoff collisional kernel cases, there have been many results. Bardos-Caflisch-Nicolaenko \cite{Bardos-Caflisch-Nicolaenko-1986-CPAM} proved the exponential decay $e^{- c x}$ in the space $ L^\infty ( e^{c x} \d x; L^2 ( (1 + |v|)^\frac{1}{2} \d v ) ) \cap L^\infty ( \d x ; L^2 (\d v) ) $ for hard sphere model $\gamma = 1$, and Golse-Poupaud \cite{Golse-Poupaud-1989-MMAS} then proved superalgebraic $O (x^{ - \infty })$ in $L^\infty (\d x; L^2 ( |v_3| \d v ))$ space for $- 2 \leq \gamma < 1$. Both of them considered the case that the Mach number of the far Maxwellian $\mathcal{M}^\infty=0$. Coron-Golse-Sulem \cite{Coron-Golse-Sulem-1988-CPAM} studied the exponential decay $ e^{ - cx } $ in $L^\infty ( e^{cx} \d x; L^2 ( |v_3 + u| \d v ) )$ space for $\gamma = 1$, for all Mach numbers.
More recently, Ukai-Yang-Yu \cite{Ukai-Yang-Yu-2003-CMP} investigated the exponential decay $e^{- cx}$ in $x$ and algebraic decay $(1 + |v|)^{- \beta}$ in $v$ in $L^\infty_{x,v}$ space for $\gamma = 1$ with $\mathcal{M}^\infty \neq 0 \,,\pm 1$.
 Later, Golse \cite{Golse-2008-BIMA} proved the same result for the degenerate case, i.e. $\mathcal{M}^\infty = 0 \,,\pm 1$. Chen-Liu-Yang \cite{Chen-Liu-Yang-2004-AA} justified the exponential decay $e^{ - c x^\frac{2}{3 - \gamma} }$ in $x$ and algebraic decay $(1 + |v|)^{- \beta}$ in $v$ in $L^\infty_{x,v}$ space for $0 < \gamma \leq 1$ and nondegenerate case. Wang-Yang-Yang \cite{Wang-Yang-Yang-2007-JMP} proved the same decay results of \cite{Chen-Liu-Yang-2004-AA} for $- 2 < \gamma \leq 0$ and nondegenerate case. Yang \cite{Yang-2011-JSP} verified the superalgebraic decay $O(x^{-\infty})$ and $O(|v|^{-\infty})$ in $L^\infty_{x,v}$ for $- 3 < \gamma \leq 1$, both degenerate and nondegenerate cases. There also were some nonlinear stability results on the Knudsen boundary layer equation with incoming data, see \cite{Ukai-Yang-Yu-2004-CMP,Wang-Yang-Yang-2006-JMP,Yang-2008-JMAA}.
	
	We also mention the related work for the Maxwell reflection boundary condition, such as \cite{Golse-Perthame-Sulem-1988-ARMA,Coron-Golse-Sulem-1988-CPAM,HW-2022-SIMA,HJW-2023-AMASES,HJW-2024-preprint,JL-2024-arXiv}. Specifically, Jiang-Luo-Wu \cite{JL-2024-arXiv} proved the exponential decay $e^{ - c x^\frac{2}{3 - \gamma} -c|v|^2}$ for some $c>0$ in $L^\infty_{x,v}$ framework with general Maxwell reflection boundary condition with all accommodation coefficients in $[0,1]$, and angular cutoff hard-soft potentials ($-3< \gamma \leq 1$). This paper considers the same decay with the incoming boundary condition and full ranges of the far field Mach number.

\subsection{Incoming and Maxwell reflection boundary conditions} In the final part of this Introduction, we make some comments on the two widely-used types of the boundary conditions of the Knudsen boundary layer equations: incoming and Maxwell reflection conditions. Physically, these two types of boundary conditions are quite different. A typical model for the incoming boundary condition is the evaporation and condensation of the liquid-gas interface (See Sone's books \cite{Sone-Book-2002, Sone-Book-2007}). Simply saying, the boundary conditions of the fluid equations derived from Boltzmann equations with these two conditions are completely different. For example, in the derivation of acoustic system, the boundary condition from Maxwell reflection condition is the well-known {\em impermeable} condition:
\begin{equation*}
  \mathrm{u}(t,x)\cdot\mathrm{n}(x)=0\,,\quad \mbox{on}\quad\! \partial\Omega\,.
\end{equation*}
where $\mathrm{n}(x)$ is the outer normal vector at $x\in \partial\Omega$ on the boundary. However, from the Boltzmann equation with incoming condition, the boundary of the acoustic system is the following {\em pressure jump boundary condition} (\cite{Aoki-Golse-Jiang, Jiang-Wu})
\begin{equation*}
  \rho(t,x)+\theta(t,x)=\Lambda (\mathrm{u}(t,x)\cdot \mathrm{n}(t,x)) \quad \text{ on } \quad\! \partial\Omega.
\end{equation*}
Here the constant $\Lambda>0$ is named as {\em pressure jump coefficient}. It comes from solving the Knudsen boundary layer equation. More interestingly, a classic result in the constant coefficient symmetric hyperbolic system tells us that there are only types of boundary conditions such that  the acoustic system is well-posed: impermeable and pressure jump boundary conditions. In this sense, incoming and Maxwell reflection conditions at kinetic level together provide the complete and complementary boundary conditions at the fluid levels.  This example illustrates that the the Boltzmann equation and its corresponding Knudsen layer boundary layer equations with incoming and Maxwell reflection boundaries deserve studies separately. Both of them have the interests of their own to be investigated. For other more complicated nonlinear fluid equations, such as compressible Euler and Navier-Stokes equations, the boundary conditions from Boltzmann equations with incoming and Maxwell reflection conditions are also completely different and complementary \cite{Jiang-Wu-2}.

\subsection{Organization of current paper}

In the next section, we present some preliminaries that will be frequently used later. Section \ref{Sec:UECA} focuses on deriving the uniform estimates for the approximate system \eqref{A1}. In Section \ref{Sec:EKLe}, we establish the existence and uniqueness of the linear problem \eqref{KL}, as stated in Theorem \ref{Thm-Linear}. Section \ref{Sec:ENL} addresses the existence and uniqueness of the nonlinear problem \eqref{KL-NL}, leading to Theorem \ref{Thm-Nonlinear}. Finally, Section \ref{Sec:YZU-proof} justifies the  $L^\infty_{x,v}$ bounds for the operators $Y_A, Z, U$, i.e., proving Lemma \ref{Lmm-ARU}.

\section{Preliminaries}\label{Sec:Prel}

\subsection{Properties of the $(x,v)$-mixed weights $\sigma (x,v)$}

From the works \cite{Chen-Liu-Yang-2004-AA,Wang-Yang-Yang-2006-JMP,Wang-Yang-Yang-2007-JMP,Yang-2008-JMAA}, the weihgt $\sigma (x,v)$ admits the following properties.

\begin{lemma}\label{Lmm-sigma}
	For $-3 < \gamma \leq 1$, some large constant $l > 0$ and small constant $\delta > 0$, there are some positive constants $c, c_1, c_2$ such that
	\begin{equation*}
		\begin{aligned}
			& \sigma (x,v) \geq c (\delta x + l)^\frac{2}{3 - \gamma} \,, \ |\sigma (x,v) - \sigma (x, v_*)| \leq c \big| |v - \u|^2 - |v_* - \u|^2 \big| \,, \\
			& 0 < c_1 \min \big\{ (\delta x + l)^{- \Theta} \,, (1 + |v - \u|)^{- 1 + \gamma} \big\} \leq \sigma_x (x,v) \leq c_2 (\delta x + l)^{- \Theta} \leq c_2 l^{- \Theta} \,, \\
			& |\sigma_x (x,v) v_3 | \leq c \nu (v) \,, \ |\sigma_{xx} (x, v) v_3| \leq \delta \sigma_x \nu (v) \,.
		\end{aligned}
	\end{equation*}
\end{lemma}

\subsection{The operators $Y_A , Z$ and $U$}

We introduce the function
\begin{equation}\label{kappa}
		 \kappa (x,v) = \int_0^x \big[ - \tfrac{ \sigma_{xx} (y,v) }{2 \sigma_x (y,v)} - \hbar \sigma_x (y, v) + \tfrac{\nu (v)}{v_3} \big] \d y \,.
\end{equation}
We then define the following linear operators
\begin{equation}\label{YAn-f}
	Y_A (f) =
	\left\{
	    \begin{aligned}
	    	0 \,, \quad & \textrm{if } v_3 > 0 \,, \\[10pt]
	    	e^{\kappa (A, v) - \kappa (x,v)} f (A, v) \,, \quad & \textrm{if } v_3 < 0 \,,
	    \end{aligned}
	\right.
\end{equation}
and
\begin{equation}\label{Rn-f}
	Z (f)=\left\{
	    \begin{aligned}
	    	e^{-\kappa(x,v)} f(0,v) \,, \quad & \textrm{if } v_3 > 0 \,, \\[10pt]
	    	0 \,, \quad & \textrm{if } v_3 < 0 \,,
	    \end{aligned}
	\right.
\end{equation}
and
\begin{equation}\label{U-f}
	U (f) =
	\left\{
	    \begin{aligned}
	    	\int_0^x e^{- [ \kappa (x,v) - \kappa (x', v) ]} \tfrac{1}{v_3} f (x' , v) \d x' \,, \quad & \textrm{if  } v_3 > 0 \,, \\
	    	- \int_x^A  e^{- [ \kappa (x,v) - \kappa (x', v) ]} \tfrac{1}{v_3} f (x' , v) \d x' \,, \quad & \textrm{if  } v_3 < 0 \,.
	    \end{aligned}
	\right.
\end{equation}
These operators will play an essential role in estimating the weighted $L^\infty_{x,v}$ norms of the Knudsen layer equations \eqref{KL-Damped}. More precisely, they obey the following estimates.

\begin{lemma}\label{Lmm-ARU}
	Let $- 3 < \gamma \leq 1$, $l \geq 1$, $\beta \in \mathbb{R}$, $A > 0$, $\vartheta > 0$. Then, for sufficiently small $\hbar, \delta > 0$, there are some positive constant $C > 0$, independent of $l$, $\delta$, $\hbar$ and $A$, such that
	\begin{enumerate}
    \item $Y_A (f)$ subjects to the estimate
    \begin{equation}\label{YAn-bnd}
    	\begin{aligned}
    		\| Y_A (f) \|_{L^\infty_{x,v}} \leq C \| f (A, \cdot) \|_{L^\infty_v} \,;
    	\end{aligned}
    \end{equation}
    \item $Z (f)$ enjoys to the estimate
    \begin{equation}\label{Z-bnd}
    	\begin{aligned}
    		\| Z (f) \|_{L^\infty_{v}} \leq C \| f (0, \cdot) \|_{L^\infty_v} \,;
    	\end{aligned}
    \end{equation}
    \item $U (f)$ obeys the bound
    \begin{equation}\label{U-bnd}
    	\begin{aligned}
    		& \| U (f) \|_{L^\infty_{x,v}} \leq C \| \nu^{-1} f \|_{L^\infty_{x,v}} \,.
    	\end{aligned}
    \end{equation}
    \end{enumerate}
\end{lemma}

The proof of Lemma \ref{Lmm-ARU} will be given in Section \ref{Sec:YZU-proof} later.

\subsection{Properties of the operator $K$}

In this subsection, the goal is to derive some useful properties of the operator $K$ defined in \eqref{K-K1-K2}-\eqref{K1}-\eqref{K2}. It then is focused on the following decay results of the operator $K$.

\begin{lemma}\label{Lmm-K-Oprt}
	Let $- 3 < \gamma \leq 1$, $A > 0$, $ \beta \in \R$, $\hbar, \vartheta \geq 0$ with $\delta_0 : = \tfrac{1}{4} - T (c \hbar + \vartheta) > 0 $, where $c > 0$ is given in Lemma \ref{Lmm-sigma}. Then there is a constant $C > 0$, independent of $A, \hbar$, such that
	\begin{equation}\label{K-bnd}
		\begin{aligned}
			\| \sigma_x^ \frac{1}{2} e^{\hbar \sigma} K f \|_{A; \beta, \vartheta} \leq C \| \sigma_x^ \frac{1}{2} e^{\hbar \sigma} f \|_{A; \beta + \gamma - 1, \vartheta} \,.
		\end{aligned}
	\end{equation}
    where the norm $\| \cdot \|_{A; \beta, \vartheta}$ is defined in \eqref{XY-space}.
\end{lemma}
\begin{proof}
	See \cite{JL-2024-arXiv} Section  8.1 for the proof.
\end{proof}

Next we prove the boundedness of the operator $K_\hbar$ from $Y_{  0, \vartheta}^\infty (A) \cap L^\infty_x L^2_v$ to $Y_{  - \gamma, \vartheta}^\infty (A)$. Here the operator $K_\hbar$ is defined as
\begin{equation}\label{Kh}
	\begin{aligned}
		K_\hbar g_\sigma = e^{\hbar \sigma} K ( e^{- \hbar \sigma} g_\sigma ) \,.
	\end{aligned}
\end{equation}
More precisely, the following results hold.

\begin{lemma}\label{Lmm-Kh-2-infty}
	Let $- 3 < \gamma \leq 1$, $A > 0 $, $m \in \R$, $0 \leq \alpha' < \frac{1}{2}$ and $\hbar, \vartheta \geq 0$ sufficiently small. Then for any $\eta_1 > 0$, there is a $C_{\eta_1} > 0$, independent of $A, \hbar$, such that
	\begin{equation}
		\begin{aligned}
			\LL \nu^{-1} K_\hbar g \RR_{A;   0, \vartheta} \leq \eta_1 \LL g \RR_{A;   0, \vartheta} + C_{\eta_1} \| \sigma_x^ \frac{1}{2}   z_{\alpha'} w_{- \gamma, \vartheta} g \|_{L^\infty_x L^2_v} \,.
		\end{aligned}
	\end{equation}
    Here the weight $z_{\alpha'}$ is given in \eqref{z-alpha}.
\end{lemma}
\begin{proof}
	See \cite{JL-2024-arXiv} Section  8.2 for the proof.
\end{proof}

Next we show the boundedness of the operator $K_\hbar$ in the weighted $L^2_v$ space. More precisely, the following conclusions hold.

\begin{lemma}\label{Lmm-Kh-L2}
	Let $- 3 < \gamma \leq 1$, $\beta \in \R$ and $\hbar, \vartheta \geq 0$ be sufficiently small. Define two nonempty sets
	\begin{equation*}
		\begin{aligned}
			& \mathcal{S}_\gamma : = \{ b_0 \in \R | b_0 < 2 \,, 0 \leq b_0 \leq 1 - \gamma \,, b_0 + \gamma + 1 > 0 \} \,, \\
			& \mathcal{T}_\gamma : = \{ b_1 \in \R | b_1 < 3 \,, 0 \leq b_1 \leq 1 - \gamma \,, b_1 + \gamma > 0 \} \,.
		\end{aligned}
	\end{equation*}
	Denote by $\mu_\gamma : = \min \{ \frac{1}{2}, \frac{\gamma + 3}{2}, \frac{b_0 + \gamma + 1}{2} , \frac{b_1 + \gamma}{2} \} > 0$ with $b_0 \in \mathcal{S}_\gamma$ and $b_1 \in \mathcal{T}_\gamma$. Let $0 \leq \alpha < \mu_\gamma$. Then there is a positive constant $C > 0$, independent of $\hbar$, such that
	\begin{equation}\label{Kh-L2-Bnd}
		\begin{aligned}
			\int_{\R^3} |\nu^{- \frac{1}{2}} z_{- \alpha} \sigma_x^ \frac{1}{2} w_{\beta, \vartheta} K_\hbar g (x,v)|^2 \d v \leq C \int_{\R^3} |\nu^\frac{1}{2} \sigma_x^ \frac{1}{2} w_{\beta, \vartheta} g (x,v)|^2 \d v \,.
		\end{aligned}
	\end{equation}
\end{lemma}
\begin{proof}
	See \cite{JL-2024-arXiv} Section  8.3 for the proof.
\end{proof}

%%%%%%%%%%%%%%%%%%%%%%%%%%%%%%%%%%%%%%%%%%%%%%%%%%%%%%%%%%%%%%%%%%%%%%%%%%%%%%%%%%%%%%%%

\section{Uniform estimates for the approximate problem \eqref{A1}}\label{Sec:UECA}

In this section, we primarily derive the a priori estimates for the approximate problem \eqref{A1}, specifically,
\begin{equation*}
	\begin{aligned}
		\left\{	
		\begin{aligned}
			& v_3 \partial_x g + \nu (v) g =  K g + h-\bar{\alpha}(\delta x + l)^{-{\Theta}} \P^+ v_3 g-\bar{\beta} (\delta x + l)^{-{\Theta}}\P^0 g \,, \\
			& g (0,v) |_{v_3 > 0} = f_b \,, \\
			& g (A, v) |_{v_3 < 0} = \varphi_A (v) \,.
		\end{aligned}
		\right.
	\end{aligned}
\end{equation*}
uniformly in the parameter $A > 1$ on the functional $\mathscr{E}^A$ given in \eqref{Eg-lambda}. This step is crucial for our current work.

\subsection{Statements of the uniform results}

We define an operator $\mathscr{L}$ as follows:
\begin{equation}\label{L-lambda}
	\left\{	
	\begin{aligned}
		& \mathscr{L} g : = v_3 \partial_x g + \nu (v) g -  K g = h -\bar{\alpha}(\delta x+l)^{-\Theta}\P^+(v_3 g)-\bar{\beta}(\delta x+l)^{-\Theta}\P^0 g \,, \\
		& g (0,v) |_{v_3 > 0} = f_b\,, \\
		& g (A, v) |_{v_3 < 0} = \varphi_A (v) \,.
	\end{aligned}
	\right.
\end{equation}
For notation convenience, let $g = \mathscr{L}^{-1} (h)$ be the solution to the problem \eqref{L-lambda}. First, we give the following a priori estimates for the operator $\mathscr{L}^{-1} (h)$.

Let $g_\sigma = e^{\hbar \sigma} g$. \eqref{L-lambda} can be equivalently expressed as
\begin{equation}\label{A3-lambda}
	\left\{
	\begin{aligned}
		& \mathscr{L}_{\hbar} g_\sigma : = v_3 \partial_x g_\sigma + [ - \hbar \sigma_x v_3 + \nu (v) ] g_\sigma -  K_\hbar g_\sigma = h_\sigma-D_\hbar g \,, \\
		& g_\sigma (0, v) |_{v_3 > 0} = f_{b, \sigma} (v) \,, \\
		& g_\sigma (A, v) |_{v_3 < 0} = \varphi_{A, \sigma} (v) \,,
	\end{aligned}
	\right.
\end{equation}
where $D_\hbar$ is defined as
\begin{equation}\label{D-def}
	D_\hbar g =\bar{\alpha}(\delta x+l)^{-\Theta}e^{\hbar \sigma}\P^+(v_3 e^{-\hbar \sigma} g_\sigma)+\bar{\beta}(\delta x+l)^{-\Theta}e^{\hbar \sigma}\P^0 (e^{-\hbar \sigma} g_\sigma)\,,
\end{equation}
and $K_\hbar$ is defined in \eqref{Kh}. Note that
\begin{equation*}
		\partial_x (\sigma_x^\frac{1}{2} g_\sigma ) + \partial_x \kappa (x,v)  (\sigma_x^ \frac{1}{2}   g_\sigma )
	 = \tfrac{1}{v_3} \underbrace{ (  \sigma_x^ \frac{1}{2}   K_\hbar g_\sigma + \sigma_x^ \frac{1}{2}   h_\sigma -\sigma_x^ \frac{1}{2} D_\hbar g ) }_{: = H (x,v)} \,.
\end{equation*}
This implies that
\begin{equation}\label{A3-1}
	\begin{aligned}
		\partial_x \big[ e^{\kappa (x,v)} \sigma_x^ \frac{1}{2}   g_\sigma (x,v) \big] = e^{\kappa (x,v)} \tfrac{1}{v_3} H (x,v) \,.
	\end{aligned}
\end{equation}

If $v_3 < 0$, integrating \eqref{A3-1} from $x$ to $A$ yields that
\begin{equation}\label{v3-n}
	\begin{aligned}
		& \sigma_x^ \frac{1}{2}   g_\sigma (x,v) \\
		= & e^{\kappa (A, v) - \kappa (x,v)} \sigma_x^ \frac{1}{2} (A, v)     \varphi_{A, \sigma} (v) - \int_x^A e^{- [\kappa (x,v) - \kappa (x', v)]} \tfrac{1}{v_3} H (x', v) \d x' \,.
	\end{aligned}
\end{equation}
If $v_3 > 0$, integrating \eqref{A3-1} from 0 to $x$ and together with the incoming boundary condition in \eqref{A3-lambda}, we obtain
\begin{equation}\label{v3-p-1}
		\sigma_x^ \frac{1}{2}   g_\sigma (x,v)
		=  e^{- \kappa (x,v)} \sigma_x^ \frac{1}{2} (0, v)  f_{b,\sigma} ( v) + \int_0^x e^{- [\kappa (x,v) - \kappa (x', v)]} \tfrac{1}{v_3} H (x', v) \d x' \,.
\end{equation}
Summarily, the equations \eqref{v3-n} and \eqref{v3-p-1} indicate that
	\begin{align}\label{A3-3}
		 \sigma_x^ \frac{1}{2}   g_\sigma = & Y_A (\sigma_x^ \frac{1}{2} (A, v) \varphi_{A, \sigma} (v),)+ Z(\sigma_x^ \frac{1}{2} (0, v) f_{b,\sigma}(v))  + U ( \sigma_x^ \frac{1}{2}   K_\hbar g_\sigma + \sigma_x^ \frac{1}{2}   h_\sigma-\sigma_x^ \frac{1}{2} D_\hbar g) \,,
	\end{align}
where the operators $Y_A (\cdot)$, $Z (\cdot)$ and $U (\cdot)$ are introduced in \eqref{YAn-f}, \eqref{Rn-f} and \eqref{U-f}, respectively.

We can then establish the uniform a priori estimates for the problem \eqref{A1} (or equivalently \eqref{A3-lambda}) in the following lemma.

\begin{lemma}[Uniform a priori estimates for \eqref{A3-lambda}]\label{Lmm-APE-A3}
	Let $- 3 < \gamma \leq 1$, sufficiently small $\delta, \hbar, \vartheta > 0$, large enough $A,\, l > 1$, integer $\beta \geq \max \{ 0, - \gamma \}$ and $0 < \alpha < \mu_\gamma$, where $\mu_\gamma > 0$ is given in Lemma \ref{Lmm-Kh-L2}. Assume that the source term $h_\sigma$ and the boundary source terms $\varphi_{A, \sigma}$ and $f_{b,\sigma}$ of the system \eqref{A3-lambda} satisfy
	\begin{equation}\label{Assmp-hsigma}
		\begin{aligned}
			\mathscr{A}^A (h_\sigma) \,, \mathscr{B} ( \varphi_{A, \sigma} )  \,, \mathscr{C} ( f_{b, \sigma} ) < \infty\,.
		\end{aligned}
	\end{equation}
	Let $ g_\sigma = g_\sigma (x,v) $ be the solution to \eqref{A3-lambda}. Then there exists a constant $C > 0$, independent of $A$ and $\delta$, but dependent on $\hbar$, such that $g_\sigma$ obeying the following bounds:
	\begin{equation}\label{Apriori-bnd}
		\begin{aligned}
			\mathscr{E}^A \big( g_\sigma \big) \leq C \big( \mathscr{A}^A (h_\sigma) + \mathscr{B} ( \varphi_{A, \sigma} ) + \mathscr{C} ( f_{b, \sigma} )  \big) \,.
		\end{aligned}
	\end{equation}
	Moreover, let $g_{\sigma i} = g_{\sigma i} (x,v) $ be the solutions corresponding to the source terms $h_{\sigma i }$ and $f_{b,\sigma i}$ $(i = 1,2)$, where $ \mathscr{A}^A (h_{\sigma i}) < \infty $ and $ \mathscr{C} ( f_{b, \sigma i} ) < \infty $  for $ i = 1,2 $. Then $g_{\sigma 2 } - g_{\sigma 1 } $ satisfies
	\begin{equation}\label{Apriori-bnd-diff}
		\begin{aligned}
			\mathscr{E}^A \big( g_{\sigma 2 } - g_{\sigma 1 } \big) \leq C \big(\mathscr{A}^A (h_{\sigma 2 } - h_{\sigma 1 } ) + \mathscr{C} ( f_{b, \sigma 2}-f_{b, \sigma 1} ) \big) \,.
		\end{aligned}
	\end{equation}
	Here the functionals $\mathscr{E}^A (\cdot)$, $\mathscr{A}^A (\cdot)$, $\mathscr{B} (\cdot)$ and $\mathscr{C} (\cdot)$ are defined in \eqref{Eg-lambda} and \eqref{Ah-lambda}.
\end{lemma}

The proof of Lemma \ref{Lmm-APE-A3} will be completed in Subsection \ref{Subsec:Uniform-Est} later.

\subsection{Weighted $L^\infty_{x,v}$ estimates in $Y^\infty_{  \beta, \vartheta} (A)$ space}

In this subsection, we aim to control the norm $\LL g_\sigma \RR_{A;  \beta, \vartheta}$ with respect to the space $Y^\infty_{\beta, \vartheta} (A)$. Moreover, we will also investigate the continuous dependence of $g_\sigma $ on the source term $h_\sigma$ and the boundary source term $f_{b,\sigma}$. Let $g_{\sigma i} = \mathscr{L}_{ \hbar}^{-1} (h_{\sigma i})$ ($i = 1,2$) for sufficiently small $\hbar \geq 0$. Then the difference ${\vartriangle} g_\sigma : = g_{\sigma 2} - g_{\sigma 1}$ satisfies
\begin{equation}\label{L-lambda-Dif}
	\left\{	
	\begin{aligned}
		& v_3 \partial_x ( {\vartriangle} g_\sigma ) + [ - \hbar \sigma_x v_3 + \nu (v) ] ( {\vartriangle} g_\sigma ) -  K_\hbar ( {\vartriangle} g_\sigma ) = {\vartriangle} h_\sigma - D_\hbar {\vartriangle} g_\sigma \,, \\
		& ( {\vartriangle} g_\sigma ) (0, v) |_{v_3 > 0} = {\vartriangle} f_{b,\sigma} \,, \\
		& ( {\vartriangle} g_\sigma ) (A, v) |_{v_3 < 0} = 0 \,,
	\end{aligned}
	\right.
\end{equation}
where ${\vartriangle} h_\sigma : = h_{\sigma 2} - h_{\sigma 1}$ and ${\vartriangle} f_{b,\sigma} : = f_{b, \sigma 2} - f_{b, \sigma 1}$.

More precisely, the following lemma holds.

\begin{lemma}\label{Lmm-Y-bnd}
	Let $- 3 < \gamma \leq 1$, $A > 1$, $0 \leq \frac{1}{2} - \mu_\gamma < \alpha' < \frac{1}{2}$, the integer $\beta \geq \max \{ 0, - \gamma \}$, where the constant $\mu_\gamma > 0$ is given in Lemma \ref{Lmm-Kh-L2}. Then for sufficiently small $\vartheta, \hbar \geq 0$, there is a constant $C > 0$, independent of $A$ and $\hbar$, such that
	\begin{equation}\label{LA-g-sigma}
		\begin{aligned}
			& \mathscr{E}_\infty^A ( g_\sigma ) + \LL g_\sigma \RR_{  \beta, \vartheta, \Sigma} \leq C \| z_{\alpha'} \sigma_x^ \frac{1}{2}   w_{\beta, \vartheta} g_\sigma \|_{L^\infty_x L^2_v} + C \big( \mathscr{A}^A_\infty (h_\sigma) + \mathscr{B}_\infty ( \varphi_{A, \sigma} )+ \mathscr{C}_\infty ( f_{b, \sigma} ) \big) \,,
		\end{aligned}
	\end{equation}
	where the functionals $ \mathscr{E}_\infty^A ( \cdot ) $, $ \mathscr{A}^A_\infty ( \cdot ) $, $ \mathscr{B}_\infty ( \cdot ) $ and $ \mathscr{C}_\infty ( \cdot ) $ are defined in \eqref{E-infty}, \eqref{As-def}, \eqref{B-def} and \eqref{C-def}, respectively. Moreover, ${\vartriangle} g_{\sigma} (x,v)$ satisfying \eqref{L-lambda-Dif} enjoys the bound
	\begin{equation}\label{LA-g-sigma-Dif}
		\begin{aligned}
			\mathscr{E}_\infty^A ( {\vartriangle} g_\sigma ) + \LL {\vartriangle} g_\sigma \RR_{  \beta, \vartheta, \Sigma} \leq C \| z_{\alpha'} \sigma_x^ \frac{1}{2}   w_{\beta, \vartheta} {\vartriangle} g_\sigma \|_{L^\infty_x L^2_v} + C \mathscr{A}^A_\infty ( {\vartriangle} h_\sigma ) + C \mathscr{C}_\infty ( {\vartriangle} f_{b, \sigma} ) \,.
		\end{aligned}
	\end{equation}
\end{lemma}

\begin{proof}
	By Lemma \ref{Lmm-ARU}, the equation \eqref{A3-3} shows
		\begin{align}\label{Y-bnd-0}
			\no  \LL g_\sigma \RR_{A;   \beta, \vartheta} + \LL g_\sigma \RR_{  \beta, \vartheta, \Sigma} = & \| \sigma_x^ \frac{1}{2}   w_{\beta, \vartheta} g_\sigma \|_{ L^\infty_{x,v} } + \| \sigma_x^ \frac{1}{2}   w_{\beta, \vartheta} g_\sigma \|_{L^\infty_\Sigma} \\
			\no \leq & \| w_{\beta, \vartheta} Y_A (\sigma_x^ \frac{1}{2} \varphi_{A, \sigma})\|_{L^\infty_{x,v}}  + \| w_{\beta, \vartheta} Z ( \sigma_x^ \frac{1}{2} f_{b,\sigma}) \|_{L^\infty_{x,v}} \\
			 \no &+  \| w_{\beta, \vartheta} U (  \sigma_x^ \frac{1}{2}   K_\hbar g_\sigma + \sigma_x^ \frac{1}{2}   h_\sigma - \sigma_x^ \frac{1}{2} D_\hbar g ) \|_{L^\infty_{x,v}} \\
			\leq & \mathscr{B}_\infty ( \varphi_{A, \sigma} )+ \mathscr{C}_\infty ( f_{b, \sigma} ) + C \| \nu^{-1} (  \sigma_x^ \frac{1}{2} K_\hbar g_\sigma + \sigma_x^ \frac{1}{2} h_\sigma -\sigma_x^ \frac{1}{2} D_\hbar g ) \|_{A; \beta, \vartheta} \\
			\no \leq & \mathscr{B}_\infty ( \varphi_{A, \sigma} )+ \mathscr{C}_\infty ( f_{b, \sigma} ) + C \big( \LL \nu^{-1} K_\hbar g_\sigma \RR_{A;   \beta, \vartheta} + \mathscr{A}^A_\infty (h_\sigma)+\mathscr{A}^A_\infty(D_\hbar g) \big) \,.
		\end{align}
	By Lemma \ref{Lmm-K-Oprt}, one has $ \LL \nu^{-1} K_\hbar g_\sigma \RR_{A;   \beta, \vartheta} \leq C \LL g_\sigma \RR_{A;   \beta - 1, \vartheta} $. It therefore infers that
	\begin{equation*}
		\begin{aligned}
			\LL g_\sigma \RR_{A;   \beta, \vartheta} + \LL g_\sigma \RR_{  \beta, \vartheta, \Sigma} \leq C \LL g_\sigma \RR_{A;   \beta - 1, \vartheta} + C \mathscr{A}^A_\infty (h_\sigma) + C \mathscr{A}^A_\infty(D_\hbar g) + C \mathscr{B}_\infty ( \varphi_{A, \sigma} )+ C \mathscr{C}_\infty ( f_{b, \sigma} ) \,.
		\end{aligned}
	\end{equation*}
	Inductively, for any given integer $\beta \geq \max \{ 0, - \gamma \}$, it follows that
	\begin{equation}\label{Y-bnd-1}
		\begin{aligned}
			\LL g_\sigma \RR_{A;   \beta, \vartheta} + \LL g_\sigma \RR_{  \beta, \vartheta, \Sigma} \leq C \LL g_\sigma \RR_{A;   0, \vartheta} + C \big( \mathscr{A}^A_\infty (h_\sigma) + \mathscr{A}^A_\infty(D_\hbar g)+  \mathscr{B}_\infty ( \varphi_{A, \sigma} ) +  \mathscr{C}_\infty ( f_{b, \sigma} ) \big) \,,
		\end{aligned}
	\end{equation}
	where we have utilized $ \LL \nu^{-1} h_\sigma \RR_{A;   i, \vartheta} \leq C \LL \nu^{-1} h_\sigma \RR_{A;   \beta, \vartheta} $ for $0 \leq i \leq \beta$.
	
	It remains to dominate the quantity $\LL g_\sigma \RR_{A;   0, \vartheta}$. Together with \eqref{A3-3}, the similar arguments in \eqref{Y-bnd-0} indicate that
	\begin{multline*}
			 \LL g_\sigma \RR_{A;   0, \vartheta} + \LL g_\sigma \RR_{  0, \vartheta, \Sigma}
			\leq  C \LL \nu^{-1} K_\hbar g_\sigma \RR_{A;   0, \vartheta} \\
			+ C \big( \LL \nu^{-1} h_\sigma  \RR_{A;   0, \vartheta} + \LL \nu^{-1} D_\hbar g \RR_{A;   0, \vartheta} + \| \sigma_x^ \frac{1}{2} (A, \cdot)  w_{0, \vartheta} \varphi_{A, \sigma} \|_{L^\infty_v} + \| \sigma_x^ \frac{1}{2} (0, \cdot)  w_{0, \vartheta} f_{b, \sigma} \|_{L^\infty_v} \big) \,.
	\end{multline*}
	Lemmas \ref{Lmm-Kh-2-infty} show that
	\begin{equation*}
		\begin{aligned}
			\LL \nu^{-1} K_\hbar g_\sigma \RR_{A;   0, \vartheta} \leq \eta_1 \LL g_\sigma \RR_{A;   0, \vartheta} + C_{\eta_1} \| z_{\alpha'} \sigma_x^ \frac{1}{2}   w_{- \gamma, \vartheta} g_\sigma \|_{L^\infty_x L^2_v}
		\end{aligned}
	\end{equation*}
	for any small $\eta_1 > 0$. Taking $\eta_1 > 0$ such that $C \eta_1 \leq \frac{1}{2}$, it follows that
	\begin{equation}\label{Y-bnd-2}
		\begin{aligned}
			& \LL g_\sigma \RR_{A;   0, \vartheta} + \LL g_\sigma \RR_{  0, \vartheta, \Sigma} \\
			\leq & C \| z_{\alpha'} \sigma_x^ \frac{1}{2}   w_{\beta, \vartheta} g_\sigma \|_{L^\infty_x L^2_v} + C \big( \mathscr{A}^A_\infty (h_\sigma) + \mathscr{A}^A_\infty (D_\hbar g) +  \mathscr{B}_\infty ( \varphi_{A, \sigma} ) +  \mathscr{C}_\infty ( f_{b, \sigma} ) \big)
		\end{aligned}
	\end{equation}
	for integer $\beta \geq \max \{ 0, - \gamma \}$ and $- 3 < \gamma \leq 1$.
Moreover, by Lemma 8.1 of \cite{JL-2024-arXiv}, we have
\begin{equation*}
	\frac{\sigma_x(x,v)}{\sigma_x(x,v_*)} \lesssim \max \Big\{1, \frac{(1+|v-\u|)^{-1+\gamma}}{(1+|v_*-\u|)^{-1+\gamma}}\Big\}.
\end{equation*}
Together with $\sigma(x,v)-\sigma(x,v_*) \leq | c \left||v-\u|^2-|v_*-\u |^2\right|$ from Lemma \ref{Lmm-sigma}, one establishes
	\begin{equation}\label{Y-bnd-3}
		\begin{aligned}
			\mathscr{A}^A_\infty (D_\hbar g)& \leq \bar{\alpha} \mathscr{A}^A_\infty ((\delta x + l)^{-\Theta} e^{\hbar \sigma}\P^+ v_3 e^{-\hbar \sigma}g_\sigma)
			+\bar{\beta} \mathscr{A}^A_\infty ((\delta x + l)^{-\Theta} e^{\hbar \sigma}\P^0 e^{-\hbar \sigma}g_\sigma)\\
			 &\leq C\hbar \mathscr{E}^A_\infty (g_\sigma).
		\end{aligned}
	\end{equation}
	Then \eqref{Y-bnd-1}, \eqref{Y-bnd-2} and \eqref{Y-bnd-3} imply the bound \eqref{LA-g-sigma} for integer $\beta \geq \max \{ 0, - \gamma \}$, $- 3 < \gamma \leq 1$, $0 \leq \frac{1}{2} - \mu_\gamma < \alpha' < \frac{1}{2}$ and sufficiently small $\vartheta, \hbar \geq 0$.
	
	By the virtue of the similar arguments in \eqref{LA-g-sigma}, ${\vartriangle} g_\sigma$ subjecting to the equation \eqref{L-lambda-Dif} satisfies the bound \eqref{LA-g-sigma-Dif}. Then the proof of Lemma \ref{Lmm-Y-bnd} is finished.
\end{proof}

\subsection{Estimate for $L^\infty_x L^2_v$ norm with weight $ z_{\alpha'} \sigma_x^ \frac{1}{2}   w_{\beta, \vartheta} $}

In this subsection, we will dominate the quantity $\| z_{\alpha'} \sigma_x^ \frac{1}{2}   w_{\beta, \vartheta} g_\sigma \|_{L^\infty_x L^2_v}$ appeared in the right hand side of \eqref{LA-g-sigma}, Lemma \ref{Lmm-Y-bnd}. The following lemma holds.

\begin{lemma}\label{Lmm-Linfty-L2}
	Let $- 3 < \gamma \leq 1$, $A > 1$, $ 0 \leq \hbar, \vartheta \ll 1 $, the integer $\beta \geq \max \{ 0, - \gamma \}$ and $\alpha = \frac{1}{2} - \alpha'$, where $\alpha'$ is given in Lemma \ref{Lmm-Y-bnd}. Then there is a constant $C > 0$, independent of $A$ and $\hbar$, such that
	\begin{equation}\label{Y-bnd-4}
		\begin{aligned}
			\| z_{\alpha'} \sigma_x^ \frac{1}{2}   w_{\beta, \vartheta} g_\sigma \|_{L^\infty_x L^2_v} \leq C \mathscr{E}_{\mathtt{cro}}^A (g_\sigma) \,,
		\end{aligned}
	\end{equation}
	where the functional $\mathscr{E}_{\mathtt{cro}}^A ( \cdot )$ is given in \eqref{E-cro}. Moreover, a similar result holds for the difference:
	\begin{equation}\label{Y-bnd-4-Dif}
		\begin{aligned}
			\| z_{\alpha'} \sigma_x^ \frac{1}{2}   w_{\beta, \vartheta} {\vartriangle} g_\sigma \|_{L^\infty_x L^2_v} \leq C \mathscr{E}_{\mathtt{cro}}^A ( {\vartriangle} g_\sigma ) \,.
		\end{aligned}
	\end{equation}
\end{lemma}

\begin{proof}
	Denote by
	\begin{equation*}
		\begin{aligned}
			& \phi (x) = \int_{\R^3} | z_{\alpha'} \sigma_x^ \frac{1}{2}   w_{\beta, \vartheta} g_\sigma (x,v)|^2 \d v \,, \bar{\phi} (x,v) = \tfrac{\partial}{\partial x} | z_{\alpha'} \sigma_x^ \frac{1}{2}   w_{\beta, \vartheta} g_\sigma |^2 (x, v) \,.
		\end{aligned}
	\end{equation*}
	For any $x, y \in \bar{\Omega}_A$, we have
	\begin{equation}\label{phi}
		\begin{aligned}
			\phi (x) - \phi (y) = \int_y^x \int_{\R^3} \bar{\phi} (x', v) \d v \d x' \,.
		\end{aligned}
	\end{equation}
	
	We claim that
	\begin{equation}\label{Claim-phi}
		\begin{aligned}
			\| \bar{\phi} \|_{L^1_{x,v}} \leq C \mathscr{E}_{\mathtt{cro}}^A (g_\sigma) \,.
		\end{aligned}
	\end{equation}
	Indeed, a direct computation yields
		\begin{align}\label{M1-M2}
			\no \| \bar{\phi} \|_{L^1_{x,v}} \leq & \underbrace{ 2 \int_0^A \int_{\R^3} |\sigma_x^ \frac{1}{2}   z_{\alpha'} w_{\beta, \vartheta} g_\sigma| \cdot |\sigma_x^ \frac{1}{2}   z_{\alpha'} w_{\beta, \vartheta} \partial_x g_\sigma| \d v \d x }_{: = \scorpio_1 } \\
			& + \underbrace{  \int_0^A \int_{\R^3} |\sigma_x^ \frac{1}{2}   z_{\alpha'} w_{\beta, \vartheta} g_\sigma| \cdot |\sigma_x^{ - \frac{1}{2} } \sigma_{xx} z_{\alpha'}   w_{\beta, \vartheta} g_\sigma| \d v \d x }_{: = \scorpio_2} \,.
		\end{align}
	By the virtue of $|z_1 \sigma_{xx}| \leq |v_3 \sigma_{xx}| \leq \delta \nu (v) \sigma_x$ derived from Lemma \ref{Lmm-sigma}, we have
	\begin{equation}\label{M2-bnd}
		\begin{aligned}
			\scorpio_2 \leq & \delta  \int_0^A \int_{\R^3} \nu (v) \sigma_x (x,v)  z_{2 \alpha' - 1} w_{\beta, \vartheta}^2 (v) g_\sigma^2 (x,v) \d v \d x \\
			= & \delta  \| \nu^\frac{1}{2} z_{- \alpha} \sigma_x^ \frac{1}{2}   w_{\beta, \vartheta} g_\sigma \|^2_A \,,
		\end{aligned}
	\end{equation}
	where the fact $z_{2 \alpha' - 1} = z_{- 2 \alpha} = z^2_{- \alpha}$ with $\alpha = \frac{1}{2} - \alpha'$ has been used.
	Furthermore, the relation $z_{2 \alpha' - 1} = z^2_{- \alpha}$ also implies
	\begin{equation}\label{M1-bnd}
		\begin{aligned}
			\scorpio_1 = & 2 \int_0^A \int_{\R^3} |\nu^\frac{1}{2} z_{- \alpha} \sigma_x^ \frac{1}{2}   w_{\beta, \vartheta} g_\sigma| \cdot |\nu^{- \frac{1}{2}} z_{- \alpha} \sigma_x^ \frac{1}{2}   z_1 w_{\beta, \vartheta} \partial_x g_\sigma| \d v \d x \\
			\leq & 2 \| \nu^\frac{1}{2} z_{- \alpha} \sigma_x^ \frac{1}{2}   w_{\beta, \vartheta} g_\sigma \|_A \| \nu^{- \frac{1}{2}} z_{- \alpha} \sigma_x^ \frac{1}{2}   z_1 w_{\beta, \vartheta} \partial_x g_\sigma \|_A \\
			\leq & \| \nu^\frac{1}{2} z_{- \alpha} \sigma_x^ \frac{1}{2}   w_{\beta, \vartheta} g_\sigma \|^2_A + \| \nu^{- \frac{1}{2}} z_{- \alpha} \sigma_x^ \frac{1}{2}   z_1 w_{\beta, \vartheta} \partial_x g_\sigma \|^2_A \,.
		\end{aligned}
	\end{equation}
	By using $w_{\beta, \vartheta} \leq w_{\beta + \beta_\gamma, \vartheta}$, one can conclude the claim \eqref{Claim-phi} from \eqref{M1-M2}-\eqref{M2-bnd}-\eqref{M1-bnd}.
	
	Since the claim \eqref{Claim-phi} holds with the finite value in the right hand side of \eqref{Claim-phi}, the relation \eqref{phi} tells us that $\phi \in C (\bar{\Omega}_A)$. Let $ M_\phi = \max_{x \in \bar{\Omega}_A} \phi (x) \geq 0 $ and $ m_\phi = \min_{x \in \bar{\Omega}_A} \phi (x) \geq 0 $. Then there are two points $x_M$, $x_m \in \bar{\Omega}_A$ such that
	\begin{equation*}
		\begin{aligned}
			0 \leq M_\phi - m_\phi = \phi (x_M) - \phi (x_m) = \int_{x_m}^{x_M} \int_{\R^3} \bar{\phi} (x,v) \d v \d x \leq \| \bar{\phi} \|_{L^1_{x,v}} \,.
		\end{aligned}
	\end{equation*}
	If $M_\phi \geq \frac{3}{2} m_\phi$, one has $ M_\phi \leq \tfrac{2}{3} M_\phi + \| \bar{\phi} \|_{L^1_{x,v}} $, which means that
	\begin{equation}\label{Mphi-1}
		\begin{aligned}
			M_\phi \leq 3 \| \bar{\phi} \|_{L^1_{x,v}} \,.
		\end{aligned}
	\end{equation}
	If $M_\phi < \tfrac{3}{2} m_\phi$, one has
	\begin{equation}\label{Mphi-2}
		\begin{aligned}
			M_\phi < & \tfrac{3}{2} m_\phi \leq \tfrac{3}{2 A} \int_0^A \phi (x) \d x = \tfrac{3}{2 A} \| \nu^{- \frac{1}{2}} z_{\alpha} z_{\alpha'} \nu^\frac{1}{2} z_{- \alpha} \sigma_x^ \frac{1}{2}   w_{\beta, \vartheta} g_\sigma \|^2_A \\
			\leq & C \| \nu^\frac{1}{2} z_{- \alpha} \sigma_x^ \frac{1}{2}   w_{\beta + \beta_\gamma, \vartheta} g_\sigma \|^2_A \,,
		\end{aligned}
	\end{equation}
	where the last inequality is derived from $A^{-1} < 1$ for $A > 1$ and $\nu^{- \frac{1}{2}} z_{\alpha} z_{\alpha'} w_{\beta, \vartheta} \leq C w_{\beta + \beta_\gamma, \vartheta}$. We remark that the index $\beta_\gamma$ given in \eqref{beta-gamma} is required here. Consequently, \eqref{Claim-phi}, \eqref{Mphi-1} and \eqref{Mphi-2} indicate that
	\begin{equation*}
		\begin{aligned}
			\| z_{\alpha'} \sigma_x^ \frac{1}{2}   w_{\beta, \vartheta} g_\sigma \|_{L^\infty_x L^2_v} = M_\phi^\frac{1}{2} \leq C \mathscr{E}_{\mathtt{cro}}^A (g_\sigma) \,.
		\end{aligned}
	\end{equation*}
	Namely, the bound \eqref{Y-bnd-4}. Furthermore, the estimate of the bound \eqref{Y-bnd-4-Dif} is similar to that of \eqref{Y-bnd-4}. Thus the proof of Lemma \ref{Lmm-Linfty-L2} is completed.
\end{proof}

\subsection{Estimate for $L^2_{x,v}$ with weight $\nu^\frac{1}{2} \sigma_x^ \frac{1}{2}   z_{- \alpha} w_{\beta + \beta_\gamma, \vartheta}$}

In this subsection, we will control the quantity $ \mathscr{E}_{\mathtt{cro}}^A (g_\sigma) = \| \nu^{- \frac{1}{2}} z_{- \alpha} \sigma_x^ \frac{1}{2}   z_1 w_{\beta + \beta_\gamma, \vartheta} \partial_x g_\sigma \|_A $$+ \| \nu^\frac{1}{2} z_{- \alpha} \sigma_x^ \frac{1}{2}   w_{\beta + \beta_\gamma, \vartheta} g_\sigma \|_A $ in the right hand side of \eqref{Y-bnd-4}.

\begin{lemma}\label{Lmm-L2xv-alpha}
	Let $- 3 < \gamma \leq 1$, $A > 1$, $ 0 \leq \hbar, \vartheta \ll 1 $, the integer $\beta \geq \max \{ 0, - \gamma \}$, and $\alpha = \frac{1}{2} - \alpha'$, where $\alpha'$ is given in Lemma \ref{Lmm-Y-bnd}. Moreover, $l \geq l_0$ for some constant $l_0 \geq 1$ independent of $A, \delta, \hbar$. Then there is a constant $C > 0$, independent of $A$, $ \delta$ and $\hbar$, such that
	\begin{equation}\label{C-L2}
		\begin{aligned}
			&  \| |v_3|^\frac{1}{2} \sigma_x^ \frac{1}{2}   z_{- \alpha} w_{\beta + \beta_\gamma, \vartheta} g_\sigma \|^2_{L^2_{\Sigma_+^0}} + \| |v_3|^\frac{1}{2} \sigma_x^ \frac{1}{2}     z_{- \alpha} w_{\beta + \beta_\gamma, \vartheta} g_\sigma \|^2_{L^2_{\Sigma_+^A}} + [ \mathscr{E}_{\mathtt{cro}}^A (g_\sigma) ]^2 \\
			\leq &  C \|(\delta x + l)^{-\frac{\Theta}{2}}  \nu^{\frac{1}{2}}  w_{\beta + \beta_\gamma, \vartheta} g_\sigma \|^2_A + C [ \mathscr{A}^A_{\mathtt{cro}} (h_\sigma) ]^2 \\
			&+ C [ \mathscr{B}_{\mathtt{cro}} ( \varphi_{A, \sigma} ) ]^2+ C [ \mathscr{C}_{\mathtt{cro}} ( f_{b, \sigma} ) ]^2 \,,
		\end{aligned}
	\end{equation}
	where the functionals $ \mathscr{E}_{\mathtt{cro}}^A (\cdot) $, $ \mathscr{A}^A_{\mathtt{cro}} (\cdot) $, $ \mathscr{B}_{\mathtt{cro}} ( \cdot ) $ and $ \mathscr{C}_{\mathtt{cro}} ( \cdot ) $ are introduced in \eqref{E-cro}, \eqref{As-def}, \eqref{B-def} and \eqref{C-def}, respectively. Similarly, there holds
	\begin{equation}\label{L2w-g-sigma-Bnd-Dif}
		\begin{aligned}
			&  \| |v_3|^\frac{1}{2} \sigma_x^ \frac{1}{2}   z_{- \alpha} w_{\beta + \beta_\gamma, \vartheta} {\vartriangle} g_\sigma \|^2_{L^2_{\Sigma_+^0}} + \| |v_3|^\frac{1}{2} \sigma_x^ \frac{1}{2}     z_{- \alpha} w_{\beta + \beta_\gamma, \vartheta} {\vartriangle} g_\sigma \|^2_{L^2_{\Sigma_+^A}} + [ \mathscr{E}_{\mathtt{cro}}^A ({\vartriangle} g_\sigma) ]^2 \\
			\leq &C\|(\delta x + l)^{-\frac{\Theta}{2}}  \nu^{\frac{1}{2}}  w_{\beta + \beta_\gamma, \vartheta} {\vartriangle} g_\sigma \|^2_A + C [ \mathscr{A}^A_{\mathtt{cro}} ({\vartriangle} h_\sigma) ]^2 + C [ \mathscr{C}_{\mathtt{cro}} ( {\vartriangle} f_{b, \sigma} ) ]^2\,.
		\end{aligned}
	\end{equation}
\end{lemma}

\begin{proof}
	For $0 \leq \alpha < \mu_\gamma$ with $\mu_\gamma > 0$ given in Lemma \ref{Lmm-Kh-L2}, multiplying \eqref{A3-lambda} by $\sigma_x z_{- \alpha}^2 w_{\beta + \beta_\gamma, \vartheta}^2 g_\sigma$ and integrating by parts over $(x,v) \in \Omega_A \times \R^3$, one has
		\begin{align*}
			& \iint_{\Omega_A \times \R^3} \partial_x \big( \tfrac{1}{2} v_3 z_{- \alpha}^2 \sigma_x  w_{\beta + \beta_\gamma, \vartheta}^2 g_\sigma^2 \big) \d v \d x  - \tfrac{1}{2} \iint_{\Omega_A \times \R^3}  v_3 z_{- \alpha}^2 \sigma_{xx}  w_{\beta + \beta_\gamma, \vartheta}^2 g_\sigma^2 \d v \d x \\
			& + \iint_{\Omega_A \times \R^3} [ - \hbar \sigma_x v_3 + \nu (v) ] \sigma_x z_{- \alpha}^2 w_{\beta + \beta_\gamma, \vartheta}^2 g_\sigma^2 \d v \d x -  \iint_{\Omega_A \times \R^3} K_\hbar g_\sigma  \sigma_x z_{- \alpha}^2 w_{\beta + \beta_\gamma, \vartheta}^2 g_\sigma^2 \d v \d x \\
			= & \iint_{\Omega_A \times \R^3} (h_\sigma-D_\hbar g)  \sigma_x z_{- \alpha}^2 w_{\beta + \beta_\gamma, \vartheta}^2 g_\sigma \d v \d x \,.
		\end{align*}
	Notice that by Lemma \ref{Lmm-sigma},
	\begin{equation}\label{N1}
		\begin{aligned}
			\iint_{\Omega_A \times \R^3} [ - \hbar \sigma_x v_3 + \nu (v) ] \sigma_x z_{- \alpha}^2 & w_{\beta + \beta_\gamma, \vartheta}^2 g_\sigma^2 \d v \d x
			 \geq c_\hbar \| \nu^\frac{1}{2} z_{- \alpha} \sigma_x^ \frac{1}{2}   w_{\beta + \beta_\gamma, \vartheta} g_\sigma \|^2_A \,,
		\end{aligned}
	\end{equation}
	where $c_\hbar = 1 - c \hbar > 0$ for sufficiently small $\hbar \geq 0$. Observe that
	\begin{equation*}
		\begin{aligned}
			& \iint_{\Omega_A \times \R^3} \partial_x \big( \tfrac{1}{2} v_3 z_{- \alpha}^2 \sigma_x w_{\beta + \beta_\gamma, \vartheta}^2 g_\sigma^2 \big) \d v \d x \\
			= & \tfrac{1}{2} \int_{\R^3} v_3 z_{- \alpha}^2 \sigma_x w_{\beta + \beta_\gamma, \vartheta}^2 g_\sigma^2 (A, v) \d v \\
			& - \tfrac{1}{2}  \int_{\R^3} v_3 z_{- \alpha}^2 \sigma_x (0, v) w_{\beta + \beta_\gamma, \vartheta}^2 (v) g_\sigma^2 (0, v) \d v \,.
		\end{aligned}
	\end{equation*}
	By the boundary conditions in \eqref{A3-lambda}, it infers that
		\begin{align*}
			& \int_{\R^3} v_3 z_{- \alpha}^2 \sigma_x w_{\beta + \beta_\gamma, \vartheta}^2 g_\sigma^2 (A, v) \d v \\
			= & \int_{v_3 > 0} v_3 z_{- \alpha}^2 \sigma_x w_{\beta + \beta_\gamma, \vartheta}^2 g_\sigma^2 (A, v) \d v + \int_{v_3 < 0} v_3 z_{- \alpha}^2 \sigma_x w_{\beta + \beta_\gamma, \vartheta}^2 g_\sigma^2 (A, v) \d v \\
			= & \int_{v_3 > 0} | v_3 | z_{- \alpha}^2 \sigma_x w_{\beta + \beta_\gamma, \vartheta}^2 g_\sigma^2 (A, v) \d v \\
			& - \int_{v_3 < 0} |v_3| z_{- \alpha}^2 \sigma_x w_{\beta + \beta_\gamma, \vartheta}^2 \varphi_{A, \sigma}^2 (A, v) \d v \\
			= & \| |v_3|^\frac{1}{2} z_{- \alpha} \sigma_x^ \frac{1}{2}     w_{\beta + \beta_\gamma, \vartheta} g_\sigma \|^2_{L^2_{\Sigma_+^A}} - \| |v_3|^\frac{1}{2} z_{- \alpha} \sigma_x^ \frac{1}{2}     w_{\beta + \beta_\gamma, \vartheta} \varphi_{A, \sigma} \|^2_{L^2_{\Sigma_-^A}} \,.
		\end{align*}
	Similarly, it follows that
	\begin{equation*}
		\begin{aligned}
			& - \int_{\R^3} v_3 z_{- \alpha}^2 \sigma_x (0, v) w_{\beta + \beta_\gamma, \vartheta}^2 (v) g_\sigma^2 (0, v) \d v \\
			= & \| |v_3|^\frac{1}{2} z_{- \alpha} \sigma_x^ \frac{1}{2}     w_{\beta + \beta_\gamma, \vartheta} g_\sigma \|^2_{L^2_{\Sigma_+^0}} - \| |v_3|^\frac{1}{2} z_{- \alpha} \sigma_x^ \frac{1}{2}     w_{\beta + \beta_\gamma, \vartheta} f_{b, \sigma} \|^2_{L^2_{\Sigma_-^0}} \,.
		\end{aligned}
	\end{equation*}
	As a result, one has
		\begin{align}\label{N2}
			\no & \iint_{\Omega_A \times \R^3} \partial_x \big( \tfrac{1}{2} v_3 z_{- \alpha}^2 \sigma_x w_{\beta + \beta_\gamma, \vartheta}^2 g_\sigma^2 \big) \d v \d x \\
			\no \geq & \tfrac{1}{2} \| |v_3|^\frac{1}{2} z_{- \alpha} \sigma_x^ \frac{1}{2} w_{\beta + \beta_\gamma, \vartheta} g_\sigma \|^2_{L^2_{\Sigma_+^0}}  \\
			 & + \tfrac{1}{2} \| |v_3|^\frac{1}{2} z_{- \alpha} \sigma_x^ \frac{1}{2}     w_{\beta + \beta_\gamma, \vartheta} g_\sigma \|^2_{L^2_{\Sigma_+^A}} \\
			\no & - \tfrac{1}{2}  \| |v_3|^\frac{1}{2} z_{- \alpha} \sigma_x^ \frac{1}{2}     w_{\beta + \beta_\gamma, \vartheta} f_{b, \sigma} \|^2_{L^2_{\Sigma_-^0}} \\
			\no & - \tfrac{1}{2}  \| |v_3|^\frac{1}{2} z_{- \alpha} \sigma_x^ \frac{1}{2}     w_{\beta + \beta_\gamma, \vartheta} \varphi_{A, \sigma} \|^2_{L^2_{\Sigma_-^A}} \,.
		\end{align}
	By Lemma \ref{Lmm-sigma}, $|v_3 \sigma_{xx}| \leq \delta \sigma_x \nu (v)$. Then it holds
	\begin{equation*}
		\begin{aligned}
			& \big| \tfrac{1}{2} \iint_{\Omega_A \times \R^3} v_3 z_{- \alpha}^2 \sigma_{xx}  w_{\beta + \beta_\gamma, \vartheta}^2 g_\sigma^2 \d v \d x \big| \\
			\leq & C \delta \| \nu^\frac{1}{2} z_{- \alpha} \sigma_x^ \frac{1}{2}   w_{\beta + \beta_\gamma, \vartheta} g_\sigma \|^2_A \leq \tfrac{c_\hbar}{16} \| \nu^\frac{1}{2} z_{- \alpha} \sigma_x^ \frac{1}{2}   w_{\beta + \beta_\gamma, \vartheta} g_\sigma \|^2_A \,,
		\end{aligned}
	\end{equation*}
	where $\delta > 0$ is taken small enough such that $C \delta \leq \frac{c_\hbar}{16}$. Moreover, the H\"older inequality shows
		\begin{align*}
			& \big| \iint_{\Omega_A \times \R^3} h_\sigma  \sigma_x z_{- \alpha}^2 w_{\beta + \beta_\gamma, \vartheta}^2 g_\sigma \d v \d x \big| \\
			\leq & \tfrac{c_\hbar}{16} \| \nu^\frac{1}{2} z_{- \alpha} \sigma_x^ \frac{1}{2}   w_{\beta + \beta_\gamma, \vartheta} g_\sigma \|^2_A + C \| \nu^{- \frac{1}{2}} z_{- \alpha} \sigma_x^ \frac{1}{2}   w_{\beta + \beta_\gamma, \vartheta} h_\sigma \|^2_A \,,
		\end{align*}
	and
	\begin{equation*}
		\begin{aligned}
			& \big| \iint_{\Omega_A \times \R^3} K_\hbar g_\sigma  \sigma_x z_{- \alpha}^2 w_{\beta + \beta_\gamma, \vartheta}^2 g_\sigma^2 \d v \d x \big| \\
			\leq & \tfrac{c_\hbar}{8} \| \nu^\frac{1}{2} z_{- \alpha} \sigma_x^ \frac{1}{2}   w_{\beta + \beta_\gamma, \vartheta} g_\sigma \|^2_A + C \| \nu^{- \frac{1}{2}} z_{- \alpha} \sigma_x^ \frac{1}{2}   w_{\beta + \beta_\gamma, \vartheta} K_\hbar g_\sigma \|^2_A \,.
		\end{aligned}
	\end{equation*}
	Recall the definition of $D_\hbar g$ in \eqref{D-def}. Lemma \ref{Lmm-sigma} shows
	$$|\sigma(x,v)-\sigma(x,v_*)|\leq c \left| |v-\u|^2-|v_*- \u|^2 \right|\,,$$
	which further implies
	\begin{align*}
		& \big| \iint_{\Omega_A \times \R^3} D_\hbar g  \sigma_x z_{- \alpha}^2 w_{\beta + \beta_\gamma, \vartheta}^2 g_\sigma \d v \d x \big| \\
		\leq & \tfrac{c_\hbar}{4} \| \nu^\frac{1}{2} z_{- \alpha} \sigma_x^ \frac{1}{2}   w_{\beta + \beta_\gamma, \vartheta} g_\sigma \|^2_A + C \| \nu^{- \frac{1}{2}} z_{- \alpha} \sigma_x^ \frac{1}{2}   w_{\beta + \beta_\gamma, \vartheta} D_\hbar g \|^2_A \,,\\
		\leq & \tfrac{c_\hbar}{4} \| \nu^\frac{1}{2} z_{- \alpha} \sigma_x^ \frac{1}{2}   w_{\beta + \beta_\gamma, \vartheta} g_\sigma \|^2_A + C \hbar \|(\delta x + l)^{-\frac{\Theta}{2}}  \nu^{\frac{1}{2}}  w_{\beta + \beta_\gamma, \vartheta} g_\sigma \|^2_A \,.
	\end{align*}
	We thereby establish
		\begin{align}\label{M1}
			\no &   \| |v_3|^\frac{1}{2} \sigma_x^ \frac{1}{2}   z_{- \alpha} w_{\beta + \beta_\gamma, \vartheta} g_\sigma \|^2_{L^2_{\Sigma_+}} \\
			\no & + \| |v_3|^\frac{1}{2} \sigma_x^ \frac{1}{2}     z_{- \alpha} w_{\beta + \beta_\gamma, \vartheta} g_\sigma \|^2_{ L^2_{ \Sigma_+^A } } + \| \nu^\frac{1}{2} z_{- \alpha} \sigma_x^ \frac{1}{2}   w_{\beta + \beta_\gamma, \vartheta} g_\sigma \|^2_A \\
			\leq & C \| \nu^{- \frac{1}{2}} z_{- \alpha} \sigma_x^ \frac{1}{2}   w_{\beta + \beta_\gamma, \vartheta} K_\hbar g_\sigma \|^2_A + C \hbar \|(\delta x + l)^{-\frac{\Theta}{2}}  \nu^{\frac{1}{2}}  w_{\beta + \beta_\gamma, \vartheta} g_\sigma \|^2_A  \\
			\no&+ C [ \mathscr{A}^A_{\mathtt{cro}} (h_\sigma) ]^2 + C [ \mathscr{B}_{\mathtt{cro}} ( \varphi_{A, \sigma} ) ]^2+ C [ \mathscr{C}_{\mathtt{cro}} ( f_{b, \sigma} ) ]^2 \,.
		\end{align}
	Lemma \ref{Lmm-Kh-L2} show that for $0 \leq \alpha < \mu_\gamma$ and $- 3 < \gamma \leq 1$,
	\begin{equation}\label{M2}
		\begin{aligned}
			\| \nu^{- \frac{1}{2}} z_{- \alpha} \sigma_x^ \frac{1}{2}   w_{\beta + \beta_\gamma, \vartheta} K_\hbar g_\sigma \|^2_A \lesssim \| \nu^\frac{1}{2} \sigma_x^ \frac{1}{2}   w_{\beta + \beta_\gamma, \vartheta} g_\sigma \|^2_A \,.
		\end{aligned}
	\end{equation}
	It follows from Lemma \ref{Lmm-sigma} that
	\begin{equation*}
		\sigma_x \leq c_2 (\delta x + l)^{-\Theta}.
	\end{equation*}
	It thereby infers
	\begin{equation*}
		\begin{aligned}
			\| \nu^{- \frac{1}{2}} z_{- \alpha} \sigma_x^ \frac{1}{2}   w_{\beta + \beta_\gamma, \vartheta} K_\hbar g_\sigma \|^2_A \lesssim \| \nu^\frac{1}{2} (\delta x + l)^{-\frac{\Theta}{2}}   w_{\beta + \beta_\gamma, \vartheta} g_\sigma \|^2_A \,.
		\end{aligned}
	\end{equation*}
	As a result, one has
	\begin{equation}\label{L2w-g-sigma-Bnd}
		\begin{aligned}
			&   \| |v_3|^\frac{1}{2} \sigma_x^ \frac{1}{2}   z_{- \alpha} w_{\beta + \beta_\gamma, \vartheta} g_\sigma \|^2_{L^2_{\Sigma_+^0}} \\
			& + \| |v_3|^\frac{1}{2} \sigma_x^ \frac{1}{2}     z_{- \alpha} w_{\beta + \beta_\gamma, \vartheta} g_\sigma \|^2_{ L^2_{ \Sigma_+^A } } + \| \nu^\frac{1}{2} z_{- \alpha} \sigma_x^ \frac{1}{2}   w_{\beta + \beta_\gamma, \vartheta} g_\sigma \|^2_A \\
			\leq & C \|(\delta x + l)^{-\frac{\Theta}{2}}  \nu^{\frac{1}{2}}  w_{\beta + \beta_\gamma, \vartheta} g_\sigma \|^2_A + C [ \mathscr{A}^A_{\mathtt{cro}} (h_\sigma) ]^2 \\
			&+ C [ \mathscr{B}_{\mathtt{cro}} ( \varphi_{A, \sigma} ) ]^2+ C [ \mathscr{C}_{\mathtt{cro}} ( f_{b, \sigma} ) ]^2
		\end{aligned}
	\end{equation}
	for $0 \leq \alpha < \mu_\gamma$. Here the constant $C > 0$ is independent of $A$ and $\hbar$.
	
	Recalling \eqref{A3-lambda}, one has $ \partial_x g_\sigma = - [ - \hbar \sigma_x + \tfrac{\nu (v)}{v_3} ] g_\sigma +  \tfrac{1}{v_3} K_\hbar g_\sigma + \tfrac{1}{v_3} (h_\sigma-D_\hbar g) $, which means
	\begin{equation*}
		\begin{aligned}
			& | \nu^{- \frac{1}{2}} z_{- \alpha} \sigma_x^ \frac{1}{2}   z_1 w_{\beta + \beta_\gamma, \vartheta} \partial_x g_\sigma | \\
			\leq & C (\hbar \nu^{-1} z_1 \sigma_x + \tfrac{z_1}{|v_3|}) |\nu^\frac{1}{2} z_{- \alpha} \sigma_x^ \frac{1}{2}   w_{\beta + \beta_\gamma, \vartheta} g_\sigma | \\
			& + C | \tfrac{z_1}{|v_3|} \nu^{- \frac{1}{2}} \sigma_x^ \frac{1}{2}   z_{- \alpha} w_{\beta + \beta_\gamma, \vartheta} K_\hbar g_\sigma | + C | \tfrac{z_1}{|v_3|} \nu^{- \frac{1}{2}} \sigma_x^ \frac{1}{2}   z_{- \alpha} w_{\beta + \beta_\gamma, \vartheta} (h_\sigma-D_\hbar g) | \,.
		\end{aligned}
	\end{equation*}
	Observe that $\frac{z_1}{|v_3|} \leq 1$, and $\sigma_x \leq c \frac{\nu (v)}{|v_3|}$ by Lemma \ref{Lmm-sigma}, which imply $ \hbar \nu^{-1} z_1 \sigma_x + \tfrac{z_1}{|v_3|} \leq c \hbar + 1 $. It therefore follows that
	\begin{equation*}
		\begin{aligned}
			\| \nu^{- \frac{1}{2}} z_{- \alpha} \sigma_x^ \frac{1}{2}   z_1 w_{\beta + \beta_\gamma, \vartheta} & \partial_x g_\sigma \|^2_A \leq C \| \nu^\frac{1}{2} z_{- \alpha} \sigma_x^ \frac{1}{2}   w_{\beta + \beta_\gamma, \vartheta} g_\sigma \|^2_A \\
			& + C \| \nu^{- \frac{1}{2}} \sigma_x^ \frac{1}{2}   z_{- \alpha} w_{\beta + \beta_\gamma, \vartheta} K_\hbar g_\sigma \|^2_A + C [ \mathscr{A}^A_{\mathtt{cro}} (h_\sigma- D_\hbar g) ]^2 \,.
		\end{aligned}
	\end{equation*}
	Lemma \ref{Lmm-Kh-L2} reads that $ \int_{\R^3} | \nu^{- \frac{1}{2}} \sigma_x^ \frac{1}{2} z_{- \alpha} w_{\beta + \beta_\gamma, \vartheta} K_\hbar g_\sigma |^2 \d v \leq C \int_{\R^3} |\nu^\frac{1}{2} z_{- \alpha} \sigma_x^ \frac{1}{2} w_{\beta + \beta_\gamma, \vartheta} g_\sigma|^2 \d v $. Then, together with \eqref{L2w-g-sigma-Bnd}, one has
	\begin{equation}\label{L2w-pxg-sigma-Bnd}
		\begin{aligned}
			& \| \nu^{- \frac{1}{2}} z_{- \alpha} \sigma_x^ \frac{1}{2}   z_1 w_{\beta + \beta_\gamma, \vartheta} \partial_x g_\sigma \|^2_A \\
			\leq & C \| \nu^\frac{1}{2} z_{- \alpha} \sigma_x^ \frac{1}{2}   w_{\beta + \beta_\gamma, \vartheta} g_\sigma \|^2_A + C \| \nu^{- \frac{1}{2}} \sigma_x^ \frac{1}{2}   z_{- \alpha} w_{\beta + \beta_\gamma, \vartheta} h_\sigma \|^2_A +C \| \nu^{- \frac{1}{2}} \sigma_x^ \frac{1}{2}   z_{- \alpha} w_{\beta + \beta_\gamma, \vartheta} D_\hbar g \|^2_A\\
			\leq & C \|(\delta x + l)^{-\frac{\Theta}{2}}  \nu^{\frac{1}{2}}  w_{\beta + \beta_\gamma, \vartheta} g_\sigma \|^2_A + C [ \mathscr{A}^A_{\mathtt{cro}} (h_\sigma) ]^2 \\
			&+ C [ \mathscr{B}_{\mathtt{cro}} ( \varphi_{A, \sigma} ) ]^2+ C [ \mathscr{C}_{\mathtt{cro}} ( f_{b, \sigma} ) ]^2\,.
		\end{aligned}
	\end{equation}
	Then the bounds \eqref{L2w-g-sigma-Bnd} and \eqref{L2w-pxg-sigma-Bnd} conclude the estimate \eqref{C-L2}.

	Furthermore, as the similar arguments in \eqref{L2w-g-sigma-Bnd} and \eqref{L2w-pxg-sigma-Bnd}, one can easily knows that ${\vartriangle} g_\sigma$ obeying the equation \eqref{L-lambda-Dif} satisfies the bound \eqref{L2w-g-sigma-Bnd-Dif}. Consequently, the proof of Lemma \ref{Lmm-L2xv-alpha} is completed.
\end{proof}

\subsection{$w_{\beta + \beta_\gamma, \vartheta}$-weighted $L^2_{x,v}$ estimates}

In this subsection, the majority is to control the quantity $\|(\delta x + l)^{-\frac{\Theta}{2}}  \nu^{\frac{1}{2}}  w_{\beta + \beta_\gamma, \vartheta} g_\sigma \|^2_A$
appeared in the right-hand side of \eqref{C-L2} in Lemma \ref{Lmm-L2xv-alpha}. Note that
\begin{equation}\label{EA2}
	\|(\delta x + l)^{-\frac{\Theta}{2}}\nu^\frac{1}{2} w_{\beta + \beta_\gamma , \vartheta} g_\sigma \|^2_A\leq C\mathscr{E}^A_{2}(w_{\beta + \beta_\gamma, \vartheta} g_\sigma)\,,
\end{equation}
it thereby suffices to control $\mathscr{E}^A_{2}(g_\sigma)$\,. More precisely, the following lemma holds.

\begin{lemma}\label{Lmm-L2xv}
	Let $- 3 < \gamma \leq 1$, $A > 1$, $l \gg 1$, $0 < \delta \ll 1$, the integer $\beta \geq \max \{ 0, - \gamma \}$, $0 \leq \hbar, \vartheta \ll 1$.
	Then there is a constant $C > 0$, independent of $A, \delta$ and $\hbar$, such that
	\begin{equation}\label{L2-g-sigma-Bnd}
		\begin{aligned}
			&   \| |v_3|^\frac{1}{2} w_{\beta + \beta_\gamma, \vartheta} g_\sigma \|^2_{L^2_{\Sigma_+^0}} + \| |v_3|^\frac{1}{2} w_{\beta + \beta_\gamma, \vartheta} g_\sigma \|^2_{L^2_{\Sigma_+^A}} + [ \mathscr{E}^A_{2 } (w_{\beta + \beta_\gamma, \vartheta} g_\sigma ) ]^2 \\
			\leq &  C  \| (\delta x + l)^{-\frac{\Theta}{2}} \P w_{\beta + \beta_\gamma, \vartheta}  g_\sigma \|^2_A + C [ \mathscr{A}^A_{ 2} (w_{\beta + \beta_\gamma, \vartheta} h_\sigma) ]^2 \\
			&+ C [ \mathscr{B}_2 (w_{\beta + \beta_\gamma, \vartheta} \varphi_{A, \sigma}) ]^2 + C [ \mathscr{C}_2 (w_{\beta + \beta_\gamma, \vartheta} f_{b, \sigma}) ]^2 \,,
		\end{aligned}
	\end{equation}
	where the functionals $ \mathscr{E}^A_\infty ( \cdot ) $, $ \mathscr{E}^A_{2 } ( \cdot ) $, $ \mathscr{A}^A_{ 2} ( \cdot ) $ $ \mathscr{B}_2 ( \cdot ) $ and $ \mathscr{C}_2 ( \cdot ) $ are defined in \eqref{E-infty}, \eqref{E2-lambda}, \eqref{As-def}, \eqref{B-def} and \eqref{C-def}, respectively. Moreover, the difference ${\vartriangle} g_\sigma$ enjoys the bound
	\begin{equation}\label{L2-g-sigma-Bnd-Dif}
		\begin{aligned}
			 & \| |v_3|^\frac{1}{2} w_{\beta + \beta_\gamma, \vartheta} {\vartriangle} g_\sigma \|^2_{L^2_{\Sigma_+}} + \| |v_3|^\frac{1}{2} w_{\beta + \beta_\gamma, \vartheta} {\vartriangle} g_\sigma \|^2_{L^2_{\Sigma_+^A}} + [ \mathscr{E}^A_{2 } ( w_{\beta + \beta_\gamma, \vartheta} {\vartriangle} g_\sigma ) ]^2 \\
			\leq & C  \| (\delta x + l)^{-\frac{\Theta}{2}} \P w_{\beta + \beta_\gamma, \vartheta} {\triangle}  g_\sigma \|^2_A  + C [ \mathscr{A}^A_{ 2} (w_{\beta + \beta_\gamma, \vartheta} {\vartriangle} h_\sigma) ]^2 \\
			& + C [ \mathscr{C}_2 (w_{\beta + \beta_\gamma, \vartheta}{\vartriangle}f_{b, \sigma}) ]^2 \,.
		\end{aligned}
	\end{equation}
\end{lemma}

\begin{proof}[Proof of Lemma \ref{Lmm-L2xv}]
	Multiplying \eqref{A3-lambda} by $w_{\beta + \beta_\gamma, \vartheta}^2 g_\sigma$ and integrating by parts over $(x,v) \in \Omega_A \times \R^3$, we have
			\begin{multline*}
				\iint_{\Omega_A \times \R^3} v_3 \partial_x g_\sigma \cdot w_{\beta + \beta_\gamma, \vartheta}^2 g_\sigma \d v \d x + \mathscr{W} (g_\sigma) \\
				= \iint_{\Omega_A \times \R^3} h_\sigma \cdot w_{\beta + \beta_\gamma, \vartheta}^2 g_\sigma \d v \d x
				 -\iint_{\Omega_A \times \R^3} D_\hbar g \cdot w_{\beta + \beta_\gamma, \vartheta}^2 g_\sigma \d v \d x \,,
			\end{multline*}
	where
	\begin{equation*}
		\begin{aligned}
			\mathscr{W} (g_\sigma) = & \iint_{\Omega_A \times \R^3} [ - \hbar \sigma_x v_3 + \nu (v) ] g_\sigma \cdot w_{\beta + \beta_\gamma, \vartheta}^2 g_\sigma \d v \d x -  \iint_{\Omega_A \times \R^3} K_\hbar g_\sigma \cdot w_{\beta + \beta_\gamma, \vartheta}^2 g_\sigma \d v \d x \,.
		\end{aligned}
	\end{equation*}
	
	Note that
	\begin{equation*}
		\begin{aligned}
			\iint_{\Omega_A \times \R^3} v_3 \partial_x g_\sigma \cdot w_{\beta + \beta_\gamma, \vartheta}^2 g_\sigma \d v \d x = \tfrac{1}{2} \int_{\R^3} v_3 w_{\beta + \beta_\gamma, \vartheta}^2 g_\sigma^2 (A, v) \d v - \tfrac{1}{2} \int_{\R^3} v_3 w_{\beta + \beta_\gamma, \vartheta}^2 g_\sigma^2 (0, v) \d v \,.
		\end{aligned}
	\end{equation*}
	The boundary condition in \eqref{A3-lambda} reduces to
	\begin{equation*}
		\begin{aligned}
			\tfrac{1}{2} \int_{\R^3} v_3 w_{\beta + \beta_\gamma, \vartheta}^2 g_\sigma^2 (A, v) \d v = \tfrac{1}{2} \| |v_3|^\frac{1}{2} w_{\beta + \beta_\gamma, \vartheta} g_\sigma \|^2_{L^2_{\Sigma_+^A}} - \tfrac{1}{2} \| |v_3|^\frac{1}{2} w_{\beta + \beta_\gamma, \vartheta} \varphi_{A, \sigma} \|^2_{L^2_{\Sigma_-^A}} \,,
		\end{aligned}
	\end{equation*}
	and
	\begin{equation*}
		\begin{aligned}
			-\tfrac{1}{2} \int_{\R^3} v_3 w_{\beta + \beta_\gamma, \vartheta}^2 g_\sigma^2 (0, v) \d v = \tfrac{1}{2} \| |v_3|^\frac{1}{2} w_{\beta + \beta_\gamma, \vartheta} g_\sigma \|^2_{L^2_{\Sigma_+^0}} - \tfrac{1}{2} \| |v_3|^\frac{1}{2} w_{\beta + \beta_\gamma, \vartheta} f_{b, \sigma} \|^2_{L^2_{\Sigma_-^0}} \,.
		\end{aligned}
	\end{equation*}
	As a consequence, one has
		\begin{align}\label{BC-L2}
			\no & \iint_{\Omega_A \times \R^3} v_3 \partial_x g_\sigma \cdot w_{\beta + \beta_\gamma, \vartheta}^2 g_\sigma \d v \d x \\
			\no \geq & \tfrac{1}{2} \| |v_3|^\frac{1}{2} w_{\beta + \beta_\gamma, \vartheta} g_\sigma \|^2_{L^2_{\Sigma_+^0}} + \tfrac{1}{2} \| |v_3|^\frac{1}{2} w_{\beta + \beta_\gamma, \vartheta} g_\sigma \|^2_{L^2_{\Sigma_+^A}} \\
			& - C [ \mathscr{B}_2 (w_{\beta + \beta_\gamma, \vartheta} \varphi_{A, \sigma} ) ]^2 - C [ \mathscr{B}_2 (w_{\beta + \beta_\gamma, \vartheta} f_{b, \sigma} ) ]^2 \,.
		\end{align}
	
	Recalling $\L_\hbar g_\sigma = \nu (v) g_\sigma - K_\hbar g_\sigma$, there holds
	\begin{equation*}
		\begin{aligned}
			\mathscr{W} (g_\sigma) = &  \iint_{\Omega_A \times \R^3} \L_\hbar g_\sigma \cdot w_{\beta + \beta_\gamma, \vartheta}^2 g_\sigma \d v \d x \\
			& +  \iint_{\Omega_A \times \R^3} ( - \hbar \sigma_x v_3 g_\sigma ) w_{\beta + \beta_\gamma, \vartheta}^2 g_\sigma \d v \d x \,.
		\end{aligned}
	\end{equation*}
    Note from Lemma \ref{Lmm-sigma} that  $|\sigma_x| \leq c_2 ( \delta x + l)^{-\Theta}$ and $|\sigma_x v_3|\leq c \nu(v)$, it holds that
	\begin{equation*}
		\begin{aligned}
			& \big|  \iint_{\Omega_A \times \R^3} ( - \hbar \sigma_x v_3 g_\sigma ) w_{\beta + \beta_\gamma, \vartheta}^2 g_\sigma \d v \d x \big|\\
			\leq &  C \hbar  \iint_{\Omega_A \times \R^3}\sigma_x v_3 (\P w_{\beta + \beta_\gamma, \vartheta} g_\sigma)^2 \d v \d x + C \hbar \iint_{\Omega_A \times \R^3}\sigma_x v_3 (\P^\perp w_{\beta + \beta_\gamma, \vartheta})^2 \d v \d x \\
			\leq & C \hbar \|(\delta x + l)^{-\frac{\Theta}{2}} \P w_{\beta + \beta_\gamma, \vartheta} g_\sigma \|^2_A+ C \hbar \| \nu^{\frac{1}{2}} \P^\perp w_{\beta + \beta_\gamma, \vartheta} g_\sigma \|^2_A\,.
		\end{aligned}
	\end{equation*}

	Moreover, by Lemma 2.2 of \cite{Chen-Liu-Yang-2004-AA} (or Lemma 2.4 of \cite{Wang-Yang-Yang-2007-JMP}) and the similar arguments in Corollary 1 of \cite{Strain-Guo-2008-ARMA}, it infers that
	\begin{equation*}
		\begin{aligned}
			 \iint_{\Omega_A \times \R^3} \L_\hbar g_\sigma \cdot w_{\beta + \beta_\gamma, \vartheta}^2 g_\sigma \d v \d x \geq &  \mu_2 \| \nu^\frac{1}{2} \P^\perp w_{\beta + \beta_\gamma, \vartheta} g_\sigma \|^2_A \\
			& -  C \hbar^2 \| (\delta x + l)^{- \frac{\Theta}{2}} \P w_{\beta + \beta_\gamma, \vartheta} g_\sigma \|^2_A
		\end{aligned}
	\end{equation*}
	for $\mu_2 > 0$. Then
	\begin{equation}\label{Wg-bnd}
		\begin{aligned}
			\mathscr{W} (g_\sigma) \geq & \tfrac{1}{2}  \mu_2 \| \nu^\frac{1}{2} \P^\perp w_{\beta + \beta_\gamma, \vartheta} g_\sigma \|^2_A \\
			& - C \hbar \| (\delta x + l)^{-\frac{\Theta}{2}} \P w_{\beta + \beta_\gamma, \vartheta}  g_\sigma \|^2_A  \,\\
		 \geq & c_0 [ \mathscr{E}_{2 }^A (w_{\beta + \beta_\gamma, \vartheta} g_\sigma) ]^2 -  C   \| (\delta x + l)^{-\frac{\Theta}{2}} \P w_{\beta + \beta_\gamma, \vartheta}  g_\sigma \|^2_A \,,
	\end{aligned}
   \end{equation}
	where the functional $\mathscr{E}_{2 }^A ( \cdot )$ is given in \eqref{E2-lambda}.
	
	As for the quantity $\iint_{\Omega_A \times \R^3} h_\sigma \cdot w_{\beta + \beta_\gamma, \vartheta}^2 g_\sigma \d v \d x$, one has
	\begin{equation}\label{h-bnd-2}
		\begin{aligned}
			& \iint_{\Omega_A \times \R^3} h_\sigma \cdot w_{\beta + \beta_\gamma, \vartheta}^2 g_\sigma \d v \d x \\
			= & \iint_{\Omega_A \times \R^3} \P w_{\beta + \beta_\gamma, \vartheta} h_\sigma \cdot \P w_{\beta + \beta_\gamma, \vartheta} g_\sigma \d v \d x + \iint_{\Omega_A \times \R^3} \P^\perp w_{\beta + \beta_\gamma, \vartheta} h_\sigma \cdot \P^\perp w_{\beta + \beta_\gamma, \vartheta} g_\sigma \d v \d x \\
			\leq & \tfrac{1}{2} c_0 \| (\delta x + l)^{- \frac{\Theta}{2}} \P w_{\beta + \beta_\gamma, \vartheta} g_\sigma \|^2_A + C \| (\delta x + l)^{ \frac{\Theta}{2}} \P w_{\beta + \beta_\gamma, \vartheta} h_\sigma \|^2_A \\
			& + \tfrac{1}{2} c_0 \| \nu^\frac{1}{2} \P^\perp w_{\beta + \beta_\gamma, \vartheta} g_\sigma \|^2_A + C \| \nu^{- \frac{1}{2}} \P^\perp w_{\beta + \beta_\gamma, \vartheta} h_\sigma \|^2_A \\
			= & \tfrac{1}{2} c_0  [ \mathscr{E}_{2 }^A (w_{\beta + \beta_\gamma, \vartheta} g_\sigma) ]^2 + C  [ \mathscr{A}_{2 }^A (w_{\beta + \beta_\gamma, \vartheta} h_\sigma) ]^2 \,.
		\end{aligned}
	\end{equation}
	
	We then control the quantity $-\iint_{\Omega_A \times \R^3} D_\hbar g \cdot w_{\beta + \beta_\gamma, \vartheta}^2 g_\sigma \d v \d x$. Recalling \eqref{D-def},
    \begin{equation}
       \begin{aligned}
			-\iint_{\Omega_A \times \R^3} D_\hbar g \cdot w_{\beta + \beta_\gamma, \vartheta}^2 g_\sigma \d v \d x =& \underbrace{-\bar{\alpha} \iint_{\Omega_A \times \R^3} ( \delta x + l)^{-\Theta} e^{\hbar \sigma } \P^+ v_3 e^{-\hbar \sigma} g_\sigma \cdot w_{\beta + \beta_\gamma, \vartheta}^2 g_\sigma \d v \d x}_{:= I_1}\\
			&\underbrace{-\bar{\beta} \iint_{\Omega_A \times \R^3} ( \delta x + l)^{-\Theta} e^{\hbar \sigma } \P^0 v_3 e^{-\hbar \sigma} g_\sigma \cdot w_{\beta + \beta_\gamma, \vartheta}^2 g_\sigma \d v \d x}_{:= I_2}\,,
        \end{aligned}
	\end{equation}
    where $\bar{\alpha}, \bar{\beta} = O(1) \hbar \ll 1$. Lemma \ref{Lmm-sigma} infers that $\sigma(x,v)-\sigma(x,v_*) \leq | c \left||v-\u|^2-|v_*-\u |^2\right|$, one thereby has
	\begin{equation}\label{I1-bnd}
		\begin{aligned}
			|I_{1}| &\leq  C \big| \bar{\alpha} \int_{\Omega_A} ( \delta x + l)^{-\Theta} \| w_{\beta + \beta_\gamma, \vartheta} g_\sigma \|^2_{L^2_v}  \d x \big| \\
			&\leq C \hbar [ \mathscr{E}_{2 }^A (w_{\beta + \beta_\gamma, \vartheta} g_\sigma) ]^2\,.
		\end{aligned}
	\end{equation}
    Similarly, one has
\begin{equation}\label{I2-bnd}
	|I_2| \leq C \hbar [\mathscr{E}^A_{2}(w_{\beta+\beta_\gamma,\vartheta} g_\sigma )]^2 \,.
\end{equation}
	Collecting the all above estimates \eqref{BC-L2}, \eqref{Wg-bnd}, \eqref{h-bnd-2}, \eqref{I1-bnd} and \eqref{I2-bnd}, one  concludes the bound \eqref{L2-g-sigma-Bnd}. Moreover, as similar arguments as in \eqref{L2-g-sigma-Bnd}, the difference ${\vartriangle} g_\sigma : = g_{\sigma 2} - g_{\sigma 1}$ enjoys the estimate \eqref{L2-g-sigma-Bnd-Dif}. Therefore, the proof of Lemma \ref{Lmm-L2xv} is completed.
\end{proof}

\subsection{$L^2_{x,v}$ estimates}

In this subsection, the majority is to control the quantity $\| (\delta x + l)^{-\frac{\Theta}{2}} \P w_{\beta+\beta_\gamma ,\vartheta} g_\sigma \|_{A}^2$
appeared in the right-hand side of \eqref{L2-g-sigma-Bnd} in Lemma \ref{Lmm-L2xv} by employing the energy method.
We first state the following lemma:
\begin{lemma}\label{Lmm-P-wg}
	Let $- 3 < \gamma \leq 1$, $A > 1$, $l \gg 1$, $0 < \delta \ll 1$, the integer $\beta \geq \max \{ 0, - \gamma \}$, $0 \leq \hbar, \vartheta \ll 1$.
	Then there is a constant $C > 0$, independent of $A, \delta$ and $\hbar$, such that
	\begin{equation}\label{P-wg}
		\| (\delta x + l)^{-\frac{\Theta}{2}} \P w_{\beta+\beta_\gamma ,\vartheta} g_\sigma \|_{A}^2 \leq C [\mathscr{E}^A_{2}(g_\sigma)]^2\,.
	\end{equation}
\end{lemma}
\begin{proof}
	 It suffices to show that
	 \begin{equation*}
		\| (\delta x + l)^{-\frac{\Theta}{2}} \P w_{\beta+\beta_\gamma ,\vartheta} g_\sigma \|_{A}^2 \leq C \big(\| (\delta x + l)^{-\frac{\Theta}{2}} \P  g_\sigma \|_{A}^2+\| \P^\perp  g_\sigma \|_{A}^2\big)\,.
	\end{equation*}

To do this, we decompose $g_\sigma$ and $w_{\beta+\beta_\gamma ,\vartheta} g_\sigma$ into
\begin{equation*}
	g_\sigma=\sum_{j=0}^{4}a_j \psi_j +\P^\perp g_\sigma, \quad  \text{and}\quad \!
	w_{\beta+\beta_\gamma ,\vartheta} g_\sigma =\sum_{j=0}^{4} \bar{a}_j \psi_j +\P^\perp w_{\beta+\beta_\gamma ,\vartheta} g_\sigma \,.
\end{equation*}
Then, the H\"older inequality shows
\begin{equation}\label{a}
\begin{aligned}
		\bar{a}_j =\tfrac{1}{(\psi_j,\psi_j)} \int_{\R^3} w_{\beta+\beta_\gamma ,\vartheta} g_\sigma \psi_j \d v &= \tfrac{1}{(\psi_j,\psi_j)} \int_{\R^3} w_{\beta+\beta_\gamma ,\vartheta} \big(\sum_{i=0}^{4}a_i \psi_i +\P^\perp g_\sigma \big) \psi_j \d v \\
		&\leq C \big(\sum_{i=0}^{4} | a_i| + \| \nu^\frac{1}{2} \P^\perp g_\sigma \|_{L^2_v}\big) \,,
\end{aligned}
\end{equation}
for $j=0,\cdots, 4$. It follows from \eqref{a} that
\begin{equation*}
	\begin{aligned}
		\| (\delta x + l)^{-\frac{\Theta}{2}} \P w_{\beta+\beta_\gamma ,\vartheta} g_\sigma \|_{A}^2 = &  \int_{\Omega_A} (\delta x + l)^{-\Theta}\sum_{j=0}^{4} \bar{a}_j^2 \d x\\
		\leq & \int_{\Omega_A} (\delta x + l)^{-\Theta}\big(\sum_{j=0}^{4} a_j^2+ \| \nu^\frac{1}{2} \P^\perp g_\sigma \|^2_{L^2_v} \big) \d x\\
		\leq & C \big(\| (\delta x + l)^{-\frac{\Theta}{2}} \P  g_\sigma \|_{A}^2+\| \P^\perp  g_\sigma \|_{A}^2\big)\,,
	\end{aligned}
\end{equation*}
since $\Theta =\frac{1-\gamma}{3-\gamma}\geq 0$. Hence, we have the estimate \eqref{P-wg}.
\end{proof}
Lemma \ref{Lmm-P-wg} shows that we only need to control the quantity $\mathscr{E}^A_{2}(g_\sigma)$. Before go into the proof, we first state the following Lemma.
\begin{lemma}\label{Lmm-macro-diss}
	Let $l \gg 1$, then there exists $c_3>0$, such that for $\bar{\alpha},\bar{\beta} =O(1)\hbar \ll 1$ suitable large,
	\begin{equation}
		\int_{2(1+|v-\u |)^{3 - \gamma} \leq l} \big( \alpha \P^+ v_3 \P g+\beta \P^0 \P g - 5 \hbar \P v_3 \P g\big) g \d v \geq \hbar c_3 \|\P g\|_{L^2_v}^2\,.
	\end{equation}
\end{lemma}
\begin{proof}
	The proof is similar to Lemma 3.1 of \cite{Chen-Liu-Yang-2004-AA}. We omit the details of the proof for brevity.
\end{proof}
\begin{lemma}\label{Lmm-L2xv.}
	Let $- 3 < \gamma \leq 1$, $A > 1$, $l \gg 1$, $0 < \delta \ll 1$, the integer $\beta \geq \max \{ 0, - \gamma \}$, $0 \leq \hbar, \vartheta \ll 1$.
	Then there is a constant $C > 0$, independent of $A$ and $ \delta$, but dependent on $\hbar$, such that
	\begin{equation}\label{L2-g-sigma-Bnd.}
		\begin{aligned}
			&   \| |v_3|^\frac{1}{2}  g_\sigma \|^2_{L^2_{\Sigma_+^0}} + \| |v_3|^\frac{1}{2}  g_\sigma \|^2_{L^2_{\Sigma_+^A}} + [ \mathscr{E}^A_{2 } ( g_\sigma ) ]^2 \\
			\leq &  C [ \mathscr{A}^A_{ 2} (h_\sigma) ]^2 + C [ \mathscr{B}_2 (\varphi_{A, \sigma}) ]^2+ C [ \mathscr{C}_2 (f_{b, \sigma}) ]^2 \,,
		\end{aligned}
	\end{equation}
	where the functionals $ \mathscr{E}^A_{2 } ( \cdot ) $, $ \mathscr{A}^A_{ 2} ( \cdot ) $ and $ \mathscr{B}_2 ( \cdot ) $ are defined in \eqref{E2-lambda}, \eqref{As-def} and \eqref{B-def}, respectively. Moreover, the difference ${\vartriangle} g_\sigma$ enjoys the bound
	\begin{equation}\label{L2-g-sigma-Bnd-Dif.}
		\begin{aligned}
			  \| |v_3|^\frac{1}{2}  {\vartriangle} g_\sigma \|^2_{L^2_{\Sigma_+}} + \| |v_3|^\frac{1}{2}  {\vartriangle} g_\sigma \|^2_{L^2_{\Sigma_+^A}} + [ \mathscr{E}^A_{2 } ( {\vartriangle} g_\sigma ) ]^2 \\
			\leq C [ \mathscr{A}^A_{ 2} ( {\vartriangle} h_\sigma) ]^2 + C [ \mathscr{C}_2 ({\vartriangle}f_{b, \sigma}) ]^2 \,,
		\end{aligned}
	\end{equation}
\end{lemma}

\begin{proof}[Proof of Lemma \ref{Lmm-L2xv.}]
	Multiplying \eqref{A3-lambda} by  $g_\sigma$ and integrating by parts over $(x,v) \in \Omega_A \times \R^3$, we have
	\begin{multline}\label{L2-g-sigma}
			\iint_{\Omega_A \times \R^3} v_3 \partial_x g_\sigma   g_\sigma \d v \d x  +  \iint_{\Omega_A \times \R^3} \mathcal{L}_\hbar g_\sigma  g_\sigma \d v \d x \\
			+ \iint_{\Omega_A \times \R^3} (D_\hbar g-\hbar \sigma_x v_3 g_\sigma) g_\sigma \d v \d x = \iint_{\Omega_A \times \R^3} h_\sigma g_\sigma \d v \d x \,.
	\end{multline}
	The boundary condition in \eqref{A3-lambda} reduces to
		\begin{align}\label{BC-L2.}
			\no  \iint_{\Omega_A \times \R^3} v_3 \partial_x g_\sigma  g_\sigma \d v \d x \geq & \tfrac{1}{2} \| |v_3|^\frac{1}{2}  g_\sigma \|^2_{L^2_{\Sigma_+^0}} + \tfrac{1}{2} \| |v_3|^\frac{1}{2}  g_\sigma \|^2_{L^2_{\Sigma_+^A}} \\
			& - C [ \mathscr{B}_2 ( \varphi_{A, \sigma} ) ]^2 - C [ \mathscr{C}_2 ( f_{b, \sigma} ) ]^2 \,.
		\end{align}
	Moreover, by Lemma 2.2 of \cite{Chen-Liu-Yang-2004-AA} (or Lemma 2.4 of \cite{Wang-Yang-Yang-2007-JMP}), it infers that
	\begin{equation*}
			 \iint_{\Omega_A \times \R^3} \L_\hbar g_\sigma  g_\sigma \d v \d x \geq   \mu_3 \| \nu^\frac{1}{2} \P^\perp  g_\sigma \|^2_A
	\end{equation*}
	for $\mu_3 > 0$.

	We then control the quantity $\iint_{\Omega_A \times \R^3} (D_\hbar g-\hbar \sigma_x v_3 g_\sigma) g_\sigma \d v \d x $.
	\begin{equation}
	\begin{aligned}
			& \iint_{\Omega_A \times \R^3} (D_\hbar g-\hbar \sigma_x v_3 g_\sigma) g_\sigma \d v \d x \\
			= &\iint_{\Omega_A \times \R^3} \big[ D_\hbar g-\bar{\alpha} (\delta x + l)^{-\Theta}  \P^+  v_3 g_\sigma -\bar{\beta} (\delta x + l)^{-\Theta} \P^0  g_\sigma \big] g_\sigma \d v \d x  \\
			 & +\iint_{\Omega_A \times \R^3} \big[ \bar{\alpha} (\delta x + l)^{-\Theta}  \P^+  v_3 g_\sigma +\bar{\beta} (\delta x + l)^{-\Theta}  \P^0  g_\sigma -\hbar \sigma_x v_3 g_\sigma \big] g_\sigma \d v \d x \\
			:=& \int_{\Omega_A} E + D \d x\,.
	\end{aligned}
	\end{equation}
	We further decompose $E$ into
	\begin{equation*}
		\begin{aligned}
				E =& \underbrace{\bar{\alpha}  (\delta x + l)^{- \Theta }\int_{\R^3} g_\sigma \big( e^{\hbar \sigma} \P^+( e^{-\hbar \sigma} v_3 g_\sigma) -  \P^+  v_3 g_\sigma \big) \d v}_{:=E_1}\\
				&+\underbrace{\bar{\beta} (\delta x + l)^{- \Theta }\int_{\R^3} g_\sigma \big( e^{\hbar \sigma} \P^0( e^{-\hbar \sigma} g_\sigma) -  \P^0  g_\sigma \big) \d v}_{:=E_2}\,.
		\end{aligned}
		\end{equation*}
	By using the fact that $\sigma (x , v ) = 5 (\delta x + l)^{\frac{ 2}{3-\gamma}} $ for $(\delta x + l)> 2(1+ |v - \u |)^{3-\gamma}$, we have for $l$ sufficiently large
	\begin{equation*}
		\begin{aligned}
			|E_1| \leq& C \hbar (\delta x + l)^{- \Theta }\int_{\R^3} g_\sigma \big( e^{\hbar \sigma} \P^+( e^{-\hbar \sigma} v_3 g_\sigma) -\P^+ v_3 \P g_\sigma \big) \d v \\
			&\qquad \qquad \qquad\qquad \qquad+C\hbar (\delta x + l)^{-\Theta} \|\P g_\sigma \|_{L^2_v}\|\P^\perp g_\sigma \|_{L^2_v} \\
			\leq&\hbar (\delta x + l)^{- \Theta } \big(\tfrac{c}{8}\|\P g_\sigma \|_{L^2_v}^2 +C \|\nu^{\frac{1}{2}}\P^\perp g_\sigma \|_{L^2_v}^2\big) \\
			& + C \hbar (\delta x + l)^{- \Theta }\int_{\delta x + l \leq 2 (1 +|v-\u|)^{3-\gamma}} g_\sigma \big( e^{\hbar \sigma} \P^+( e^{-\hbar \sigma} v_3 g_\sigma) -  \P^+  v_3 \P g_\sigma \big) \d v \\
			\leq& \hbar (\delta x + l)^{- \Theta } \big(\tfrac{ c }{4}\|\P g_\sigma \|_{L^2_v}^2 +C \|\nu^{\frac{1}{2}}\P^\perp g_\sigma \|_{L^2_v}^2\big) \,,
		\end{aligned}
	\end{equation*}
	for some $c>0$ to be chosen later.	Similarly, it holds that
	\begin{equation*}
		|E_2|\leq \hbar (\delta x + l)^{- \Theta } \big(\tfrac{c}{4}\|\P g_\sigma \|_{L^2_v}^2 +C \|\nu^{\frac{1}{2}}\P^\perp g_\sigma \|_{L^2_v}^2\big) \,. \\
	\end{equation*}
	As a consequence,
	\begin{equation}\label{E-bnd}
		|E|\leq \hbar (\delta x + l)^{- \Theta } \big(\tfrac{c}{2}\|\P g_\sigma \|_{L^2_v}^2 +C \|\nu^{\frac{1}{2}}\P^\perp g_\sigma \|_{L^2_v}^2\big) \,.\\
	\end{equation}

    Next, we control the term $D$. By using Lemma \ref{Lmm-macro-diss} and the fact that $\psi_j$ decays like $e^{-c|v|^2}$, we have
    \begin{equation}
		\begin{aligned}
			D=& \partial_x (\delta x + l)^{\frac{2}{3-\gamma}}\int_{\delta x + l \geq 2 (1 +|v-\u|)^{3-\gamma}} g_\sigma \big( \tfrac{ \bar{\alpha} (3-\gamma)}{2} \P^+ v_3 \P g_\sigma+\tfrac{\bar{\beta}(3-\gamma)}{2} \P^0 \P g_\sigma - 5 \hbar \P v_3 \P g_\sigma \big) \d v\\
			&+(\delta x + l)^{-\Theta}\int_{\delta x + l \leq 2 (1 +|v-\u|)^{3-\gamma}} g_\sigma \big( \bar{\alpha}  \P^+  v_3 \P g_\sigma + \bar{\beta}  \P^0   \P g_\sigma -g_\sigma   \hbar \sigma_x \P v_3 \P g_\sigma \big) \d v \\
			&-\hbar (\delta x + l)^{-\Theta} \int_{\delta x + l \leq 2 (1 +|v-\u|)^{3-\gamma}} g_\sigma \sigma_x( \P v_3 \P^\perp g_\sigma  +\P^\perp v_3 g_\sigma )\d v \\
			\geq & c\hbar (\delta x + l)^{-\Theta} \|\P g_\sigma \|^2_{L^2_v} - c_1 \hbar e^{-c(\delta x + l)^{-\frac{2}{3-\gamma}}} \sum_{j=0}^{4} a_j^2 -C \hbar \|\nu^{\frac{1}{2}} \P^\perp g_\sigma \|^2_{L^2_v}\\
			\geq& c \hbar (\delta x + l)^{-\Theta} \|\P g_\sigma \|^2_{L^2_v} -C \hbar\|\nu^{\frac{1}{2}}\P^\perp g_\sigma \|^2_{L^2_v}\,
 		\end{aligned}
	\end{equation}
	for some $c>0$. As a result,
	\begin{equation}\label{Wg-bnd.}
		\iint_{\Omega_A \times \R^3}(D_\hbar g - \hbar \sigma_x v_3 g_\sigma )g_\sigma \d v \d x \geq c\hbar \|(\delta x + l)^{-\Theta} \P g_\sigma \|^2_A -C \hbar\|\nu^{\frac{1}{2}}\P^\perp g_\sigma \|^2_A\,.
	\end{equation}
	We then control the quantity $\iint_{\Omega_A \times \R^3} h_\sigma g_\sigma \d v \d x$.
	\begin{equation}\label{h-bnd-2.}
		\begin{aligned}
			& \iint_{\Omega_A \times \R^3} h_\sigma g_\sigma \d v \d x \\
			= & \iint_{\Omega_A \times \R^3} \P  h_\sigma  \P  g_\sigma \d v \d x + \iint_{\Omega_A \times \R^3} \P^\perp  h_\sigma  \P^\perp  g_\sigma \d v \d x \\
			\leq & \tfrac{1}{2} c_0 \| (\delta x + l)^{- \frac{\Theta}{2}} \P  g_\sigma \|^2_A + C \| (\delta x + l)^{ \frac{\Theta}{2}} \P  h_\sigma \|^2_A \\
			& + \tfrac{1}{2} c_0 \| \nu^\frac{1}{2} \P^\perp  g_\sigma \|^2_A + C \| \nu^{- \frac{1}{2}} \P^\perp  h_\sigma \|^2_A \\
			= & \tfrac{1}{2} c_0  [ \mathscr{E}_{2 }^A (g_\sigma) ]^2 + C  [ \mathscr{A}_{2 }^A (h_\sigma) ]^2 \,.
		\end{aligned}
	\end{equation}
	where the functional $ \mathscr{A}_{2 }^A ( \cdot ) $ is defined in \eqref{As-def}.

	Collecting the all above estimates \eqref{BC-L2.}, \eqref{Wg-bnd.} and \eqref{h-bnd-2.} and using \eqref{L2-g-sigma}, one  concludes the bound \eqref{L2-g-sigma-Bnd.}. Moreover, as similar arguments as in \eqref{L2-g-sigma-Bnd.}, the difference ${\vartriangle} g_\sigma : = g_{\sigma 2} - g_{\sigma 1}$ enjoys the estimate \eqref{L2-g-sigma-Bnd-Dif.}. Therefore, the proof of Lemma \ref{Lmm-L2xv.} is completed.
\end{proof}

\subsection{Close the uniform estimates: Proof of Lemma \ref{Lmm-APE-A3}}\label{Subsec:Uniform-Est}

In this subsection, based on Lemma \ref{Lmm-Y-bnd}-\ref{Lmm-L2xv.}, we will close the required uniform a priori estimate in Lemma \ref{Lmm-APE-A3}.

First, it follows from the inequalities \eqref{LA-g-sigma} in Lemma \ref{Lmm-Y-bnd} and \eqref{Y-bnd-4} in Lemma \ref{Lmm-Linfty-L2} that
\begin{equation}\label{X1}
	\begin{aligned}
		\mathscr{E}^A_\infty ( g_\sigma ) + \LL g_\sigma \RR_{  \beta, \vartheta, \Sigma} \lesssim \mathscr{E}^A_{\mathtt{cro}} (g_\sigma) + \mathscr{A}_\infty^A ( h_\sigma ) + \mathscr{B}_\infty ( \varphi_{A, \sigma} ) +\mathscr{C}_\infty (f_{b,\sigma}) \,.
	\end{aligned}
\end{equation}
Together with \eqref{C-L2} in Lemma \ref{Lmm-L2xv-alpha} and \eqref{X1}, one then gains
\begin{equation}\label{X2}
	\begin{aligned}
		& \mathscr{E}^A_\infty ( g_\sigma ) + \LL g_\sigma \RR_{  \beta, \vartheta, \Sigma} + \mathscr{E}^A_{\mathtt{cro}} (g_\sigma) \\
		\lesssim &\|(\delta x + l)^{-\frac{\Theta}{2}}  \nu^{\frac{1}{2}}  w_{\beta + \beta_\gamma, \vartheta} g_\sigma \|_A+ \mathscr{A}^A_{ \mathtt{cro} } ( h_\sigma ) + \mathscr{B}_{ \mathtt{cro} } ( \varphi_{A, \sigma} ) + \mathscr{C}_{ \mathtt{cro} } ( f_{b, \sigma} )\\
		& + \mathscr{A}_\infty^A ( h_\sigma ) + \mathscr{B}_\infty ( \varphi_{A, \sigma} ) + \mathscr{C}_{\infty } ( f_{b, \sigma} )\,.
	\end{aligned}
\end{equation}
Observe that $ \|(\delta x + l)^{-\frac{\Theta}{2}}  \nu^{\frac{1}{2}}  w_{\beta + \beta_\gamma, \vartheta} g_\sigma \|_A \lesssim \mathscr{E}^A_{ 2} ( w_{\beta + \beta_\gamma , \vartheta} g_\sigma) $. Combining with the bounds \eqref{X2} and \eqref{L2-g-sigma-Bnd} in Lemma \ref{Lmm-L2xv}, one easily sees that
\begin{equation}\label{X3}
	\begin{aligned}
		& \mathscr{E}^A_\infty ( g_\sigma ) + \LL g_\sigma \RR_{  \beta, \vartheta, \Sigma} + \mathscr{E}^A_{\mathtt{cro}} (g_\sigma) + \mathscr{E}^A_{2} (w_{\beta + \beta_\gamma , \vartheta} g_\sigma) \\
		\lesssim &  \| (\delta x + l)^{-\frac{\Theta}{2}} \P w_{\beta + \beta_\gamma, \vartheta}  g_\sigma \|^2_A
		+ \mathscr{A}^A_{2} (w_{\beta + \beta_\gamma, \vartheta} h_\sigma) +\mathscr{B}_2(w_{\beta + \beta_\gamma, \vartheta} \varphi_{A,\sigma}) \\
		& + \mathscr{C}_2 (w_{\beta + \beta_\gamma, \vartheta} f_{b,\sigma})
		 + \mathscr{A}^A_{ \mathtt{cro} } ( h_\sigma ) + \mathscr{B}_{ \mathtt{cro} } ( \varphi_{A, \sigma} )+ \mathscr{C}_{ \mathtt{cro} } ( f_{b, \sigma} )\\
		&  + \mathscr{A}_\infty^A ( h_\sigma ) + \mathscr{B}_\infty ( \varphi_{A, \sigma} )+ \mathscr{C}_{\infty } ( f_{b, \sigma} ) \,.
	\end{aligned}
\end{equation}
Furthermore, the estimates \eqref{P-wg} in Lemma \ref{Lmm-P-wg} show that $\| (\delta x + l)^{-\frac{\Theta}{2}} \P w_{\beta + \beta_\gamma, \vartheta}  g_\sigma \|^2_A \lesssim [\mathscr{E}^A_{2} (g_\sigma )]^2$. Then, combining the estimates \eqref{X3} and \eqref{L2-g-sigma-Bnd.} in Lemma \ref{Lmm-L2xv}, one establishes
\begin{equation}
	\begin{aligned}\label{X4}
		& \mathscr{E}^A_\infty ( g_\sigma ) + \LL g_\sigma \RR_{  \beta, \vartheta, \Sigma} + \mathscr{E}^A_{\mathtt{cro}} (g_\sigma) + \mathscr{E}^A_{ 2} (g_\sigma) + \mathscr{E}^A_{ 2} (w_{\beta + \beta_\gamma, \vartheta} g_\sigma) \\
		\lesssim &  \mathscr{A}^A_{2} (h_\sigma) + \mathscr{B}_2 ( \varphi_{A, \sigma} ) + \mathscr{C}_2 ( f_{b, \sigma} ) + \mathscr{A}^A_{2} (w_{\beta + \beta_\gamma, \vartheta} h_\sigma) \\
		& +\mathscr{B}_2(w_{\beta + \beta_\gamma, \vartheta} \varphi_{A,\sigma}) + \mathscr{C}_2 (w_{\beta + \beta_\gamma, \vartheta} f_{b,\sigma})
		 + \mathscr{A}^A_{ \mathtt{cro} } ( h_\sigma ) + \mathscr{B}_{ \mathtt{cro} } ( \varphi_{A, \sigma} )+ \mathscr{C}_{ \mathtt{cro} } ( f_{b, \sigma} )\\
		&  + \mathscr{A}_\infty^A ( h_\sigma ) + \mathscr{B}_\infty ( \varphi_{A, \sigma} )+ \mathscr{C}_{\infty } ( f_{b, \sigma} )\\
		= & \mathscr{A}^A (h_\sigma) + \mathscr{B} ( \varphi_{A, \sigma} ) + \mathscr{C} ( f_{b, \sigma} )   \,,
	\end{aligned}
\end{equation}
where the functionals $ \mathscr{A}^A ( \cdot ) $ and $ \mathscr{B} ( \cdot ) $ are defined in \eqref{Ah-lambda}. Thereby \eqref{X4} concludes the uniform estimate \eqref{Apriori-bnd}. Moreover, by the similar arguments in \eqref{X4}, we can also easily prove the bound \eqref{Apriori-bnd-diff} for the difference ${\vartriangle} g_\sigma$. Therefore, the proof of Lemma \ref{Lmm-APE-A3} is finished.

%%%%%%%%%%%%%%%%%%%%%%%%%%%%%%%%%%%%%%%%%%%%%%%%%%%%%%%%%%%%%%%%%%%%%%%%%%%%%%%%%%%%%%%
%%%%%%%%%%%%%%%%%%%%%%%%%%%%%%New Section%%%%%%%%%%%%%%%%%%%%%%%%%%%%%%%%%%%%%%%%%%%%%% %%%%%%%%%%%%%%%%%%%%%%%%%%%%%%%%%%%%%%%%%%%%%%%%%%%%%%%%%%%%%%%%%%%%%%%%%%%%%%%%%%%%%%%

\section{Existence of linear problem \eqref{KL}: Proof of Theorem \ref{Thm-Linear}}\label{Sec:EKLe}

\subsection{Existence of $\eqref{A1}$}

In this subsection, based on Lemma \ref{Lmm-APE-A3}, we will prove the existence and uniqueness of \eqref{A1} (or equivalently \eqref{A3-lambda}). First we define the Banach space $\mathbb{B}_1$ as
\begin{equation*}
	\begin{aligned}
		\mathbb{B}_1 : = \big\{ f = f(x,v) :  \mathscr{E}^A (f) < \infty \big\} \,,
	\end{aligned}
\end{equation*}
which are endowed with the norm
\begin{equation}\label{Bi}
	\begin{aligned}
		\| f \|_{\mathbb{B}_1} : = \mathscr{E}^A (f) \,.
	\end{aligned}
\end{equation}
 Here the functional $\mathscr{E}^A (f)$ is defined in \eqref{Eg-lambda}.

\begin{lemma}\label{Lmm-A1}
	Let $- 3 < \gamma \leq 1$, sufficiently small $\delta, \hbar, \vartheta > 0$, large enough $A, l > 1$, integer $\beta \geq \max \{ 0, - \gamma \}$ and $0 < \alpha < \mu_\gamma$, where $\mu_\gamma > 0$ is given in Lemma \ref{Lmm-Kh-L2}. Assume that the source term $h (x,v) = e^{- \hbar \sigma (x,v)} h_\sigma (x,v)$ and the boundary source term $f_b (v) = e^{ - \hbar \sigma (0, v) } f_{b, \sigma} (v)$ $\varphi_A (v) = e^{ - \hbar \sigma (A, v) } \varphi_{A, \sigma} (v)$ satisfy \eqref{Assmp-hsigma}, i.e.,
	\begin{equation*}
		\begin{aligned}
			\mathscr{A}^A (h_\sigma) \,, \mathscr{B} ( \varphi_{A, \sigma} )\,,\mathscr{C} (f_{b,\sigma}) < \infty \,.
		\end{aligned}
	\end{equation*}
	Then the problem \eqref{A1} admits a unique mild solution $g = g (x,v)$ satisfying $g_\sigma (x,v) = g (x, v) e^{ \hbar \sigma (x,v)} \in \mathbb{B}_1 $ with the bound
	\begin{equation}\label{A2-bnd}
		\begin{aligned}
			\mathscr{E}^A ( g_\sigma ) \leq C \big( \mathscr{A}^A (h_\sigma) + \mathscr{B} ( \varphi_{A, \sigma} )+ \mathscr{C} ( f_{b, \sigma} ) \big) \,.
		\end{aligned}
	\end{equation}
    Here the constant $C > 0$ is independent of $A$, $\delta$, but it depends on $\hbar$;  all functionals are defined in Subsection \ref{Subsubsec:EF}.
\end{lemma}

\begin{proof}[Proof of Lemma \ref{Lmm-A1}]
    In fact, similar to the proof of Theorem 3.2 in \cite{Chen-Liu-Yang-2004-AA}, using the Hahn-Banach theorem and the Riesz representation theorem, there exists a unique weak solution of the linear problem \eqref{A3-lambda} satisfying
	\begin{equation*}\label{L2-g-sigma-Bnd..}
		\begin{aligned}
			&   \| |v_3|^\frac{1}{2}  g_\sigma \|^2_{L^2_{\Sigma_+^0}} + \| |v_3|^\frac{1}{2}  g_\sigma \|^2_{L^2_{\Sigma_+^A}} + [ \mathscr{E}^A_{2 } ( g_\sigma ) ]^2 \\
			\leq &  C [ \mathscr{A}^A_{ 2} (h_\sigma) ]^2 + C [ \mathscr{B}_2 (\varphi_{A, \sigma}) ]^2+ C [ \mathscr{C}_2 (f_{b, \sigma}) ]^2 \,.
		\end{aligned}
	\end{equation*}
	 Moreover, by the uniqueness of the weak solution and the a priori estimate \eqref{Apriori-bnd}, one establishes the uniqueness of the mild solution and the estimate \eqref{A2-bnd}. This completes the proof of Lemma \ref{Lmm-A1}.
\end{proof}

\subsection{Limits from \eqref{A1} to \eqref{KL-Damped}}

In this section, based on Lemma \ref{Lmm-A1}, we will construct the existence of the problem \eqref{KL-Damped} by taking limit $A \to + \infty$ in the equation \eqref{A1}.

For small $\hbar > 0$ given in Lemma \ref{Lmm-A1}, we initially define a Banach space
\begin{equation}\label{Bh-1}
	\begin{aligned}
		\mathbb{B}_\infty^\hbar : = \{ f (x,v) ; \mathscr{E}^\infty ( e^{\hbar \sigma} f ) < \infty \}
	\end{aligned}
\end{equation}
with the norm
\begin{equation}\label{Bh-2}
	\begin{aligned}
		\| f \|_{ \mathbb{B}_\infty^\hbar } = \mathscr{E}^\infty ( e^{\hbar \sigma} f ) \,,
	\end{aligned}
\end{equation}
where the functional $\mathscr{E}^\infty (\cdot ) $ is given in \eqref{Eg-lambda} with $A = \infty$.

Now we state the existence result on the problem \eqref{KL-Damped} as follows.

\begin{lemma}\label{Lmm-KLe}
	Let $- 3 < \gamma \leq 1$, sufficiently small $\delta, \hbar, \vartheta > 0$, large enough $l > 1$, $0 < \alpha < \mu_\gamma$ and integer $\beta \geq \beta_\gamma + \frac{1 - \gamma}{2} + \max \{ 0, - \gamma \}$, where $\mu_\gamma > 0$ is given in Lemma \ref{Lmm-Kh-L2}. Assume that the source term $h (x,v) $ and the boundary source term $f_b(v)$ satisfies
	\begin{equation}\label{Asmp-h}
		\begin{aligned}
			\mathscr{A}^\infty ( e^{\hbar \sigma} h ) < \infty \,, \mathscr{C} ( e^{\hbar \sigma} f_{b} ) < \infty\,.
		\end{aligned}
	\end{equation}
    where the functional $\mathscr{A}^\infty (\cdot)$ is defined in \eqref{Ah-lambda} with $A = \infty$. Then the problem \eqref{KL-Damped} admits a unique solution $g = g (x,v) \in \mathbb{B}_\infty^\hbar$ such that
    \begin{equation}
    	\begin{aligned}
    		\| g \|_{ \mathbb{B}_\infty^\hbar } \leq C (\mathscr{A}^\infty ( e^{\hbar \sigma} h ) + \mathscr{C} ( e^{\hbar \sigma} f_{b} ))
    	\end{aligned}
    \end{equation}
    for some constant $C > 0$.
\end{lemma}

\begin{proof}[Proof of Lemma \ref{Lmm-KLe}]
We now give a family of functions $\{ \varphi_A (v) \}_{A > 1}$ satisfying
\begin{equation}\label{Asmp-phi}
	\begin{aligned}
		\sup_{A > 1} \mathscr{B} ( e^{\hbar \sigma} \varphi_A ) < \infty \,,
	\end{aligned}
\end{equation}
where the functional $\mathscr{B} (\cdot)$ is defined in \eqref{Ah-lambda}. The assumptions \eqref{Asmp-phi} and \eqref{Asmp-h} tell us that for any $1 < A < \infty$,
\begin{equation}
	\begin{aligned}
		\mathscr{A}^A ( e^{\hbar \sigma} h ) + \mathscr{B} ( e^{\hbar \sigma} \varphi_A ) + \mathscr{C} ( e^{\hbar \sigma} f_{b} )  \leq \mathfrak{C}_\flat : = \mathscr{A}^\infty ( e^{\hbar \sigma} h ) + \sup_{A > 1} \mathscr{B} ( e^{\hbar \sigma} \varphi_A ) + \mathscr{C} ( e^{\hbar \sigma} f_{b} ) < \infty.
	\end{aligned}
\end{equation}
Then Lemma \ref{Lmm-A1} indicates that there is a unique solution $g^A = g^A (x,v)$, $(x,v) \in (0, A) \times \R^3$ satisfying
\begin{equation}
	\begin{aligned}
		\mathscr{E}^A ( e^{ \hbar \sigma } g^A ) \leq C ( \mathscr{A}^A ( e^{\hbar \sigma} h ) + \mathscr{B} ( e^{\hbar \sigma} \varphi_A )+ \mathscr{C} ( e^{\hbar \sigma} f_{b} )  ) \leq C \mathfrak{C}_\flat \,,
	\end{aligned}
\end{equation}
where $C > 0$ is independent of $A > 1$ and the functional $\mathscr{E}^A (\cdot)$ is given in \eqref{Eg-lambda}.

We now extend $ g^A (x,v) $ as follows
\begin{equation}\label{Extnd-g}
	\begin{aligned}
		\tilde{g}^A (x,v) = \mathbf{1}_{x \in (0, A)} g^A (x,v) \,, (x,v) \in (0, \infty) \times \R^3 \,.
	\end{aligned}
\end{equation}
It is easy to see that $\tilde{g}^A (x,v) \in \mathbb{B}_\infty^\hbar$ with
\begin{equation}\label{gA-h-bnd}
	\begin{aligned}
		\| \tilde{g}^A \|_{ \mathbb{B}_\infty^\hbar } = \mathscr{E}^\infty ( e^{\hbar \sigma} \tilde{g}^A ) \leq C \mathfrak{C}_\flat \,.
	\end{aligned}
\end{equation}
Then there is a $g' = g' (x,v) \in \mathbb{B}_\infty^\hbar $ such that
\begin{equation}\label{Cnv-g1}
	\begin{aligned}
		\tilde{g}^A (x,v) \to g' (x,v) \quad \textrm{weakly in } \mathbb{B}_\infty^\hbar
	\end{aligned}
\end{equation}
as $A \to + \infty$ (in the sense of subsequence, if necessary). Moreover, $g' (x,v)$ obeys
\begin{equation}\label{g-prime-bnd}
	\begin{aligned}
		\| g' \|_{ \mathbb{B}_\infty^\hbar } = \mathscr{E}^\infty ( e^{\hbar \sigma} g' ) \leq C \mathfrak{C}_\flat \,.
	\end{aligned}
\end{equation}

For any $1 < A_1 < A_2 < + \infty$, let $g^{A_i} (x, v)$ $(i = 1,2)$ be the solution to the problem
\begin{equation*}
	\begin{aligned}
		\left\{
		  \begin{aligned}
		  	& v_3 \partial_x g^{A_i} + \L g^{A_i} = h - \bar{\alpha} ( \delta x + l)^{-\frac{\Theta}{2}} \P^+ v_3 g^{A_i} -\bar{\beta} ( \delta x + l)^{-\frac{\Theta}{2}} \P^0 g^{A_i}  \,, \quad x \in (0, A_i) \,, \\
		  	& g^{A_i} (0, v) |_{v_3 > 0} = f_b \,, \\
		  	& g^{A_i} (A_i, v) |_{v_3 < 0} = \varphi_{A_i} (v) \,,
		  \end{aligned}
		\right.
	\end{aligned}
\end{equation*}
constructed in Lemma \ref{Lmm-A1}. Then $g^{A_i} (x, v)$ $(i = 1,2)$ subject to the estimates
\begin{equation}\label{gAi-bnd}
	\begin{aligned}
		\mathscr{E}^{A_i} ( e^{ \hbar \sigma } g^{A_i} ) \leq C \big( \mathscr{A}^{A_i} (e^{\hbar \sigma} h ) + \mathscr{B} ( \varphi_{A_i, \sigma} ) + \mathscr{C} ( f_{b,\sigma}) \big) \leq C \mathfrak{C}_\flat \,.
	\end{aligned}
\end{equation}
Let $\tilde{g}^{A_i}$ be the extension of $g^{A_i}$ as given in \eqref{Extnd-g}. Then
\begin{equation}\label{E1-bnd1}
	\begin{aligned}
		( \tilde{g}^{A_1} - \tilde{g}^{A_2} ) (x,v) = \mathbf{1}_{x \in (0, A_1)} ( g^{A_1} - g^{A_2} ) (x,v) - \mathbf{1}_{x \in [A_1, A_2) } g^{A_2} (x,v) \,.
	\end{aligned}
\end{equation}

It is easy to see that $g^{A_1} - g^{A_2}$ obeys
\begin{equation*}
	\begin{aligned}
		\left\{
		  \begin{aligned}
		  	& v_3 \partial_x ( \tilde{g}^{A_1} - \tilde{g}^{A_2} ) + \L( \tilde{g}^{A_1} - \tilde{g}^{A_2} ) =- \bar{\alpha} ( \delta x + l)^{-\frac{\Theta}{2}} \P^+ v_3 (g^{A_2}- g^{A_1}) \\
			&\qquad\qquad\qquad\qquad\qquad\qquad\qquad\quad-\bar{\beta} ( \delta x + l)^{-\frac{\Theta}{2}} \P^0 (g^{A_2}- g^{A_1}) \,, \ 0 < x < A_1 \,, \\
		  	& ( \tilde{g}^{A_1} - \tilde{g}^{A_2} ) (0, v) |_{v_3 > 0} = 0 \,, \\
		  	& ( \tilde{g}^{A_1} - \tilde{g}^{A_2} ) (A_1, v) |_{v_3 < 0} = \varphi_{A_1} (v) - g^{A_2} (A_1, v) \,.
		  \end{aligned}
		\right.
	\end{aligned}
\end{equation*}
It thereby follows from the similar arguments in Lemma \ref{Lmm-A1} that for any fixed $\hbar' \in (0, \hbar)$,
\begin{equation}\label{EA2A1-bnd1}
	\begin{aligned}
		\mathscr{E}^{A_1} ( e^{\hbar' \sigma} ( g^{A_1} - g^{A_2} ) ) \leq & C \mathscr{B} ( e^{\hbar' \sigma (A_1, v) } (\varphi_{A_1} (v) - g^{A_2} (A_1, v) ) ) \\
		\leq & C \mathscr{B} ( e^{\hbar' \sigma (A_1, v) } \varphi_{A_1} (v) ) + C \mathscr{B} ( e^{\hbar' \sigma (A_1, v) } g^{A_2} (A_1, v) ) \,.
	\end{aligned}
\end{equation}

By the definition of $\mathscr{B} (\cdot)$ given in \eqref{Ah-lambda}, one has
\begin{equation*}
	\begin{aligned}
		\mathscr{B} ( e^{\hbar' \sigma (A_1, v) } \varphi_{A_1} (v) ) \leq \sup_{v \in \R^3} e^{ - (\hbar - \hbar') \sigma (A_1, v) } \mathscr{B} ( e^{\hbar \sigma (A_1, v) } \varphi_{A_1} (v) ) \leq \mathfrak{C}_\flat \sup_{v \in \R^3} e^{ - (\hbar - \hbar') \sigma (A_1, v) } \,.
	\end{aligned}
\end{equation*}
Lemma \ref{Lmm-sigma} shows that $ \sigma (A_1, v) \geq c (\delta A_1 + l)^\frac{2}{3 - \gamma} $ uniformly in $v \in \R^3$, which indicates that
\begin{equation}\label{B-1}
	\begin{aligned}
		\sup_{v \in \R^3} e^{ - (\hbar - \hbar') \sigma (A_1, v) } \leq e^{ - c (\hbar - \hbar') (\delta A_1 + l)^\frac{2}{3 - \gamma} } \,.
	\end{aligned}
\end{equation}
Then
\begin{equation}\label{EA1A2-bnd2}
	\begin{aligned}
		\mathscr{B} ( e^{\hbar' \sigma (A_1, v) } \varphi_{A_1} (v) ) \leq  \mathfrak{C}_\flat e^{ - c (\hbar - \hbar') (\delta A_1 + l)^\frac{2}{3 - \gamma} } \,.
	\end{aligned}
\end{equation}

In the following, we focus on controlling the quantity $\mathscr{B} ( e^{\hbar' \sigma (A_1, v) } g^{A_2} (A_1, v) )$. Recalling \eqref{Ah-lambda}, one has
\begin{equation}\label{Bbnd-0}
	\begin{aligned}
		\mathscr{B} (\cdot) = \mathscr{B}_\infty (\cdot) + \mathscr{B}_{\mathtt{cro}} (\cdot) + \mathscr{B}_2 (w_{\beta+\beta_\gamma,\vartheta}\cdot) + \mathscr{B}_2 (\cdot) \,.
	\end{aligned}
\end{equation}

\underline{\em Case 1. Control of $ \mathscr{B}_\infty (e^{\hbar' \sigma (A_1, v) } g^{A_2} (A_1, v)) $.}

By the definition of $\mathscr{B}_\infty (\cdot)$ in \eqref{B-def}, one has
\begin{equation}\label{Bbnd-1}
	\begin{aligned}
		& \mathscr{B}_\infty (e^{\hbar' \sigma (A_1, v) } g^{A_2} (A_1, v)) \\
		= & \| e^{\hbar' \sigma (A_1, v) } \sigma_x^ \frac{1}{2} (A_1, v)  w_{\beta, \vartheta} (v) g^{A_2} (A_1, v) \|_{L^\infty_v} \\
		\leq & \sup_{v \in \R^3} e^{ - (\hbar - \hbar') \sigma (A_1, v) } \| e^{\hbar \sigma (A_1, v) } \sigma_x^ \frac{1}{2} (A_1, v)  w_{\beta, \vartheta} (v) g^{A_2} (A_1, v) \|_{L^\infty_v} \\
		\leq & e^{ - c (\hbar - \hbar') (\delta A_1 + l)^\frac{2}{3 - \gamma} } \LL e^{\hbar \sigma} g^{A_2} \RR_{A_2;   \beta, \vartheta} \\
		= & e^{ - c (\hbar - \hbar') (\delta A_1 + l)^\frac{2}{3 - \gamma} } \mathscr{E}_\infty^{A_2} ( e^{\hbar \sigma} g^{A_2} ) \,,
	\end{aligned}
\end{equation}
where the inequality \eqref{B-1} is utilized, and $\mathscr{E}_\infty^{A_2} (\cdot)$ is defined in \eqref{E-infty}.

\underline{\em Case 2. Control of $ \mathscr{B}_2 ( w_{\beta + \beta_\gamma, \vartheta} e^{\hbar' \sigma (A_1, v) } g^{A_2} (A_1, v)) + \mathscr{B}_2 (e^{\hbar' \sigma (A_1, v) } g^{A_2} (A_1, v)) $.}

By the definition of $\mathscr{B}_2 (\cdot)$ in \eqref{B-def}, one has
\begin{equation*}
	\mathscr{B}_2 (e^{\hbar' \sigma (A_1, v) } g^{A_2} (A_1, v)) \leq
	\mathscr{B}_2 ( w_{\beta + \beta_\gamma, \vartheta} e^{\hbar' \sigma (A_1, v) } g^{A_2} (A_1, v)) \,.
\end{equation*}
Then, it is sufficient to control  the quantity $\mathscr{B}_2 ( w_{\beta + \beta_\gamma, \vartheta} e^{\hbar' \sigma (A_1, v) } g^{A_2} (A_1, v))$. One has
	\begin{align*}
		& \mathscr{B}_2 ( w_{\beta + \beta_\gamma, \vartheta} e^{\hbar' \sigma (A_1, v) } g^{A_2} (A_1, v)) \\
		= & \| |v_3|^\frac{1}{2} w_{\beta + \beta_\gamma, \vartheta} e^{\hbar' \sigma (A_1, v) } g^{A_2} (A_1, v) \|_{L^2_{\Sigma_+^{A_1}}} \\
		= & \Big( \int_{v_3 > 0} |v_3| w_{\beta + \beta_\gamma, \vartheta}^2 e^{2 \hbar' \sigma (A_1, v) } | g^{A_2} (A_1, v) |^2 \d v \Big)^\frac{1}{2} \\
		= & \Big( \int_{v_3 > 0} \Phi (A_1, v) | z_{\alpha'} (v) \sigma_x^ \frac{1}{2} (A_1, v)  w_{\beta, \vartheta} (v) e^{ \hbar \sigma (A_1, v) } g^{A_2} (A_1, v) |^2 \d v \Big)^\frac{1}{2} \,,
	\end{align*}
where $\alpha' \in (\tfrac{1}{2} - \mu_\gamma, \frac{1}{2})$ is given in Lemma \ref{Lmm-Y-bnd}, and
\begin{equation*}
	\begin{aligned}
		\Phi (A_1, v) = |v_3| z_{- \alpha'}^2 (v) (1 + |v|)^{2 \beta_\gamma} \sigma_x^{- 1} (A_1, v)e^{ - 2 (\hbar - \hbar') \sigma (A_1, v) } \,.
	\end{aligned}
\end{equation*}
Lemma \ref{Lmm-sigma} implies that
\begin{equation}\label{B-2}
	\begin{aligned}
		\sigma_x^{- 1} (A_1, v) \leq C (\delta A_1 + l)^{ \Theta } (1 + |v - \u|)^{ 1 - \gamma } \,.
	\end{aligned}
\end{equation}
It is also easy to verify that $ \sigma (A_1, v) \geq c' (1 + |v - \u|^2) $ for some constant $c' > 0$, which means that
\begin{equation}\label{B-3}
	\begin{aligned}
		e^{ - (\hbar- \hbar') \sigma (A_1, v) } \leq e^{ - c' (\hbar - \hbar') } e^{ - c' (\hbar - \hbar') |v - \u|^2 } \leq e^{ - c' (\hbar - \hbar') |v - \u|^2 } \,.
	\end{aligned}
\end{equation}
Then the bounds \eqref{B-1}, \eqref{B-2} and \eqref{B-3} show that
\begin{equation*}
	\begin{aligned}
		\Phi (A_1, v) \leq & C |v_3| z_{- \alpha'}^2 (v) (1 + |v|)^{2 \beta_\gamma + (1 - \gamma)} e^{ - c' (\hbar - \hbar') |v - \u|^2 } \\
		& \times (\delta A_1 + l)^{\Theta } e^{ - c (\hbar - \hbar') (\delta A_1 + l)^\frac{2}{3 - \gamma} } \\
		\leq & C e^{ - \frac{c}{2} (\hbar - \hbar') (\delta A_1 + l)^\frac{2}{3 - \gamma} }
	\end{aligned}
\end{equation*}
uniformly in $v \in \R^3$. It therefore follows that
\begin{equation*}
	\begin{aligned}
		& \mathscr{B}_2 ( w_{\beta + \beta_\gamma, \vartheta} e^{\hbar' \sigma (A_1, v) } g^{A_2} (A_1, v)) \\
		\leq & C e^{ - \frac{c}{4} (\hbar - \hbar') (\delta A_1 + l)^\frac{2}{3 - \gamma} } \Big( \int_{v_3 > 0} | z_{\alpha'} (v) \sigma_x^ \frac{1}{2} (A_1, v)  w_{\beta, \vartheta} (v) e^{ \hbar \sigma (A_1, v) } g^{A_2} (A_1, v) |^2 \d v \Big)^\frac{1}{2} \\
		\leq & C e^{ - \frac{c}{4} (\hbar - \hbar') (\delta A_1 + l)^\frac{2}{3 - \gamma} } \| z_{\alpha'} \sigma_x^ \frac{1}{2}   w_{\beta, \vartheta} e^{\hbar \sigma} g^{A_2} \|_{L^\infty_x L^2_v} \,.
	\end{aligned}
\end{equation*}
From Lemma \ref{Lmm-Linfty-L2}, it is derived that
\begin{equation*}
	\begin{aligned}
		& \| z_{\alpha'} \sigma_x^ \frac{1}{2}   w_{\beta, \vartheta} e^{\hbar \sigma} g^{A_2} \|_{L^\infty_x L^2_v} \\
		\leq & C \| \nu^{- \frac{1}{2}} z_{- \alpha} \sigma_x^ \frac{1}{2}   z_1 w_{\beta + \beta_\gamma, \vartheta} \partial_x ( e^{\hbar \sigma} g^{A_2} ) \|_{A_2} + C \| \nu^{ \frac{1}{2}} z_{- \alpha} \sigma_x^ \frac{1}{2}   w_{\beta + \beta_\gamma, \vartheta} e^{\hbar \sigma} g^{A_2} \|_{A_2} \\
		= & C \mathscr{E}_{\mathtt{cro}}^{A_2} ( e^{\hbar \sigma} g^{A_2} ) \,,
	\end{aligned}
\end{equation*}
where the functional $\mathscr{E}_{\mathtt{cro}}^{A_2} ( \cdot ) $ is defined in \eqref{E-cro}. Then there holds
\begin{equation}\label{Bbnd-2}
	\begin{aligned}
		&\mathscr{B}_2 ( w_{\beta + \beta_\gamma, \vartheta} e^{\hbar' \sigma (A_1, v) } g^{A_2} (A_1, v)) + \mathscr{B}_2 (e^{\hbar' \sigma (A_1, v) } g^{A_2} (A_1, v)) \\
		\leq & C e^{ - \frac{c}{4} (\hbar - \hbar') (\delta A_1 + l)^\frac{2}{3 - \gamma} } \mathscr{E}_{\mathtt{cro}}^{A_2} ( e^{\hbar \sigma} g^{A_2} ) \,.
	\end{aligned}
\end{equation}

\underline{\em Case 3. Control of $ \mathscr{B}_{\mathtt{cro}} (e^{\hbar' \sigma (A_1, v) } g^{A_2} (A_1, v))$.}

By employing the similar arguments of controlling the quantity $\mathscr{B}_2 (e^{\hbar' \sigma (A_1, v) } g^{A_2} (A_1, v))$, one obtains
\begin{equation}\label{Bbnd-3}
	\begin{aligned}
		& \mathscr{B}_{\mathtt{cro}} (e^{\hbar' \sigma (A_1, v) } g^{A_2} (A_1, v))  \\
		\leq & C e^{ - \frac{c}{4} (\hbar - \hbar') (\delta A_1 + l)^\frac{2}{3 - \gamma} } \| z_{\alpha'} \sigma_x^ \frac{1}{2}   w_{\beta, \vartheta} e^{\hbar \sigma} g^{A_2} \|_{L^\infty_x L^2_v} \\
		\leq & C e^{ - \frac{c}{4} (\hbar - \hbar') (\delta A_1 + l)^\frac{2}{3 - \gamma} } \mathscr{E}_{\mathtt{cro}}^{A_2} ( e^{\hbar \sigma} g^{A_2} ) \,.
	\end{aligned}
\end{equation}

As a consequence, the relations \eqref{Bbnd-0}, \eqref{Bbnd-1}, \eqref{Bbnd-2} and \eqref{Bbnd-3} reduce to
\begin{equation}\label{EA1A2-bnd3}
	\begin{aligned}
		\mathscr{B} ( e^{\hbar' \sigma (A_1, v) } g^{A_2} ( A_1, v) ) \leq C e^{ - \frac{c}{4} (\hbar - \hbar') (\delta A_1 + l)^\frac{2}{3 - \gamma} } ( \mathscr{E}_\infty^{A_2} ( e^{\hbar \sigma} g^{A_2} ) + \mathscr{E}_{\mathtt{cro}}^{A_2} ( e^{\hbar \sigma} g^{A_2} ) ) \\
		\leq C e^{ - \frac{c}{4} (\hbar - \hbar') (\delta A_1 + l)^\frac{2}{3 - \gamma} } \mathscr{E}^{A_2} ( e^{\hbar \sigma} g^{A_2} ) \leq C \mathfrak{C}_\flat e^{ - \frac{c}{4} (\hbar - \hbar') (\delta A_1 + l)^\frac{2}{3 - \gamma} } \,,
	\end{aligned}
\end{equation}
where the last inequality is derived from the bound \eqref{gAi-bnd}. Collecting the estimates \eqref{EA2A1-bnd1}, \eqref{EA1A2-bnd2} and \eqref{EA1A2-bnd3}, one gains
\begin{equation}\label{E1-bnd2}
	\begin{aligned}
		\mathscr{E}^\infty ( \mathbf{1}_{x \in (0, A_1)} e^{\hbar' \sigma} ( g^{A_1} - g^{A_2} ) ) = \mathscr{E}^{A_1} ( e^{\hbar' \sigma} ( g^{A_1} - g^{A_2} ) ) \leq C \mathfrak{C}_\flat e^{ - \frac{c}{4} (\hbar - \hbar') (\delta A_1 + l)^\frac{2}{3 - \gamma} }
	\end{aligned}
\end{equation}
for any $\hbar' \in (0, \hbar)$.

Note that
\begin{equation*}
	\begin{aligned}
		\mathscr{E}^\infty ( \mathbf{1}_{x \in [ A_1, A_2 ) } e^{\hbar' \sigma} g^{A_2} ) = & \mathscr{E}^{A_2} ( \mathbf{1}_{x \in [ A_1, A_2 ) } e^{\hbar' \sigma} g^{A_2} ) \\
		\leq & \mathscr{E}^{A_2} ( e^{\hbar \sigma} g^{A_2} ) \sup_{x > 0, v \in \R^3} \mathbf{1}_{x \in [ A_1, A_2 ) } e^{ - (\hbar - \hbar') \sigma (x,v) } \\
		\leq & C \mathfrak{C}_\flat \sup_{x > 0, v \in \R^3} \mathbf{1}_{x \in [ A_1, A_2 ) } e^{ - (\hbar - \hbar') \sigma (x,v) } \,,
	\end{aligned}
\end{equation*}
where the last inequality is deduced from \eqref{gAi-bnd}. Since $\sigma (x,v) \geq c (\delta x + l)^\frac{2}{3 - \gamma}$ by Lemma \ref{Lmm-sigma}, one has $ \mathbf{1}_{x \in [ A_1, A_2 ) } e^{ - (\hbar - \hbar') \sigma (x,v) } \leq \mathbf{1}_{x \in [ A_1, A_2 ) } e^{ - c (\hbar - \hbar') (\delta x + l)^\frac{2}{3 - \gamma} } \leq e^{ - c (\hbar - \hbar') (\delta A_1 + l)^\frac{2}{3 - \gamma} } $ uniformly in $x$ and $v$. Then
\begin{equation}\label{E1-bnd3}
	\begin{aligned}
		\mathscr{E}^\infty ( \mathbf{1}_{x \in [ A_1, A_2 ) } e^{\hbar' \sigma} g^{A_2} ) \leq C \mathfrak{C}_\flat e^{ - c (\hbar - \hbar') (\delta A_1 + l)^\frac{2}{3 - \gamma} } \,.
	\end{aligned}
\end{equation}
Therefore, the decomposition \eqref{E1-bnd1} and the estimates \eqref{E1-bnd2}-\eqref{E1-bnd3} imply that for any fixed $1 < A_1 < A_2 < \infty$ and $\hbar' \in (0, \hbar)$,
\begin{equation}\label{E-bnd1}
	\begin{aligned}
		\mathscr{E}^\infty ( e^{\hbar' \sigma} ( \tilde{g}^{A_1} - \tilde{g}^{A_2} ) ) \leq & \mathscr{E}^\infty ( \mathbf{1}_{x \in (0, A_1)} e^{\hbar' \sigma} ( g^{A_1} - g^{A_2} ) ) + \mathscr{E}^\infty ( \mathbf{1}_{x \in [ A_1, A_2 ) } e^{\hbar' \sigma} g^{A_2} ) \\
		\leq & C \mathfrak{C}_\flat e^{ - \frac{c}{4} (\hbar - \hbar') (\delta A_1 + l)^\frac{2}{3 - \gamma} } \to 0
	\end{aligned}
\end{equation}
as $A_1 \to + \infty$. Moreover, by \eqref{gA-h-bnd},
\begin{equation}\label{E-bnd2}
	\begin{aligned}
		\mathscr{E}^\infty ( e^{\hbar' \sigma} \tilde{g}^A ) \leq \| \tilde{g}^A \|_{ \mathbb{B}_\infty^\hbar } = \mathscr{E}^\infty ( e^{\hbar \sigma} \tilde{g}^A ) \leq C \mathfrak{C}_\flat
	\end{aligned}
\end{equation}
for any $\hbar' \in (0, \hbar)$.

Denote by $\mathbb{B}_\infty^{\hbar'}$ be a Banach space defined as the same way of $\mathbb{B}_\infty^\hbar$ in \eqref{Bh-1}-\eqref{Bh-2}. It is easy to see that $\mathbb{B}_\infty^\hbar \subseteq \mathbb{B}_\infty^{\hbar'}$, i.e., $\| g \|_{ \mathbb{B}_\infty^{\hbar'} } \leq \| g \|_{ \mathbb{B}_\infty^\hbar }$. Then the estimates \eqref{E-bnd1} and \eqref{E-bnd2} tell us that $\{ \tilde{g}^A \}_{A > 1}$ is a bounded Cauchy sequence in $\mathbb{B}_\infty^{\hbar'}$. As a result, there is a unique $g = g (x,v) \in \mathbb{B}_\infty^{\hbar'}$ such that
\begin{equation}\label{Cnv-g2}
	\begin{aligned}
		\tilde{g}^A (x,v) \to g (x,v) \quad \textrm{strongly in } \mathbb{B}_\infty^{\hbar'}
	\end{aligned}
\end{equation}
as $A \to + \infty$. Combining with the convergences \eqref{Cnv-g1} and \eqref{Cnv-g2}, the uniqueness of the limit shows that $ g (x,v) = g' (x,v) $. Moreover, the bound \eqref{g-prime-bnd} infers that $g (x,v)$ satisfies the bound
\begin{equation}\label{g-bnd}
	\begin{aligned}
		\| g \|_{\mathbb{B}_\infty^\hbar} \leq C \mathfrak{C}_\flat \,.
	\end{aligned}
\end{equation}
Taking limit $A \to + \infty$ in the mild solution form of \eqref{A1}, i.e.,
\begin{equation*}
	\left\{\begin{aligned}
		 v_3 \partial_x \tilde{g}^A + \L \tilde{g}^A  &= h -\bar{\alpha} (\delta x + l)^{-\Theta} \P^+ v_3\tilde{g}^A  -\bar{\beta} (\delta x + l)^{-\Theta} \P^0 \tilde{g}^A \,, \ 0 < x < A \,, \\
		 \tilde{g}^A (0, v) |_{v_3 > 0} &= f_b \,,
	\end{aligned}\right.
\end{equation*}
we easily know that the limit $g (x,v)$ of $\tilde{g}^A (x,v)$ subjects to
\begin{equation}\label{g-equ}
	\left\{\begin{aligned}
		 v_3 \partial_x g + \L g &= h -\bar{\alpha} (\delta x + l)^{-\Theta} \P^+ v_3 g -\bar{\beta} (\delta x + l)^{-\Theta} \P^0 g \,, x > 0 \,, \\
		 g (0, v) |_{v_3 > 0} &= f_b
	\end{aligned}\right.
\end{equation}
in the sense of mild solution. Furthermore, the bounds \eqref{B-2} and \eqref{g-bnd} imply
\begin{equation}\label{g-point-bnd}
	\begin{aligned}
		| g (x, v) | \leq & e^{- \hbar \sigma (x,v)} \sigma_x^{-  \frac{1}{2}} (x, v) w^{-1}_{\beta, \vartheta} (v) \| g \|_{ \mathbb{B}_\infty^\hbar } \\
		\leq & C \mathfrak{C}_\flat e^{- c \hbar (\delta x + l)^\frac{2}{3 - \gamma} } \cdot C (\delta x + l)^{ \Theta } (1 + |v - \u|)^{ 1 - \gamma } w^{-1}_{\beta, \vartheta} (v) \\
		\leq & C' \mathfrak{C}_\flat e^{- \frac{c}{2} \hbar (\delta x + l)^\frac{2}{3 - \gamma} } \to 0
	\end{aligned}
\end{equation}
as $x \to + \infty$. Hence, $ \lim_{x \to + \infty} g (x,v) = 0 $. Therefore, the function $g (x,v)$ solves the problem \eqref{KL-Damped}. Note that we should take the approximate incoming data $\varphi_A (v)$ merely satisfying \eqref{Asmp-phi}, namely, $\sup_{A > 1} \mathscr{B} ( e^{\hbar \sigma} \varphi_A ) < \infty$. Without loss of generality, we can take $ \varphi_A (v) \equiv 0 $, whose corresponding solution $g (x,v)$ subjects to the estimate
\begin{equation}\label{g-Bnd-Final}
	\begin{aligned}
		\| g \|_{\mathbb{B}_\infty^\hbar} \leq C( \mathscr{A}^\infty (e^{\hbar \sigma} h) +\mathscr{C} (f_{b,\sigma}))< \infty \,.
	\end{aligned}
\end{equation}

At the end, we justify the uniqueness. Assume that $g_i (x,v)$ $(i = 1,2)$ are both the solutions to the problem \eqref{KL-Damped} enjoying the bound \eqref{g-Bnd-Final}. Then, for any fixed $A > 1$, the difference $g_1 - g_2$ obeys
\begin{equation*}
	\begin{aligned}
		\left\{
		  \begin{aligned}
		  	& v_3 \partial_x (g_1 - g_2) + \L (g_1 - g_2) = -\bar{\alpha} (\delta x + l)^{-\Theta} \P^+ v_3(g_1 - g_2) -\bar{\beta} (\delta x + l)^{-\Theta} \P^0(g_1 - g_2)\,, \\
			&\qquad\qquad\qquad\qquad\qquad\qquad\qquad\qquad\qquad\qquad\qquad\qquad\qquad\qquad\qquad\qquad \ 0 < x < A \,, \\
		  	& (g_1 - g_2) (0, v) |_{v_3 > 0} = 0 \,, \\
		  	& (g_1 - g_2) (A, v) |_{v_3 < 0} = \phi_A (v) : = \mathbf{1}_{v_3 < 0} (g_1 - g_2) (A, v) \,.
		  \end{aligned}
		\right.
	\end{aligned}
\end{equation*}
Following the similar arguments in Lemma \ref{Lmm-A1}, one knows that
\begin{equation*}
	\begin{aligned}
		\mathscr{E}^A ( e^{ \hbar' \sigma } (g_1 - g_2) ) \leq C \mathscr{B} ( e^{\hbar' \sigma (A, v)} \phi_A (v) ) = C \mathscr{B} ( e^{\hbar' \sigma (A, v)} \mathbf{1}_{v_3 < 0} (g_1 - g_2) (A, v) )
	\end{aligned}
\end{equation*}
for any fixed $\hbar' \in (0, \hbar)$. By employing the same arguments of \eqref{EA1A2-bnd3}, one has
\begin{equation*}
	\begin{aligned}
		\mathscr{B} ( e^{\hbar' \sigma (A, v)} \mathbf{1}_{v_3 < 0} (g_1 - g_2) (A, v) ) \leq & C e^{ - \frac{c}{4} (\hbar - \hbar') (\delta A + l)^\frac{2}{3 - \gamma} } \sum_{i=1,2} \mathscr{E}^A ( e^{\hbar \sigma} g_i ) \\
		\leq & C \mathscr{A}^\infty (e^{\hbar \sigma} h) e^{ - \frac{c}{4} (\hbar - \hbar') (\delta A + l)^\frac{2}{3 - \gamma} } \,,
	\end{aligned}
\end{equation*}
where the last inequality is derived from \eqref{g-Bnd-Final}. As a result,
\begin{equation*}
	\begin{aligned}
		\mathscr{E}^\infty ( e^{ \hbar' \sigma } (g_1 - g_2) ) = \lim_{A \to + \infty} \mathscr{E}^A ( e^{ \hbar' \sigma } (g_1 - g_2) ) \leq \lim_{A \to + \infty} \big[ C \mathscr{A}^\infty (e^{\hbar \sigma} h) e^{ - \frac{c}{4} (\hbar - \hbar') (\delta A + l)^\frac{2}{3 - \gamma} } \big] = 0 \,,
	\end{aligned}
\end{equation*}
which means that $g_1 = g_2$. Therefore, the uniqueness holds. The proof of Lemma \ref{Lmm-KLe} is finished.
\end{proof}

\subsection{Linear problem \eqref{KL}: Proof of Theorem \ref{Thm-Linear}}
In this subsection, based on Lemma \ref{Lmm-KLe}, we will prove Theorem \ref{Thm-Linear} by showing the system \eqref{KL-Damped} becomes the system \eqref{KL} under suitable solvability conditions.
\begin{proof}[Proof of Theorem \ref{Thm-Linear}]
	Recall \eqref{damp-operator} the definitions of the operators $\P^+\,,\P^0$ and $\mathbb{P}$. We  first split \eqref{KL} into
	\begin{equation}\label{}
		\left\{	
		\begin{aligned}
			& v_3 \partial_x (I- \P^0 ) f + \L (I- \P^0 )f = ( I - \mathbb{P} ) S   \,,\\
			& \lim_{x \to + \infty} (I- \P^0 ) f (x,v) = 0 \,,
		\end{aligned}
		\right.
	\end{equation}
and
\begin{equation}\label{}
	\left\{	
	\begin{aligned}
		& v_3  \P^0  f =  \mathbb{P}  S   \,,\\
		& \lim_{x \to + \infty} \P^0 f (x,v) = 0 \,.
	\end{aligned}
	\right.
\end{equation}

 We now consider the following system
 \begin{equation}\label{linear-split}
	\left\{
		\begin{aligned}
			& v_3 \partial_x  f_1 + \L  f_1 =  (I-\mathbb{P}) S -\bar{\alpha}(\delta x + l)^{-\Theta} \P^+ v_3 f_1- \bar{\beta} (\delta x + l)^{-\Theta} \P^0  f_1  \,, \ x > 0 \,, v \in \R^3 \,, \\
			& v_3 \partial_x f_2 =  \mathbb{P} S \,, \ x > 0 \,, v \in \R^3 \,, \\
			&f=f_1+f_2 \,,\\
			& f (0, v) |_{v_3 > 0} = f_b (v) \,, \\
			& \lim_{x \to + \infty} f_1 (x, v) = \lim_{x \to + \infty} f_2 (x, v) = 0 \,,
		\end{aligned}
	\right.
\end{equation}
    We define $\mathbb{I}^{\hbar}$ as the linear solution operator given by
	\begin{equation}
		\mathbb{I}^\hbar({f_b}) := f_1(0,\cdot)\,.
	\end{equation}
	Recalling \eqref{Xj}, we can explicitly write $f_2$  as
\begin{equation}\label{P0-expression}
	 \P^0 f_2 (x,v)= -\frac{1}{v_3}\int_{x}^{\infty} \mathbb{P} S (y , \cdot) \d y =-\sum_{j\in I^0} \psi_j \int_{x}^{\infty} \frac{ (X_j, S(y,\cdot))}{ (X_j, \L X_j)} \d y   \,.
\end{equation}
	Multiplying both sides of $\eqref{linear-split}_1$ by $\psi_j$ with ${j\in I^+}$ and integrating respect to $v$, we obtain
\begin{equation}
	\tfrac{\d}{\d x} (\psi_j, v_3 f_1)= (\delta x + l)^{-\Theta} (\psi_j, v_3 f_1)\,.
\end{equation}
Then if $(\psi_j, v_3 f_1)(0)=0$, we immediately have
\begin{equation}\label{P+}
	(\psi_j, v_3 f_1) \equiv 0\,, \quad \text{for}\quad \! x \geq 0.
\end{equation}
Multiplying $\eqref{linear-split}_1$ by $\psi_j$ with ${j\in I^0}$ and integrating respect to $v$ leads to
 \begin{equation}\label{A_1}
	\tfrac{\d }{\d x} ( \psi_j , v_3 f_1)= -\bar{\beta} (\delta x+ l )^{-\Theta} (X_j, v_3 f_1) \frac{(\psi_j,\psi_j)}{(X_j, \L X_j)} \,,
 \end{equation}
where $\tfrac{(\psi_j,\psi_j)}{(X_j, \L X_j)} >0$.

Notice that the definitions of $\mathbb{P}$ and $X_j$ imply that $(I-\mathbb{P}) S$ satisfies
\begin{equation*}
	\begin{aligned}
		((I-\mathbb{P}) S, X_j)= (S, X_j)-\sum_{i\in I^0} \frac{(X_i,S)}{(X_i, \L X_i)}(\L X_i,X_j)
		= 0\,, \quad \text{for}\quad \! j\in I^0\,,
	\end{aligned}
\end{equation*}
where we used the fact that $(\L X_i, X_j)=0$ for $i\neq j$. Therefore, by multiplying both sides of $\eqref{linear-split}_1$ by $X_j$ with ${j\in I^0}$, one obtains that
\begin{equation*}
	\tfrac{\d }{\d x} ( X_j , v_3 f_1) + (X_j, \L f_1)=0 \,.
\end{equation*}
Note that $X_j= \L^{-1} v_3\psi_j$, where $\L^{-1}$ is the pseudo-inverse of $\L$ defined on $\operatorname{Null}^\perp(\L)$, whose decay properties can be found in \cite{JLT-2022-arXiv}. We obtain that
\begin{equation*}
	\tfrac{\d }{\d x} ( X_j , v_3 f_1) + (\psi_j, v_3 f_1)=0 \,.
\end{equation*}
Together with \eqref{A_1}, we have
\begin{equation*}
	\tfrac{\d^2 }{\d x^2} ( X_j , v_3 f_1)=  -\bar{\beta} (\delta x+ l )^{-\Theta} (X_j, v_3 f_1) \frac{(\psi_j,\psi_j)}{(X_j, \L X_j)}\,, \quad \text{for}\quad \!  j\in I^0 \,.
 \end{equation*}
Using the similar arguments based on the so-called {\em Freezing Point Method} from \cite{JL-2024-arXiv}, we have
 \begin{equation*}
	(X_j, v_3 f_1) \equiv 0\,, \quad \text{for} \quad \! x \geq 0  \,,
 \end{equation*}
if and only if $(X_j, v_3 f_1)(0)=0$. Consequently, once
 \begin{equation}\label{condition}
    \P^+\mathbb{I}^\hbar(f_b)= \P^0 \mathbb{I}^\hbar(f_b)=0  \,,
 \end{equation}
 then the solution to \eqref{linear-split} is a solution to \eqref{KL}, with $f_1 =(I-\P^0) f$ and $f_2 =\P^0 f$.
 Moreover, by Lemma \ref{Lmm-KLe}, it holds that
 \begin{equation}\label{1A}
	\mathscr{E}^\infty(e^{\hbar \sigma} f_1)  \leq C (\mathscr{A}^\infty ( e^{\hbar \sigma} (I-\mathbb{P})S ) + \mathscr{C} ( e^{\hbar \sigma} (f_{b}-f_2)(0,v)  )).
 \end{equation}
Direct calculations show that
 \begin{equation}\label{2A}
	\mathscr{C}(e^{\hbar \sigma} f_2(0,v)) \leq C \mathscr{A}^\infty(e^{\hbar \sigma} S) \,, \quad \textrm{ and} \quad \!
	\mathscr{A}^\infty(e^{\hbar \sigma} \mathbb{P} S) \leq C\mathscr{A}^\infty(e^{\hbar \sigma} S) \,.
 \end{equation}
 Recall \eqref{D-def1} the definition of $\mathscr{D}_\hbar$. It holds from \eqref{P0-expression}, \eqref{1A} and \eqref{2A} that
 \begin{equation*}
	\begin{aligned}
		\mathscr{E}^\infty(e^{\hbar \sigma}f)  &\leq \mathscr{E}^\infty(e^{\hbar \sigma}f_1)+ \mathscr{E}^\infty(e^{\hbar \sigma} f_2) \\
		&\leq C (\mathscr{A}^\infty(e^{\hbar \sigma} S) +\mathscr{D}_\hbar (S) + \mathscr{C} ( e^{\hbar \sigma} f_{b}(0,v)  ) )  .
	\end{aligned}
 \end{equation*}
Therefore, the estimate \eqref{Bnd-f} holds. Moreover, the condition \eqref{condition} defines a codimension $\#\{ I^+\cup I^0 \}$ subset of boundary data $f_b$. The classification of the boundary data satisfying the solvability condition \eqref{nonlinear-condition} with respect to the  Mach number can be summarized in Table \ref{table3}.  The proof of Theorem \ref{Thm-Linear} is completed.
	\begin{table}[h!]
	\begin{center}
	\caption{}
	\begin{tabular}{c|c}\label{table3}
	$\mathcal{M}^\infty <-1$ & $\operatorname{Codim}\left(\left\{f_b \in L_{v_3,+}^{2} \mid \P^+\mathbb{I}^\hbar(f_b)= \P^0 \mathbb{I}^\hbar(f_b)=0\right\}\right)=0$; \\
	$-1 \leq \mathcal{M}^\infty < 0$  & $\operatorname{Codim}\left(\left\{f_b \in L_{v_3,+}^{2} \mid\P^+\mathbb{I}^\hbar(f_b)= \P^0 \mathbb{I}^\hbar(f_b)=0\right\}\right)=1$;\\
	$0 \leq\mathcal{M}^\infty <1 $  & $\operatorname{Codim}\left(\left\{f_b \in L_{v_3,+}^{2} \mid \P^+\mathbb{I}^\hbar(f_b)= \P^0 \mathbb{I}^\hbar(f_b)=0\right\}\right)=4$; \\
	$1 \leq \mathcal{M}^\infty $  & $\operatorname{Codim}\left(\left\{f_b \in L_{v_3,+}^{2} \mid \P^+\mathbb{I}^\hbar(f_b)= \P^0 \mathbb{I}^\hbar(f_b)=0\right\}\right)=5$.
		\end{tabular}
	\end{center}
	\end{table}
\end{proof}

\section{Nonlinear problem \eqref{KL-NL}: Proof of Theorem \ref{Thm-Nonlinear}}\label{Sec:ENL}

In this section, we devote to investigating the nonlinear problem \eqref{KL-NL} near the Maxwellian $\M (v)$ by employing the linear theory constructed in Theorem \ref{Thm-Linear}. It is equivalent to study the equation \eqref{KL-NL-f}. The key point is to dominate the quantity $ \mathscr{A}^\infty ( e^{ \hbar \sigma } \Gamma ( f, g ) ) $, where the functional $ \mathscr{A}^\infty ( \cdot ) $ is given in \eqref{Ah-lambda}.

Notice that $ \mathscr{A}^\infty ( e^{ \hbar \sigma } \Gamma ( f, g ) ) $ is composed of the weighted $L^\infty_{x,v}$-norms and $L^2_{x,v}$-norms of $ e^{ \hbar \sigma } \Gamma ( f, g ) $. Concerning the weights involved in $ \mathscr{A}^\infty ( e^{ \hbar \sigma } \Gamma ( f, g ) ) $, one has
\begin{equation}\label{Gamma-sigma}
	\begin{aligned}
		e^{ \hbar \sigma } \Gamma ( f, g ) = e^{ \hbar \sigma } \Gamma ( e^{ - \hbar \sigma } e^{ \hbar \sigma } f, e^{ - \hbar \sigma } e^{ \hbar \sigma } g ) \,.
	\end{aligned}
\end{equation}
By the properties of $\sigma$ in Lemma \ref{Lmm-sigma}, it holds $ e^{\hbar \sigma} \geq e^{ c (\delta x + l)^\frac{2}{3 - \gamma} } e^{ c |v|^2 } $. As shown in \eqref{Gamma-sigma}, loosely speaking, one of $ e^{ - \hbar \sigma } $ can absorb the factor $e^{ \hbar \sigma }$ out of $\Gamma$, and the other decay factor $ e^{ - \hbar \sigma } \leq e^{ - c (\delta x + l)^\frac{2}{3 - \gamma} } e^{ - c |v|^2 } $ can be used to adjust the required weights.

It is easy to see that $\Gamma (f, g)$ can be pointwise bounded by the $L^\infty_{x,v}$-norms of $f$ and $g$. By the similar arguments in Lemma 3 of \cite{Strain-Guo-2008-ARMA}, the weighted $L^2_{x,v}$-norms of $ e^{ \hbar \sigma } \Gamma ( f, g ) $ can be bounded by the quantity with form $\| w_1 e^{ \hbar \sigma } f \|_{L^2_{x,v}} \| w_2 e^{ \hbar \sigma } g \|_{L^\infty_{x,v}} + \| w_1 e^{ \hbar \sigma } g \|_{L^2_{x,v}} \| w_2 e^{ \hbar \sigma } f \|_{L^\infty_{x,v}} $, where $w_1, w_2$ are some required weights. As a result, we can establish the following lemma.

\begin{lemma}\label{Lmm-Gamma}
	For any fixed functions $f (x,v)$ and $g (x,v)$, there holds
	\begin{equation}\label{Star-1}
		\begin{aligned}
			\mathscr{A}^\infty ( e^{ \hbar \sigma } \Gamma ( f, g ) ) \leq C  \mathscr{E}^\infty ( e^{ \hbar \sigma } f ) \mathscr{E}^\infty ( e^{ \hbar \sigma } g ) \,,
		\end{aligned}
	\end{equation}
	and
	\begin{equation}\label{E-infty-infty}
		\begin{aligned}
			\mathscr{E}^\infty_\infty ( e^{ \hbar \sigma } \Gamma ( f, g ) ) \leq C e^{- \hbar^2 (\delta x + l)^{\frac{2}{3-\gamma}}} \mathscr{E}^\infty_\infty ( e^{ \hbar \sigma } f ) \mathscr{E}^\infty_\infty ( e^{ \hbar \sigma } g ) \,,
		\end{aligned}
	\end{equation}
    where the functionals $ \mathscr{E}^\infty ( \cdot ) $, $\mathscr{E}^\infty_\infty (\cdot) $ and, $ \mathscr{A}^\infty ( \cdot ) $ are defined in \eqref{Eg-lambda}, \eqref{E-infty} and, \eqref{Ah-lambda}.
\end{lemma}

The proof of Lemma \ref{Lmm-Gamma} can be finished by applying the properties of the weight $\sigma (x,v)$ and employing the similar arguments in Lemma 3 of \cite{Strain-Guo-2008-ARMA}. For simplicity, we omit the details here.

Based on Lemma \ref{Lmm-Gamma}, we will study the nonlinear problem \eqref{KL-NL-f} (equivalently \eqref{KL-NL}) by employing the iterative approach, hence, prove Theorem \ref{Thm-Nonlinear}.

\begin{proof}[\bf Proof of Theorem \ref{Thm-Nonlinear}]
Similar to the linear case, we split $f=(I-\P^0)f + \P^0 f$. Then the nonlinear problem \eqref{KL-NL-f} becomes
\begin{equation}\label{splitf}
	\left\{
		\begin{aligned}
			& v_3 \partial_x (I-\P^0) f + \L (I-\P^0)f= (I-\mathbb{P}) \Gamma ( f, f ) + (I-\mathbb{P}) \tilde{h} \,, \ x > 0 \,, v \in \R^3 \,, \\
			& v_3 \partial_x \P^0 f = \mathbb{P} \Gamma ( f, f ) + \mathbb{P} \tilde{h} \,, \ x > 0 \,, v \in \R^3 \,, \\
			&f(0,v) \mid_{v_3>0}=\tilde{f}_b(v) \,,\\
			& \lim_{x \to + \infty} f (x, v) = 0 \,.
		\end{aligned}
	\right.
\end{equation}
To solve \eqref{splitf}, we design the following iteration scheme with an artificial damping term
	\begin{equation}\label{Iter-fi}
		\left\{
		    \begin{aligned}
		    	& v_3 \partial_x  f^{i+1}_1 + \L  f^{i+1}_1 = (I-\mathbb{P})\Gamma ( f^i, f^i ) + (I-\mathbb{P}) \tilde{h} \\
				&\qquad \qquad\qquad \qquad-\bar{\alpha}(\delta x + l)^{-\Theta} \P^+ v_3 f^{i+1}_1- \bar{\beta} (\delta x + l)^{-\Theta} \P^0  f^{i+1}_1  \,, \ x > 0 \,, v \in \R^3 \,, \\
				& v_3 \partial_x f^{i+1}_2 = \mathbb{P} \Gamma ( f^i, f^i ) + \mathbb{P} \tilde{h} \,, \ x > 0 \,, v \in \R^3 \,, \\
				&f^{i}=f^{i}_1+f^i_2 \,,\\
		    	& f^{i+1} (0, v) |_{v_3 > 0} = \tilde{f}_b (v) \,, \\
		    	& \lim_{x \to + \infty} f_1^{i+1} (x, v) = \lim_{x \to + \infty} f_2^{i+1} (x, v) = 0 \,,
		    \end{aligned}
		\right.
	\end{equation}
	which starts from $f^0_1(x,v)=f^0_2(x,v) \equiv 0$. By Lemma \ref{Lmm-KLe}, one has
	\begin{equation}
		\begin{aligned}
			\mathscr{E}^\infty ( e^{ \hbar \sigma }  f_1^{i+1} ) \leq C \mathscr{A}^\infty ( e^{ \hbar \sigma } \Gamma ( f^i, f^i ) ) + C \varsigma
		\end{aligned}
	\end{equation}
	for some constant $ C > 0$, where the quantity $\varsigma$ is defined in \eqref{varsigma}. Together with \eqref{Star-1}, it infers that
	\begin{equation}\label{Iter-fi1}
		\begin{aligned}
			\mathscr{E}^\infty ( e^{ \hbar \sigma } f_1^{i+1} ) \leq C [ \mathscr{E}^\infty ( e^{ \hbar \sigma } f^i ) ]^2 + C \varsigma \,.
		\end{aligned}
	\end{equation}

	Moreover, using \eqref{E-infty-infty} and \eqref{varsigma}, we claim that
	\begin{equation}\label{Iter-fi2}
		\begin{aligned}
			\mathscr{E}^\infty ( e^{ \hbar \sigma }  f_2^{i+1} )
			&\leq C_1 [\mathscr{E}^\infty (e^{\hbar \sigma} f^i)]^2  +  \varsigma
		\end{aligned}
	\end{equation}
    for some $C_1>0$. We only consider the weighted $L^2_{x,v}$-norm $\| \nu^\frac{1}{2} z_{-\alpha} \sigma^{\frac{1}{2}}_x w_{\beta + \beta_\gamma,\vartheta} e^{ \hbar \sigma }  f_2^{i+1} \|_\infty$, and the proof of other cases is easier. In this case,
	\begin{equation*}
		\begin{aligned}
			&\| \nu^\frac{1}{2} z_{-\alpha} \sigma^{\frac{1}{2}}_x w_{\beta + \beta_\gamma,\vartheta} e^{ \hbar \sigma }  f_2^{i+1} \|_\infty \\
			\leq & C\sum_{j\in I^0} \int_{0}^{\infty} \int_{\R^3} \nu^{\frac{1}{2}} z_{-\alpha} \sigma_x^{\frac{1}{2}} w_{\beta+\beta_\gamma,\vartheta} \psi_j  \d v \d x  \\
			&\times \int_{x}^{\infty}\int_{\R^3} e^{\hbar\sigma(x ,v )- \hbar \sigma( y, v_*)} \sigma_y^{-\frac{1}{2}}(y,v_*) w_{\beta,\vartheta}^{-1} X_j(v_*) \mathscr{E}^\infty_\infty(e^{\hbar \sigma} \Gamma (f^i ,f^i )) \d v_* \d y + \varsigma \\
			\leq & C [\mathscr{E}^\infty_\infty(e^{\hbar \sigma} f^i)]^2 \int_{0}^{\infty}\int_{x}^{\infty} (\delta y + l)^{\frac{\Theta}{2}} e^{-\hbar^2 ( \delta y + l)^{\frac{2}{3-\gamma}}} \d y \d x + \varsigma\\
			\leq & C[\mathscr{E}^\infty_\infty(e^{\hbar \sigma} f^i)]^2 + \varsigma \,,
		\end{aligned}
	\end{equation*}
	where we have used \eqref{varsigma}, Lemma \ref{Lmm-sigma} and the fact that for $x\leq y$,
	\begin{equation*}
		e^{\hbar \sigma(x ,v )- \hbar\sigma( y, v_*)}=e^{\hbar(\sigma(x ,v )-\sigma(x, v_*)+ \sigma ( x, v_*) -\sigma( y, v_*))}\leq e^{c\hbar \left| |v-\u|^2- |v_*-\u|^2 \right|}\,.
	\end{equation*}
	Consequently, \eqref{Iter-fi2} holds. One thereby concludes from \eqref{Iter-fi1} and \eqref{Iter-fi2} that
	\begin{equation}\label{Star-2}
		\mathscr{E}^\infty ( e^{ \hbar \sigma } f^{i+1} ) \leq \mathscr{E}^\infty ( e^{ \hbar \sigma } f_1^{i+1} ) +\mathscr{E}^\infty ( e^{ \hbar \sigma } f_2^{i+1} ) \leq C [\mathscr{E}^\infty ( e^{\hbar \sigma } f^i )]^2 + C_0 \varsigma \,.
	\end{equation}
	
	Now we assert that there is a small $\varsigma_0 > 0$ such that if $ \varsigma \leq \varsigma_0 $, then for all $i \geq 1$,
	\begin{equation}\label{Claim-fi}
		\begin{aligned}
			\mathscr{E}^\infty ( e^{ \hbar \sigma } f^{i+1} ) \leq 2 C_0 \varsigma \,.
		\end{aligned}
	\end{equation}
	
	Note that $f^0 \equiv 0$. \eqref{Star-2} shows that $ \mathscr{E}^\infty ( e^{ \hbar \sigma } f^1 ) \leq C_0 \varsigma < 2 C_0 \varsigma $, i.e., the claim \eqref{Claim-fi} holds for $i = 1$. Assume that the claim \eqref{Claim-fi} holds for $1, 2, \cdots, i$. Then the case $i + 1$ can be carried out by \eqref{Star-2} that $ \mathscr{E}^\infty ( e^{ \hbar \sigma } f^{i+1} ) \leq C ( 2 C_0 \varsigma )^2 + C_0 \varsigma = ( 4 C C_0 \varsigma + 1  ) C_0 \varsigma $. Take $\varsigma_0 > 0$ small such that $ 4 C C_0 \varsigma_0 \leq 1 $. Then if $\varsigma \leq \varsigma_0$, one has $ \mathscr{E}^\infty ( e^{ \hbar \sigma } f^{i+1} ) \leq ( 4 C C_0 \varsigma + 1  ) C_0 \varsigma \leq 2 C_0 \varsigma $. Therefore, the Induction Principle concludes the claim \eqref{Claim-fi}, which indicates that $\{ f^i \}_{i \geq 1}$ is bounded in the Banach space $ \mathbb{B}_\hbar^\infty $ defined in \eqref{Bh-1}.
	
	Next we will show that $\{ f^i \}_{i \geq 1}$ is a Cauchy sequence in $ \mathbb{B}_{\hbar}^\infty $. Observe that $f^{i+1} - f^i$ subjects to
	\begin{equation*}
		\left\{
		\begin{aligned}
			& v_3 \partial_x ( f_1^{i+1} - f_1^i ) + \L ( f_1^{i+1} - f_1^i ) = (I-\mathbb{P})\Gamma ( f^i - f^{i-1}, f^i ) + (I-\mathbb{P})\Gamma ( f^{i-1}, f^i - f^{i-1} )\\
			&\quad \quad-\bar{\alpha}(\delta x + l)^{-\Theta} \P^+ v_3 (f_1^{i+1}-f_1^i)-\bar{\beta}(\delta x + l)^{-\Theta} \P^0 (f_1^{i+1}-f_1^i)
			\,, \ x > 0 \,, v \in \R^3 \,, \\
			&v_3 \partial_x ( f_2^{i+1} - f_2^i ) =\mathbb{P}\Gamma ( f^i - f^{i-1}, f^i ) + \mathbb{P} \Gamma ( f^{i-1}, f^i - f^{i-1} ) \,, \ x > 0 \,, v \in \R^3 \,,\\
			& (f^{i+1}-f^i)(0, v) |_{v_3 > 0} = 0 \,, \\
			& \lim_{x \to + \infty} ( f^{i+1} - f^i ) (x, v) = 0 \,.
		\end{aligned}
		\right.
	\end{equation*}
	Lemma \ref{Lmm-KLe} reduces to
	\begin{equation*}
		\begin{aligned}
			\mathscr{E}^\infty ( e^{ \hbar \sigma }( f_1^{i+1} - f_1^i ) ) \leq C  \mathscr{A}^\infty ( e^{ \hbar \sigma } \Gamma ( f^i - f^{i-1}, f^i ) ) + C \mathscr{A}^\infty ( e^{ \hbar \sigma } \Gamma ( f^{i-1}, f^i - f^{i-1} ) ) \,.
		\end{aligned}
	\end{equation*}
	By \eqref{Star-1} and \eqref{Claim-fi}, one obtains that
	\begin{equation}\label{differE1}
		\begin{aligned}
			& \mathscr{A}^\infty ( e^{ \hbar \sigma } \Gamma ( f^i - f^{i-1}, f^i ) ) + \mathscr{A}^\infty ( e^{ \hbar \sigma } \Gamma ( f^{i-1}, f^i - f^{i-1} ) ) \\
			\leq & C \big[ \mathscr{E}^\infty ( e^{ \hbar \sigma } f^i ) + \mathscr{E}^\infty ( e^{ \hbar \sigma } f^{i-1} ) \big] \mathscr{E}^\infty ( e^{ \hbar \sigma } ( f^i - f^{i-1} ) )
			\leq 4 C C_0 \varsigma \mathscr{E}^\infty ( e^{ \hbar \sigma } ( f^i - f^{i-1} ) )\,.
		\end{aligned}
	\end{equation}
	Similarly to obtaining \eqref{Iter-fi2} and using \eqref{Star-1}, one has
	\begin{equation}\label{differE2}
		\begin{aligned}
			\mathscr{E}^\infty (e^{\hbar \sigma}(f_2^{i+1} -f_2^{i}) )
			&\leq C_1 \big[ \mathscr{E}^\infty ( e^{ \hbar \sigma } f^i ) + \mathscr{E}^\infty ( e^{ \hbar \sigma } f^{i-1} ) \big] \mathscr{E}^\infty ( e^{  \hbar \sigma } ( f^i - f^{i-1} ) )  \\
			&\leq  4 C_1 C_0 \varsigma \mathscr{E}^\infty ( e^{ \hbar \sigma } ( f^i - f^{i-1} ) ) \,.
		\end{aligned}
	\end{equation}
    We thereby establish
	\begin{equation}\label{Star-3}
		\begin{aligned}
			\mathscr{E}^\infty ( e^{ \hbar \sigma } ( f^{i+1} - f^i ) )
			\leq \left(4 C_0 (C^2+C_1) \varsigma\right)\mathscr{E}^\infty ( e^{ \hbar \sigma } ( f^i - f^{i-1} ) )
			\leq \tfrac{1}{2} \mathscr{E}^\infty ( e^{ \hbar \sigma } ( f^i - f^{i-1} ) )
		\end{aligned}
	\end{equation}
	by further taking $\varsigma_0 > 0$ suitable small.
	
	Note that $f^0 \equiv 0$. It follows from iterating \eqref{Star-3} and employing \eqref{Claim-fi} that
	\begin{equation}\label{Star-4}
		\begin{aligned}
			\mathscr{E}^\infty ( e^{ \hbar \sigma } ( f^{i+1} - f^i ) ) \leq ( \tfrac{1}{2} )^i \mathscr{E}^\infty ( e^{ \hbar \sigma } f^1 ) \leq 2 C_0 \varsigma ( \tfrac{1}{2} )^i \to 0
		\end{aligned}
	\end{equation}
    as $i \to + \infty$. Consequently, \eqref{Claim-fi} and \eqref{Star-4} tell us that $ \{ f^i \}_{i \geq 1} $ is a bounded Cauchy sequence in $ \mathbb{B}_{\hbar}^\infty $. Then there is a unique $ f (x, v)\in \mathbb{B}_{\hbar}^\infty $ such that
    \begin{equation*}
    	\begin{aligned}
    		f^i (x, v) \to f (x,v) \quad \textrm{strongly in } \mathbb{B}_{\hbar}^\infty \,,
    	\end{aligned}
    \end{equation*}
	with the estimate $ \mathscr{E}^\infty ( e^{ \hbar \sigma } f ) \leq 2 C_0\varsigma $. Furthermore, $\{ f_1^i \}_{i \geq 1}$ and $\{ f_2^i \}_{i \geq 1}$ are also Cauchy sequences in $ \mathbb{B}_{\hbar}^\infty $. There exist unique limits $f_1$ and $f_2$ such that
	\begin{equation*}
		f_j^i (x, v) \to f_j (x,v) \quad \textrm{strongly in } \mathbb{B}_{\hbar}^\infty \,
	\end{equation*}
    for $j=1,2$, with $f=f_1+f_2$.

	Finally, it remains to show that $f(x,v)$ is a solution of the nonlinear problem \eqref{KL-NL-f} under suitable solvability conditions.
	We define the operator
	\begin{equation*}
		\mathcal{I}^\gamma (\tilde{f}_b) \equiv f_1 (0,\cdot)\,.
	\end{equation*}
	Similar to the arguments in the linear case,  $f (x,v)$ solves the nonlinear problem \eqref{KL-NL-f} if and only if
	\begin{equation}\label{nonlinear-condition}
		\P^+v_3 \mathcal{I}^\gamma (\tilde{f}_b)=\P^0  \mathcal{I}^\gamma (\tilde{f}_b)=0\,.
	\end{equation}
	This provides $\# \{ I^+\cup I^0 \}$ solvability conditions for the source term $\tilde{h}$ and the boundary source term $\tilde{f}_b$.
	The classification of the boundary data near the far field, which satisfies the solvability condition \eqref{nonlinear-condition} with respect to the  Mach number, is summarized in Table \ref{table4}. The proof of Theorem \ref{Thm-Nonlinear} is therefore finished.
	
	\begin{table}[h!]
		\begin{center}
		\caption{}
		\begin{tabular}{c|c}\label{table4}
		$\mathcal{M}^\infty <-1$ & $\operatorname{Codim}\left(\left\{\tilde{f}_b \in L_{v_3,+}^{2} \mid \P^+\mathcal{I}^\hbar(\tilde{f}_b)= \P^0 \mathcal{I}^\hbar(\tilde{f}_b)=0\right\}\right)=0$; \\
		$-1 \leq \mathcal{M}^\infty < 0$  & $\operatorname{Codim}\left(\left\{\tilde{f}_b \in L_{v_3,+}^{2} \mid\P^+\mathcal{I}^\hbar(\tilde{f}_b)= \P^0 \mathcal{I}^\hbar(\tilde{f}_b)=0\right\}\right)=1$;\\
		$0 \leq\mathcal{M}^\infty <1 $  & $\operatorname{Codim}\left(\left\{\tilde{f}_b \in L_{v_3,+}^{2} \mid \P^+\mathcal{I}^\hbar(\tilde{f}_b)= \P^0 \mathcal{I}^\hbar(\tilde{f}_b)=0\right\}\right)=4$; \\
		$1 \leq \mathcal{M}^\infty $  & $\operatorname{Codim}\left(\left\{\tilde{f}_b \in L_{v_3,+}^{2} \mid \P^+\mathcal{I}^\hbar(\tilde{f}_b)= \P^0 \mathcal{I}^\hbar(\tilde{f}_b)=0\right\}\right)=5$.
			\end{tabular}
		\end{center}
		\end{table}
\end{proof}

\section{Bounds for operators $Y_A, Z$ and $U$: Proof of Lemma \ref{Lmm-ARU}}\label{Sec:YZU-proof}

In this section, we mainly aim at verifying the proof of Lemma \ref{Lmm-ARU}.

\begin{proof}[Proof of Lemma \ref{Lmm-ARU}]
	The proof of this lemma will be divided into three steps as follows.
	
	{\bf Step 1. $L^\infty_{x,v}$ bound for the operators $Y_A$ and $Z$.}
	
	We now show the inequality \eqref{YAn-bnd}.
	
	If $v_3 > 0$, Lemma \ref{Lmm-sigma} indicates that $0 < v_3 \sigma_x \leq c \nu (v)$ and $|\sigma_{xx} v_3| \leq \delta \sigma_x \nu (v)$, which mean that
	\begin{equation}\label{k-bnd-1}
		\begin{aligned}
			\kappa (x,v) = & \int_0^x \big[ - \tfrac{ \sigma_{xx} (y,v) }{2 \sigma_x (y,v)} - \hbar \sigma_x (y, v) + \tfrac{\nu (v)}{v_3} \big] \d y \\
			\geq & \int_0^x  (1 - c \hbar - \tfrac{1}{2} \delta ) \tfrac{\nu (v)}{v_3} \d y = (1 - c \hbar - \tfrac{1}{2} \delta) \tfrac{\nu (v)}{v_3} x \geq 0 \,,
		\end{aligned}
	\end{equation}
	provided that $c_{\hbar, \delta} : = 1 - c \hbar - \tfrac{1}{2} \delta > 0$. Actually, one can take sufficiently small $\hbar, \delta > 0$ such that $c_{\hbar, \delta} > \frac{1}{2}$. Then one has
	\begin{equation}\label{k-bnd1-prime}
		\begin{aligned}
			e^{- \kappa (x,v)} \leq 1
		\end{aligned}
	\end{equation}
	for $1 - c \hbar - \tfrac{1}{2} \delta > 0$.

	Then
	\begin{equation}\label{Z-1}
		 \| \mathbf{1}_{v_3 >0} e^{-k(x,v)} f(0,v) \|_{L^\infty_v} \leq C \|f(0,\cdot)\|_{L^\infty_v} \,.
	\end{equation}

	Moreover, if $ v_3 < 0$, one similarly knows that
	\begin{equation}\label{k-Ax}
		\begin{aligned}
			\kappa (A, v) - \kappa (x,v) = \int_x^A [ - \tfrac{\sigma_{xx} (y,v)}{2 \sigma_x (y,v)} - \hbar \sigma_x (y,v) - \tfrac{\nu (v)}{|v_3|}] \d y \\
			\leq - (1 - c \hbar - \tfrac{1}{2} \delta ) \tfrac{\nu (v)}{|v_3|} (A - x) \leq 0 \,,
		\end{aligned}
	\end{equation}
	which means that
	\begin{equation}\label{k-bnd2-prime}
		\begin{aligned}
			e^{\kappa (A, v) - \kappa (x,v)} \leq e^{ - (1 - c \hbar -  \frac{1}{2} \delta) \tfrac{\nu (v)}{|v_3|} (A - x) } \leq 1 \,.
		\end{aligned}
	\end{equation}	
	Consequently, there holds
	\begin{equation}\label{Y-3}
		\begin{aligned}
			\| \mathbf{1}_{v_3 < 0} e^{\kappa (A, v) - \kappa (x,v)} f (A, v) \|_{L^\infty_{x,v}} \leq \| f (A, \cdot ) \|_{L^\infty_v} \,.
   		\end{aligned}
	\end{equation}
	Recalling the definition of the operators  $Z(f)$ and $Y_A (f)$ in \eqref{Rn-f} and \eqref{YAn-f}, the estimates \eqref{Z-1} and \eqref{Y-3} infer that
	\begin{equation*}
		\begin{aligned}
			\| Z (f) \|_{L^\infty_{v}} \leq C \| f (0, \cdot) \|_{L^\infty_v}\,,\\
			\| Y_A (f) \|_{L^\infty_{v}} \leq C \| f (A, \cdot) \|_{L^\infty_v}
		\end{aligned}
	\end{equation*}
	for some constant $C > 0$ independent of $A$, $\hbar$. Namely the bounds \eqref{YAn-bnd} and \eqref{Z-bnd} holds.

	{\bf Step 2. $L^\infty_{x,v}$ bound for the operator $U$.}
	
	We now justify the bound \eqref{U-bnd}. By the definition of $\kappa (x,v)$ in \eqref{kappa}, one has
	\begin{equation*}
		\begin{aligned}
			\mathbf{1}_{x' \in (0, x)} \mathbf{1}_{v_3 > 0} e^{ - [ \kappa (x,v) - \kappa (x', v) ] } = e^{ - \int_{x'}^x [ - \frac{\sigma_{xx} (y, v)}{2 \sigma_x (y,v)} - \hbar \sigma_x (y, v) + \frac{\nu (v)}{v_3} ] \d y } \,.
		\end{aligned}
	\end{equation*}
	By the similar arguments in \eqref{k-bnd-1}, one has
	\begin{equation*}
		\begin{aligned}
			 \mathbf{1}_{x' \in (0, x)} \mathbf{1}_{v_3 > 0} e^{ - [ \kappa (x,v) - \kappa (x', v) ] }
			& \leq \mathbf{1}_{x' \in (0, x)} \mathbf{1}_{v_3 > 0}  e^{ - c_{\hbar, \delta} \frac{\nu (v)}{v_3} (x - x') } \\
			& \leq C e^{ - c_{\hbar, \delta} \frac{\nu (v)}{v_3} (x - x') } \,.
		\end{aligned}
	\end{equation*}
	Together with the definition of the operator $U$ in \eqref{U-f}, it infers that
		\begin{align}\label{U-1}
			\no | \mathbf{1}_{v_3 > 0} U (f) | \leq & C \int_0^x \tfrac{\nu (v)}{v_3} e^{ -  c_{\hbar, \delta} \frac{\nu (v)}{v_3} (x - x') } | \nu^{-1} f (x' , v) | \d x' \\
			\no \leq & C \| \nu^{-1} f \|_{L^\infty_{x,v}} \int_0^x \tfrac{\nu (v)}{v_3} e^{ -  c_{\hbar, \delta} \frac{\nu (v)}{v_3} (x - x') } \d x' \\
			= & \tfrac{ C}{c_{\hbar, \delta}} \| \nu^{-1} f \|_{L^\infty_{x,v}} ( 1 - e^{ -  c_{\hbar, \delta} \frac{\nu (v)}{v_3} x } ) \leq \tfrac{ C}{c_{\hbar, \delta}} \| \nu^{-1} f \|_{L^\infty_{x,v}} \,.
		\end{align}
	
	We then consider the case $v_3 < 0$. Combining with the definition of $\kappa (x,v)$ in \eqref{kappa} and the inequality \eqref{k-Ax}, one derives
	\begin{equation*}
		\begin{aligned}
			\mathbf{1}_{x' \in (x,A)} \mathbf{1}_{v_3 < 0} e^{ - [ \kappa (x,v) - \kappa (x', v) ] } = & \mathbf{1}_{x' \in (x,A)} \mathbf{1}_{v_3 < 0} e^{ \int^{x'}_x [ - \frac{\sigma_{xx} (y, v)}{2 \sigma_x (y,v)} - \hbar \sigma_x (y, v) - \frac{\nu (v)}{|v_3|} ] \d y } \\
			\leq & \mathbf{1}_{x' \in (x,A)} \mathbf{1}_{v_3 < 0} e^{ - c_{\hbar, \delta} \frac{\nu (v)}{|v_3|} (x' - x) } \,.
		\end{aligned}
	\end{equation*}
	Then
	\begin{equation}\label{U-2}
		\begin{aligned}
			| \mathbf{1}_{v_3 < 0} U (f) | \leq & \int_x^A e^{ - c_{\hbar, \delta} \frac{\nu (v)}{|v_3|} (x' - x) } \tfrac{\nu (v)}{|v_3|} | \nu^{-1} (v) f (x', v) | \d x' \\
			\leq & \| \nu^{-1} f \|_{L^\infty_{x,v}} \int_x^A e^{ - c_{\hbar, \delta} \frac{\nu (v)}{|v_3|} (x' - x) } \tfrac{\nu (v)}{|v_3|} \d x' \\
			= & \tfrac{1}{c_{\hbar, \delta}} \| \nu^{-1} f \|_{L^\infty_{x,v}} ( 1 - e^{ - c_{\hbar, \delta} \frac{\nu (v)}{|v_3|} (A - x) } ) \leq \tfrac{1}{c_{\hbar, \delta}} \| \nu^{-1} f \|_{L^\infty_{x,v}} \,.
		\end{aligned}
	\end{equation}
	Then the bounds \eqref{U-1} and \eqref{U-2} conclude the estimate \eqref{U-bnd}. The proof of Lemma \ref{Lmm-ARU} is therefore completed.
\end{proof}

\section*{Acknowledgments}
This work was supported by National Key R\&D Program of China under the grant 2023Y-FA1010300. The author N. Jiang is supported by the grants from the National Natural Foundation of China under contract Nos. 11971360, 12371224 and 12221001. The author Y.-L. Luo is supported by grants from the National Natural Science Foundation of China under contract No. 12201220, the Guang Dong Basic and Applied Basic Research Foundation under contract No. 2021A1515110210 and 2024A1515012358. The author T. Yang is supported by a fellowship award from the Research Grants Council of the Hong Kong Special Administrative Region, China (Proiect no. SRFS2021-1S01). T. Yang would also like to thank the Kuok foundation for its generous support.

%%%%%%%%%%%%%%%%%%%%%%%%%%%%%%%%%%%%%%%%%%%%%%%%%%%%%%%%%%%%%%%%%%%%%%%%%%%
%\bigskip
% \phantomsection
% \addcontentsline{toc}{section}{\refname}

%\bibliographystyle{unsrtnat}
%\nocite{*}
\bibliography{reference}

\end{document}